\documentclass[11pt]{amsart}
\usepackage[utf8]{inputenc}

\usepackage{enumitem,amsmath,amssymb,amsthm,amsfonts,todonotes,mathtools,float}
\usepackage{tikz}
\usepackage{ytableau}
\usepackage{tikz-cd}
\usetikzlibrary{decorations.pathreplacing}

\newcommand{\C}{\mathbb{C}}
\newcommand{\TT}{\mathcal{T}}
\newcommand{\XX}{\mathcal{X}}
\newcommand{\OO}{\mathcal{O}}
\newcommand{\DD}{\mathcal{D}}
\newcommand{\LL}{\mathcal{L}}
\newcommand{\PP}{\mathcal{P}}
\newcommand{\SG}{\mathfrak{S}}
\newcommand{\bfa}{{\boldsymbol{a}}}
\newcommand{\bfb}{{\boldsymbol{b}}}
\newcommand{\bfc}{{\boldsymbol{c}}}
\newcommand{\bfd}{{\boldsymbol{d}}}
\newcommand{\bft}{{\boldsymbol{t}}}
\newcommand{\bfm}{{\boldsymbol{m}}}
\newcommand{\bfi}{{\boldsymbol{i}}}
\newcommand{\MP}{{\rm M\!P\!}}
\newcommand{\MPT}{\mathcal{MPT}}
\newcommand{\SPT}{\mathcal{SPT}}
\newcommand{\SSPT}{\mathcal{SSPT}}
\newcommand{\SSMPT}{\mathcal{SSMPT}}
\newcommand{\ttau}{{\tilde{\tau}}}
\newcommand{\tpi}{{\tilde{\pi}}}
\newcommand{\tgamma}{{\tilde{\gamma}}}
\newcommand{\tmu}{{\tilde{\mu}}}
\newcommand{\tnu}{{\tilde{\nu}}}
\newcommand{\trho}{{\tilde{\rho}}}
\newcommand{\tB}{{\tilde{B}}}
\newcommand{\tS}{{\tilde{S}}}
\newcommand{\tR}{{\tilde{R}}}
\newcommand{\tT}{{\tilde{T}}}
\newcommand{\tU}{{\tilde{U}}}

\newcommand{\multi}[1]{\left\{\mskip-5mu\left\{#1\right\}\mskip-5mu\right\}}
\newcommand{\ov}[1]{\bar{#1}}
\newcommand{\ovb}[1]{\left[\ov{#1}\right]}
\newcommand{\ovov}[1]{\bar{\bar{#1}}}
\newcommand{\ovovb}[1]{\left[\ovov{#1}\right]}
\newcommand{\abs}[1]{\left| #1\right |}
\newcommand{\inline}[1]{\begin{array}{c}#1\end{array}}

\definecolor{defncolor}{HTML}{C44800}
\newcommand{\defn}[1]{\textcolor{defncolor}{\emph{#1}}}

\DeclareMathOperator{\nbw}{nbw}
\DeclareMathOperator{\vb}{vb}
\DeclareMathOperator{\Span}{span}
\DeclareMathOperator{\End}{End}
\DeclareMathOperator{\Hom}{Hom}
\DeclareMathOperator{\Break}{break}

\definecolor{c1}{HTML}{39A0FF}
\definecolor{c2}{HTML}{DC7D00}
\definecolor{c3}{HTML}{90D300}
\definecolor{c4}{HTML}{9728A1}
\newcommand{\colortop}[2]{\foreach \i in {#1}  {\filldraw[fill=#2,draw=#2,line width = 1pt] (T\i) circle (4pt);}}
\newcommand{\colormid}[2]{\foreach \i in {#1}  {\filldraw[fill=#2,draw=#2,line width = 1pt] (M\i) circle (4pt);}}
\newcommand{\colorbot}[2]{\foreach \i in {#1}  {\filldraw[fill=#2,draw=#2,line width = 1pt] (B\i) circle (4pt);}}
\newcommand{\colortopfixed}[2]{\foreach \i in {#1}  {\filldraw[fill=white,draw=#2,line width = 1pt] (T\i) circle (4pt);}}
\newcommand{\colorbotfixed}[2]{\foreach \i in {#1}  {\filldraw[fill=white,draw=#2,line width = 1pt] (B\i) circle (4pt);}}

\newcommand{\inlinediagram }[2]{\raisebox{-.2em}{$\begin{array}{c}
    \begin{tikzpicture}[xscale=.5,yscale=.5,line width=1.25pt] 
        \foreach \i in {1,...,#1}  { \path (\i,1.25) coordinate (T\i); \path (\i,.25) coordinate (B\i);  } 
        \filldraw[fill= black!12,draw=black!12,line width=4pt]  (T1) -- (T#1) -- (B#1) -- (B1) -- (T1);
        #2
    \end{tikzpicture}
\end{array}$}}

\usepackage{hyperref}

\hypersetup{pdftitle={A Diagram-Like Basis for the Multiset Partition Algebra},pdfauthor={Alexander Wilson}} 
\hypersetup{colorlinks=true,linkcolor=blue,anchorcolor=blue,citecolor=blue}
\usepackage[nameinlink]{cleveref}

\newtheorem{prop}{Proposition}[section]
\newtheorem{theorem}[prop]{Theorem}
\newtheorem{lemma}[prop]{Lemma}

\newtheorem{defi}[prop]{Definition}
\newtheorem{rema}[prop]{Remark}
\newtheorem{exam}[prop]{Example}

\title{A Diagram-Like Basis for the Multiset Partition Algebra}
\author{Alexander Wilson}

\begin{document}

\maketitle

\textbf{Abstract} There is a classical connection between the representation theory of the symmetric group and the general linear group called Schur-Weyl duality. Variations on this principle yield analogous connections between the symmetric group and other objects such as the partition algebra and more recently the multiset partition algebra. The partition algebra has a well-known basis indexed by graph-theoretic diagrams which allows the multiplication in the algebra to be understood visually as combinations of these diagrams. We construct an analogous basis for the multiset partition algebra called the diagram-like basis and use this basis to construct its irreducible representations and give a generating set. We also provide a change-of-basis formula from the orbit basis of the multiset partition algebra to this diagram-like basis which exhibits similarities to the analogous change-of-basis formula for the partition algebra.

\section{Introduction}

Let $V_n$ be an $n$-dimensional vector space, and write ${V_n}^{\otimes r}$ for its $r$th tensor power. The general linear group $GL_n$ acts on the tensor power diagonally, where a matrix $M\in GL_n$ acts on each tensor factor: \[M.(v_1\otimes\dots\otimes v_r)=(Mv_1)\otimes\dots\otimes (Mv_r).\] The symmetric group $\SG_r$ acts on the tensor power by permuting tensor factors: \[\sigma\cdot(v_1\otimes\dots\otimes v_r)=v_{\sigma^{-1}(1)}\otimes\dots\otimes v_{\sigma^{-1}(r)}.\] These two actions are mutual centralizers. That is, taking the endomorphisms of ${V_n}^{\otimes r}$ that commute with one action recovers the other. This is an example of Schur-Weyl duality, and one consequence of such a duality is the decomposition \[{V_n}^{\otimes r}\cong\bigoplus_{\lambda}W_{GL_n}^\lambda\otimes W_{\SG_r}^\lambda\] where for a group or algebra $A$, we write $W_A^\lambda$ to represent an irreducible representation of $A$ and the sum is over partitions $\lambda$ of $r$. This establishes a correspondence between irreducible $GL_n$-modules and irreducible $\SG_r$-modules and allows information to be passed between the representation theory of the two groups.

Any situation in which two actions mutually centralize each other leads to an analogous decomposition. The study of these centralizer algebras is connected to long-standing questions in the representation theory of the symmetric group and $GL_n$ including the Kronecker problem \cite{bowman2015partition,orellana2020howe}, the restriction problem, and plethysm \cite{orellana2021uniform}.

For another example, let $V_{n,k}=(\C^n)^k$ be the space of $n\times k$ matrices and let $\PP^r(V_{n,k})$ be the space of polynomial forms on $V_{n,k}$. The group $GL_n$ naturally acts on $V_{n,k}$ by left-multiplication. This action extends to $\PP^r(V_{n,k})$ where $M\in GL_n$ acts on $f(X)\in\PP^r(V_{n,k})$ by \[(M.f)(X)=f(M^{-1}X).\] The group $GL_k$ can act analogously where $N\in GL_k$ acts by: \[(N.f)(X)=f(XN).\] In \cite{howe}, Roger Howe determined that these actions are mutual centralizers leading to a decomposition \[\PP^r(V_{n,k})\cong\bigoplus_{\lambda} W_{GL_n}^\lambda\otimes W_{GL_k}^\lambda\] where the sum is over partitions $\lambda$ of $r$ of length at most $\min\{n,k\}$ (see e.g. Section 5.2.6 of \cite{gw} for details).

In the 1990's, Martin \cite{martin} introduced the partition algebra $P_r(n)$ as a generalization of the Temperley-Lieb algebra and the Potts model in statistical mechanics. In the case that $n\geq 2r$, the partition algebra is the algebra whose action centralizers the action of $\SG_n\subseteq GL_n$ as the subgroup of permutation matrices acting on ${V_n}^{\otimes r}$ \cite{jones1993potts,martin1991potts,Martin1994AlgebrasIH}. From this perspective, there is a natural basis for $P_r(n)$ called the orbit basis, but its product is complicated. A second basis called the diagram basis has a much simpler product in terms of graph-theoretic diagrams. The partition algebra, its generators, and its representations have been well-studied, and some major milestones are outlined in the following timeline.

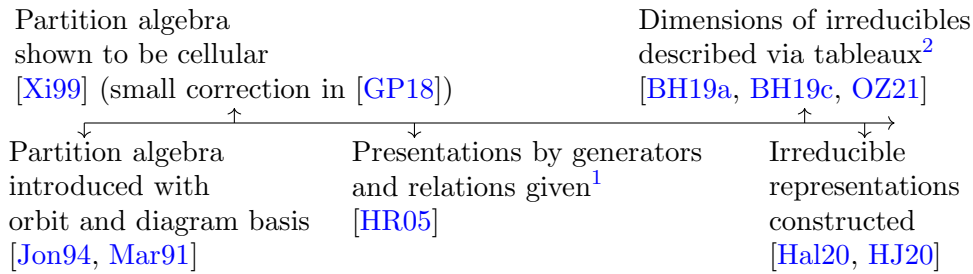
\newlength\yearposx
\hspace{-.7cm}\begin{tikzpicture}[scale=0.57]
    \foreach \x in {1994,1999, 2005,2018,2020,2021}{
        \pgfmathsetlength\yearposx{(\x-1990)*.7cm};
        \coordinate (y\x)   at (\yearposx,0);
        \coordinate (y\x t) at (\yearposx,+10pt);
        \coordinate (y\x b) at (\yearposx,-10pt);
    }
    \draw [->] (y1994) -- (y2021);
   \foreach \x in {1994,2005,2020}
        \draw[->] (y\x) -- (y\x b);
   \foreach \x in {1999,2018}
        \draw[->] (y\x) -- (y\x t);
    \node at (y1994) [below=3pt,align=left,xshift=1cm] {Partition algebra\\ introduced with\\ orbit and diagram basis \\\cite{jones1993potts,martin1991potts}};
    \node at (y1999) [above=3pt, align=left] {Partition algebra\\ shown to be cellular \\ \cite{xi1999partition} (small correction in \cite{green2018iterated})};
    \node at (y2005) [below=3pt,align=left, xshift=1.5cm] {Presentations by generators\\and relations given\footnotemark\\ \cite{HALVERSON2005869}};
    \node at (y2018) [above=3pt,align=left] {Dimensions of irreducibles\\described via tableaux\footnotemark\\ \cite{Benkart2019PartitionAA,Benkart2019PartitionAP,Orellana2015SymmetricGC}};
    \node at (y2020) [below=3pt,align=left] {Irreducible\\representations\\constructed\\ \cite{Halverson2019SetPartitionTS,Halverson2018SetpartitionTA}};
\end{tikzpicture}

\footnotetext[1]{Other notable presentations for $P_r(n)$ are given in \cite{james2011generators, enyang2013jucys}}
\footnotetext{Note that once the cellularity and semisimplicity of $P_r(n)$ were established, the dimensions of the irreducibles are encoded in the cellular data. However, these papers were the first to describe these dimensions explicitly in terms of tableaux}

A major motivation for studying the partition algebra is to understand its representations and use them to study objects in the representation theory of the symmetric group such as the Kronecker coefficients \cite{Benkart2019PartitionAA, Benkart2019PartitionAP, benkart2017dimensions,bowman2021the,enyang2013seminormal,halverson2001characters,halverson2004RSK,Halverson2019SetPartitionTS}. 

In \cite{orellana2020howe}, the authors Orellana and Zabrocki restrict the $GL_n$ action of Howe duality to the $n\times n$ permutation matrices to obtain the multiset partition algebra $\MP_{r,k}(n)$ and provide a basis analogous to the orbit basis for $P_r(n)$. In \cite{orellana2020combinatorial,orellana2020howe} Orellana and Zabrocki use symmetric function methods to compute the dimensions of irreducible $\MP_{r,k}(n)$ as counting the number of semistandard multiset partition tableaux. Our aim in this paper is to fill in three gaps in the timeline for $\MP_{r,k}(n)$ by:
\begin{enumerate}
    \item[(a)] providing a basis analogous to the diagram basis for $P_r(n)$,
    \item[(b)] providing generators for $\MP_{r,k}(n)$, and
    \item[(c)] constructing the irreducible representations of $\MP_{r,k}(n)$ as actions on the tableaux enumeratively predicted by Orellana and Zabrocki.
\end{enumerate}

In \cite{narayanan}, the authors investigate the centralizer $\End_{\SG_n}(\PP^{\bfa_1}(V_n)\otimes\dots\otimes \PP^{\bfa_k}(V_n))$ for a weak composition $\bfa$ of $r$ of length $k$, also dubbing it the multiset partition algebra, written $\MP_\bfa(n)$. Orellana and Zabrocki state that $\MP_\bfa(n)$ should be isomorphic to a subalgebra of their multiset partition algebra, and in this paper we describe that isomorphic subalgebra.

In Section \ref{sec:prelim}, we define the relevant combinatorial objects and introduce the partition algebra and multiset partition algebra. In Section \ref{sec:orbits}, we collect some useful facts about orbits of the combinatorial objects under the action of Young subgroups. In Section \ref{sec:construction}, we describe the centralizer of an algebra $A$ acting on a direct sum of projections of a semisimple module $V$ by idempotents. We then construct the irreducible modules over this centralizer in terms of irreducible modules of $\End_A(V)$. In the following two sections, we specialize these constructions to the setting of the multiset partition algebra. In Section \ref{sec:diagram_like_basis}, we use a decomposition of $\PP^r(V_{n,k})$ as a $GL_n$-module to obtain the diagram-like basis and describe $\End_G(\PP^r(V_{n,k}))$ for general subgroups $G$ of $GL_n$. In Section \ref{sec:reps}, we construct the irreducible representations of $\MP_{r,k}(n)$ with bases indexed by multiset-valued tableaux. In Section \ref{sec:generators}, we provide a generating set for $\MP_{r,k}(n)$, and finally in Section \ref{app:changeOfBasis} we provide a change-of-basis formula from the orbit basis of Orellana and Zabrocki to our diagram-like basis.

\section{Preliminaries and Definitions}\label{sec:prelim}

\subsection{Set and Multiset Partitions}

A \defn{set partition} $\rho$ of a set $S$ is a set of nonempty subsets of $S$ called \defn{blocks} whose disjoint union is $S$. We write $\ell(\rho)$ for the number of blocks in $\rho$. We define $[r]=\{1,\dots,r\}$ the \defn{unbarred} numbers and $\ovb{r}=\{\ov1,\dots,\ov r\}$ the \defn{barred} numbers. We write $\Pi_{2(r)}$ for the set of set partitions of $[r]\cup\ovb{r}$. For a set $R$ and a set partition $\rho$ of $S$, write $\rho|_R$ for the set partition obtained by restricting the elements in $\rho$ to only the elements appearing in $R$.

\begin{exam} A set partition and its restriction to a set:
    \begin{align*}
        \pi&=\{\{1,3\},\{2,4,\ov1,\ov2\},\{\ov3,\ov4\}\}\in\Pi_{2(4)}\\
        \pi|_{[4]}&=\{\{1,3\},\{2,4\}\}
    \end{align*}
\end{exam}

A \defn{weak composition} of an integer $r$ of length $k$ is a sequence of $k$ non-negative integers which sum to $r$. Write $W_{r,k}$ for the set of weak compositions of $r$ of length $k$. For $\bfa\in W_{r,k}$, write $\bfa_i$ for the $i$th number in the sequence.

A \defn{multiset} of size $r$ from a set $S$ is a finite collection of $r$ unordered elements of $S$ which can be repeated. We will write multisets in $\multi{,}$ to differentiate them from sets and we will usually denote them by a capital letter with a tilde. We may write multisets using exponential notation $\tS=\multi{{s_1}^{m_1},\dots,{s_k}^{m_k}}$ where the multiplicity of the element $s_i$ is given by the exponent $m_i$. We write $m_{s_i}(\tS)=m_i$ for this multiplicity. Given multisets $\tS=\multi{{s_1}^{m_1},\dots,{s_k}^{m_k}}$ and $\tR=\multi{{s_1}^{n_1},\dots,{s_k}^{n_k}}$, write $\tS\uplus\tR=\multi{{s_1}^{m_1+n_1},\dots,{s_k}^{m_k+n_k}}$ for the union of the two multisets.

A \defn{multiset partition} $\tilde\rho$ of a multiset $\tS$ is a multiset of multisets called blocks whose union is $\tS$. We write $\ell(\tilde\rho)$ for the number of blocks. Write $\tilde\Pi_{2(r),k}$ for the set of multiset partitions with $r$ elements from $[k]$ and $r$ elements from $\ovb{k}$. For $\trho$ a multiset partition and $R$ a set, write $\trho|_{R}$ for the multiset partition obtained by restricting the elements in $\trho$ to only the elements appearing in $R$. 

\begin{exam} A multiset partition and its restriction to a set:
    \begin{align*}
        \tpi&=\multi{\multi{1,\ov1,\ov1},\multi{1,2,\ov2}}\in\tilde\Pi_{2(3),2}\\
        \tpi|_{\ovb2}&=\multi{\multi{\ov1,\ov1},\multi{\ov2}}
    \end{align*}
\end{exam}

Finally, we define partial orders on multisets and multiset partitions, which restrict to sets and set partitions as a special case. Let $\tS$ and $\tR$ be multisets from $[k]$. We say that $\tS<\tR$ if one of the following conditions hold: (i) $\tS$ is empty and $\tR$ is not, (ii) $\max(\tS)<\max(\tR)$, or (iii) $\max(\tS)=\max(\tR)=m$ and $\tS\smallsetminus\{m\}<\tR\smallsetminus\{m\}$. We call this the \defn{last-letter order} on multisets.

\begin{exam} Comparisons of multisets:
    \begin{align*}
        \multi{1,3}&\leq\multi{1,1,3}\\
        \multi{1,1,1,2}&\leq\multi{3}
    \end{align*}
\end{exam}

If $\tnu=\{\tS_1,\dots,\tS_k\}$ and $\tpi=\{\tR_1,\dots,\tR_\ell\}$ are multiset partitions, we say that $\tnu\leq\tpi$, or $\tnu$ is \defn{coarser} than $\tpi$, if $\tnu$ can be obtained by combining blocks of $\tpi$. Precisely, there exists a set partition $\{C_1,\dots,C_k\}$ of $[\ell]$ so that $\tilde{S}_i=\biguplus_{j\in C_i}\tilde{R}_j$ for all $i$.

\begin{exam} An example of comparing two multiset partitions:
    \begin{align*}
        \multi{\multi{1,1,3,3},\multi{2,3}}&\leq\multi{\multi{1,3},\multi{1,3},\multi{2},\multi{3}}
    \end{align*}
\end{exam}

\subsection[Diagrams]{Set and Multiset Partition Diagrams}

For a set partition $\pi\in\Pi_{2(r)}$, there is a classical graph-theoretic representation of $\pi$ on two rows of vertices with the top row being labeled $1$ through $r$ and the bottom being labeled $\ov1$ through $\ov{r}$. Two vertices of this graph are in the same connected component if and only if their labels are in the same block of $\pi$. 

\begin{exam} The set partition $\pi=\{\{1,\ov1,\ov2,\ov3\},\{2,3\},\{\ov4\},\{4,5,\ov5\}\}$ could be represented by either of the following two graphs.
    \begin{align*}
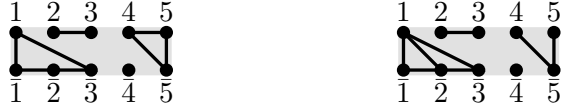
        \begin{array}{c}
            \begin{tikzpicture}[xscale=.5,yscale=.5,line width=1.25pt] 
                \foreach \i in {1,2,3,4,5}  { \path (\i,1.25) coordinate (T\i); \path(\i,1.8) node {$\i$}; \path (\i,.25) coordinate (B\i); \path(\i,-.3) node {$\ov{\i}$}; } 
                \filldraw[fill= black!12,draw=black!12,line width=4pt]  (T1) -- (T5) -- (B5) -- (B1) -- (T1);
                \draw[black] (T1)--(B3)--(B2)--(B1)--(T1);
                \draw[black] (T2)--(T3);
                \draw[black] (T4)--(T5)--(B5)--(T4);
                \colortop{1,2,3,4,5}{black}
                \colorbot{1,2,3,4,5}{black}
            \end{tikzpicture}
        \end{array} && \begin{array}{c}
            \begin{tikzpicture}[xscale=.5,yscale=.5,line width=1.25pt] 
                \foreach \i in {1,2,3,4,5}  { \path (\i,1.25) coordinate (T\i); \path(\i,1.8) node {$\i$}; \path (\i,.25) coordinate (B\i); \path(\i,-.3) node {$\ov{\i}$}; } 
                \filldraw[fill= black!12,draw=black!12,line width=4pt]  (T1) -- (T5) -- (B5) -- (B1) -- (T1);
                \draw[black] (T1)--(B3)--(B2)--(B1)--(T1);
                \draw[black] (T1)--(B2);
                \draw[black] (T2)--(T3);
                \draw[black] (T5)--(B5)--(T4);
                \colortop{1,2,3,4,5}{black}
                \colorbot{1,2,3,4,5}{black}
            \end{tikzpicture}
        \end{array}
    \end{align*}
\end{exam}

Note that there could be many such graphs that represent the set partition $\pi$, so we consider two graphs equivalent if their connected components give rise to the same set partition. The \defn{diagram} of $\pi$ is the equivalence class of graphs with the same connected components.

We can similarly consider a graph-theoretic representation of any multiset partition $\tpi\in\tilde\Pi_{2(r),k}$. This time we place $r$ vertices on the top labeled by the unbarred elements of the blocks of $\tpi$ in weakly increasing order and place $r$ vertices on the bottom labeled by the barred elements in weakly increasing order. We then connect the vertices in any way so that the labeled connected components taken together are $\tpi$.

\begin{exam}\label{MSPDiagramExample} The multiset partition $\tpi=\multi{\multi{1,\ov1,\ov1,\ov2},\multi{1,1},\multi{\ov2},\multi{2,2,\ov2}}$ could be represented by any of the following graphs.
    \begin{align*}
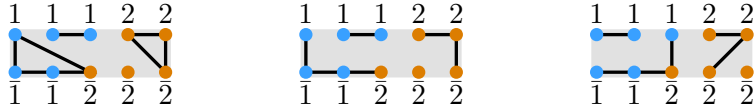
        \begin{array}{c}
            \begin{tikzpicture}[xscale=.5,yscale=.5,line width=1.25pt] 
                \foreach \i in {1,2,3,4,5}  { \path (\i,1.25) coordinate (T\i); \path (\i,.25) coordinate (B\i); } 
                \path(1,1.8) node {$1$};
                \path(2,1.8) node {$1$};
                \path(3,1.8) node {$1$};
                \path(4,1.8) node {$2$};
                \path(5,1.8) node {$2$};
                \path(1,-.3) node {$\ov1$};
                \path(2,-.3) node {$\ov1$};
                \path(3,-.3) node {$\ov2$};
                \path(4,-.3) node {$\ov2$};
                \path(5,-.3) node {$\ov2$};
                \filldraw[fill= black!12,draw=black!12,line width=4pt]  (T1) -- (T5) -- (B5) -- (B1) -- (T1);
                \draw[black] (T1)--(B3)--(B2)--(B1)--(T1);
                \draw[black] (T2)--(T3);
                \draw[black] (T4)--(T5)--(B5)--(T4);
                \colortop{1,2,3}{c1}
                \colortop{4,5}{c2}
                \colorbot{1,2}{c1}
                \colorbot{3,4,5}{c2}
            \end{tikzpicture}
        \end{array} && \begin{array}{c}
            \begin{tikzpicture}[xscale=.5,yscale=.5,line width=1.25pt] 
                \foreach \i in {1,2,3,4,5}  { \path (\i,1.25) coordinate (T\i); \path (\i,.25) coordinate (B\i); } 
                \path(1,1.8) node {$1$};
                \path(2,1.8) node {$1$};
                \path(3,1.8) node {$1$};
                \path(4,1.8) node {$2$};
                \path(5,1.8) node {$2$};
                \path(1,-.3) node {$\ov1$};
                \path(2,-.3) node {$\ov1$};
                \path(3,-.3) node {$\ov2$};
                \path(4,-.3) node {$\ov2$};
                \path(5,-.3) node {$\ov2$};
                \filldraw[fill= black!12,draw=black!12,line width=4pt]  (T1) -- (T5) -- (B5) -- (B1) -- (T1);
                \draw[black] (B3)--(B2)--(B1)--(T1);
                \draw[black] (T2)--(T3);
                \draw[black] (T4)--(T5)--(B5);
                \colortop{1,2,3}{c1}
                \colortop{4,5}{c2}
                \colorbot{1,2}{c1}
                \colorbot{3,4,5}{c2}
            \end{tikzpicture}
        \end{array} && \begin{array}{c}
            \begin{tikzpicture}[xscale=.5,yscale=.5,line width=1.25pt] 
                \foreach \i in {1,2,3,4,5}  { \path (\i,1.25) coordinate (T\i); \path (\i,.25) coordinate (B\i); } 
                \path(1,1.8) node {$1$};
                \path(2,1.8) node {$1$};
                \path(3,1.8) node {$1$};
                \path(4,1.8) node {$2$};
                \path(5,1.8) node {$2$};
                \path(1,-.3) node {$\ov1$};
                \path(2,-.3) node {$\ov1$};
                \path(3,-.3) node {$\ov2$};
                \path(4,-.3) node {$\ov2$};
                \path(5,-.3) node {$\ov2$};
                \filldraw[fill= black!12,draw=black!12,line width=4pt]  (T1) -- (T5) -- (B5) -- (B1) -- (T1);
                \draw[black] (T3)--(B3)--(B2)--(B1);
                \draw[black] (T1)--(T2);
                \draw[black] (T4)--(T5)--(B4);
                \colortop{1,2,3}{c1}
                \colortop{4,5}{c2}
                \colorbot{1,2}{c1}
                \colorbot{3,4,5}{c2}
            \end{tikzpicture}
        \end{array}
    \end{align*}
\end{exam}

Again we may have many graphs representing the same multiset partition. The \defn{diagram} of $\tpi$ is the equivalence class of graphs whose labeled connected components give $\tpi$.

We will often drop the labels on these graphs. A set partition diagram will be distinguished by the black color of its vertices, and it will be understood that the vertices are labeled in increasing order left-to-right. A multiset partition diagram will be distinguished by its colored vertices. Its vertices are understood to be labeled with \textcolor{c1}{\textbf{blue}}, \textcolor{c2}{\textbf{orange}}, \textcolor{c3}{\textbf{green}}, and \textcolor{c4}{\textbf{purple}} representing 1, 2, 3, and 4 respectively.

Because of this graphical representation, for a set $S$ with elements from $[r]\cup\ovb{r}$ we will sometimes refer to elements at the ``top'' of $S$ to mean the unbarred elements and elements at the ``bottom'' of $S$ to mean the barred elements, and likewise with multisets.

\subsection{Tableaux}

A \defn{partition} of $n$ is a weakly-decreasing sequence $\lambda$ of positive integers called parts which sum to $n$. We write $\ell(\lambda)$ for the number of parts of $\lambda$. We will write $\lambda\vdash n$ to mean that $\lambda$ is a partition of $n$ and write $|\lambda|=n$. Write $\lambda_i$ for the $i$th element of the sequence $\lambda$, called the $i$th part of $\lambda$, and ${\color{defncolor}\lambda^*}$ for the partition $(\lambda_2,\dots,\lambda_\ell)$ obtained by removing the first part. Given a partition $\lambda$, its \defn{Young diagram} is an array of left-justified boxes where the $i$th row from the bottom has $\lambda_i$ boxes.

\begin{exam} The Young diagram of the partition $(3,3,1)\vdash 7$ is \begin{align*}
    \inline
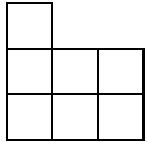
{\begin{ytableau}
        \; \\
        \; & \; & \; \\
        \; & \; & \;
    \end{ytableau}}
\end{align*}
\end{exam}

When we refer to the $i$th row of a Young diagram, we mean the $i$th row from the bottom, which corresponds to the $i$th part of $\lambda$.

A \defn{tableau} of shape $\lambda$ will be a filling of these boxes in $\lambda$'s Young diagram with mathematical objects---in this paper the objects will be positive integers, sets, or multisets. We will call these integer-valued, set-valued, and multiset-valued tableaux respectively. We take a moment to define some particular classes of tableaux.

Let $\rho$ be a set partition of $[r]$ and $\lambda\vdash n$ such that $|\lambda^*|\leq\ell(\rho)$. A \defn{set partition tableau} of shape $\lambda$ and content $\rho$ is a filling $T$ of the Young diagram of $\lambda$ with the blocks of $\rho$ along with $n-\ell(\rho)$ empty boxes in the first row. A \defn{standard set partition tableau} is a set partition tableau whose rows increase left-to-right and columns increase bottom-to-top with respect to the last-letter order. Write $\SPT_{\lambda,r}$ for the set of set partition tableaux of shape $\lambda$ with content a set partition of $[r]$ and write $\SSPT_{\lambda,r}$ for the subset of $\SPT_{\lambda,r}$ consisting of standard set partition tableaux.

\begin{exam}The tableau
    \begin{align*}
\inline
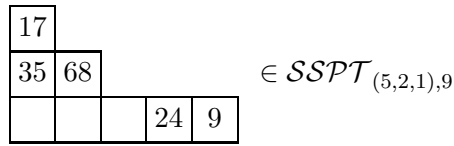
{\begin{ytableau}
        17\\
        35 & 68\\
        \; & \; & \; & 24 & 9
\end{ytableau}}&\in\SSPT_{(5,2,1),9}\\
        \intertext{is standard, whereas the tableau}
        \inline
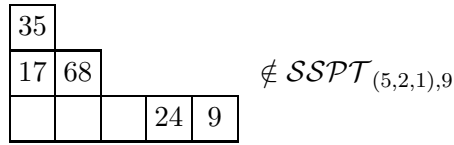
{\begin{ytableau}
        35\\
        17 & 68\\
        \; & \; & \; & 24 & 9
\end{ytableau}}&\notin\SSPT_{(5,2,1),9}\\
        \intertext{has a decrease in its first column, making it nonstandard, and the tableau}
        \inline
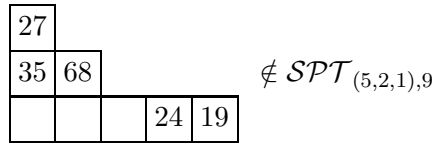
{\begin{ytableau}
        27\\
        35 & 68\\
        \; & \; & \; & 24 & 19
\end{ytableau}}&\notin\SPT_{(5,2,1),9}\\
    \end{align*} does not use each number in $[9]$ exactly once, making it not a set partition tableau.
\end{exam}

Write $\Lambda^{P_r(n)}$ for the set of $\lambda$ for which $\SSPT_{\lambda,r}\neq\emptyset$ (we choose this notation because $\Lambda^{P_r(n)}$ will also be an indexing set for the irreducible $P_r(n)$-modules).

Let $\tilde\rho$ be a multiset partition from $[k]$ and $\lambda\vdash n$ such that $|\lambda^*|\leq\ell(\tilde\rho)$. A \defn{multiset partition tableau} of shape $\lambda$ and content $\tilde\rho$ is a filling $\tilde{T}$ of the Young diagram with the blocks of $\tilde\rho$ along with $n-\ell(\tilde\rho)$ empty boxes in the first row. A \defn{semistandard multiset partition tableau} is a multiset partition tableau that strictly increases along columns and weakly increases along rows under the last-letter order. Write $\MPT_{\lambda,r,k}$ for the set of multiset partition tableaux of shape $\lambda$ with content a multiset partition from $[k]$ with a total of $r$ numbers. Write $\SSMPT_{\lambda,r,k}$ for the subset of $\MPT_{\lambda,r,k}$ consisting of semistandard multiset partition tableaux.

\begin{exam} The tableau
    \begin{align*}
        \inline{\begin{ytableau}
        112\\
        12\\
        11 & 11\\
        \; & \; & 22 & 3
\end{ytableau}}&\in \SSMPT_{(4,2,1,1),12,3}\\
        \intertext{is semistandard, whereas the tableau}
        \inline{\begin{ytableau}
        112\\
        11\\
        11 & 12 \\
        \; & \; & \; & 22 & 3
\end{ytableau}}&\notin \SSMPT_{(4,2,1,1),12,3}
    \end{align*} has a repeat in its first column, making it not semistandard.
\end{exam}

Write $\Lambda^{\MP_{r,k}(n)}$ for the set of partitions $\lambda\vdash n$ for which $\SSMPT_{\lambda,r,k}\neq\emptyset$. 

\begin{rema} We conclude this section on tableaux with two comments regarding the empty boxes. Note that, except for these empty boxes, the standard set partition tableaux and the semistandard multiset partition tableaux closely analogize standard Young tableaux (where each number in $[n]$ is used once) and semistandard Young tableaux (where repetition is allowed) respectively. Additionally, for reasons discussed in Section \ref{sec:partitionalgebra}, we will usually assume that $n\geq2r$. In this case, the number of empty boxes in a tableau will always be at least the number of boxes in its second row.
\end{rema}

\subsection{Double-Centralizer Theorem}

The algebras of interest in this paper arise as the algebras of endomorphisms of $\SG_n$-modules, commonly called \defn{centralizer algebras} of $\SG_n$. The following theorem summarizes a general case encompassing the Schur-Weyl duality and Howe duality discussed in the introduction as well as the duality between $\SG_n$ and its centralizer algebras.

\begin{theorem} \cite[Section 6.2.5]{procesi}\cite[Section 4.2.1]{gw}Let $A$ be a semisimple algebra acting faithfully on a module $V$ and set $B=\End_A(V)$. Then $B$ is semisimple and $\End_B(V)\cong A$. Furthermore, there is a set $P$ (a subset of the indexing set of the irreducible representations of $A$) such that for each $x\in P$, $W_A^x$ is an irreducible $A$-module occurring in the decomposition of $V$ as an $A$-module. If we set $W_B^x=\Hom(W_A^x, V)$, then $W_B^x$ is an irreducible $B$-module and the decomposition of $V$ as an $A\times B$-module is \[V\cong\bigoplus_{x\in P}W_A^x\otimes W_B^x.\] Moreover, the dimension of $W_A^x$ is equal to the multiplicity of $W_B^x$ in $V$ as a $B$-module and the dimension of $W_B^x$ is equal to the multiplicity of $W_A^x$ in $V$ as an $A$-module.
\end{theorem}

The decomposition in the above theorem gives a correspondence between irreducible representations of the algebras $A$ and $B$ and allows information like dimensions and multiplicities to be passed between them. In this paper, we are interested in setting $A=\C\SG_n$. In this case, the indexing set $P$ for the irreducible representations is a subset of the integer partitions of $n$.

\subsection{Partition Algebra}\label{sec:partitionalgebra}

For $r$ a positive integer and an indeterminate $x$, the partition algebra $P_r(x)$ is an associative algebra over $\C(x)$ first introduced as a generalization of the Temperley-Lieb algebra and the Potts model in statistical mechanics by Jones \cite{jones1993potts} and Martin \cite{martin, Martin1994AlgebrasIH}. When $x$ is specialized to an integer $n\geq 2r$, the algebra $P_r(n)$ is isomorphic to the algebra of endomorphisms $\End_{\SG_n}({V_n}^{\otimes r})$. The partition algebra has two distinguished bases: the \defn{orbit basis} $\{\TT_\pi:\pi\in\Pi_{2(r)}\}$ which arises naturally from the structure of $P_r(n)$ as a centralizer algebra and the \defn{diagram basis} $\{\LL_\pi:\pi\in\Pi_{2(r)}\}$ whose product has a combinatorial interpretation in terms of partition diagrams. The change-of-basis formula is obtained by summing over coarsenings of a diagram: \[\LL_\pi=\sum_{\nu\leq\pi}\TT_\nu.\]

The product formula for $\LL_\pi$ can be stated in terms of diagrams as follows. To compute the product of $\LL_\pi$ and $\LL_\nu$, place a graph representing $\pi$ above one representing $\nu$ and identify the vertices on the bottom of $\pi$ with the corresponding vertices of $\nu$ to create a three-tiered diagram. Let $\gamma$ be the restriction of this diagram to the very top and very bottom, preserving which vertices are connected and let $c(\pi,\nu)$ be the number of components entirely in the middle of the three-tier diagram. Then $\LL_\pi\LL_\nu=n^{c(\pi,\nu)}\LL_\gamma$.

\begin{exam} Here we show the product of two diagram basis elements. Notice that two components are entirely in the middle, giving a coefficient of $n^2$.
    \[\begin{array}{c}
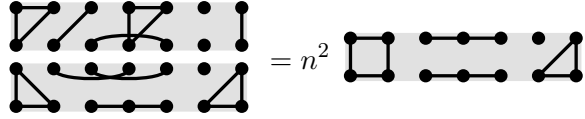
        \begin{tikzpicture}[xscale=.5,yscale=.5,line width=1.25pt] 
        \foreach \i in {1,2,3,4,5,6,7}  { \path (\i,1.25) coordinate (T\i); \path (\i,.25) coordinate (B\i); } 
            \filldraw[fill= black!12,draw=black!12,line width=4pt]  (T1) -- (T7) -- (B7) -- (B1) -- (T1);
            \draw[black] (T1)--(T2)--(B1)--(T1);
            \draw[black] (T3)--(B2);
            \draw[black] (T4)--(T5)--(B4)--(T4);
            \draw[black] (T7)--(B7);
            \draw[black] (B3) .. controls +(0,+.35) and +(0,+.35) .. (B5);
            \colortop{1,...,7}{black}
            \colorbot{1,...,7}{black}
        \end{tikzpicture} \\
       \begin{tikzpicture}[xscale=.5,yscale=.5,line width=1.25pt] 
        \foreach \i in {1,2,3,4,5,6,7}  { \path (\i,1.25) coordinate (T\i); \path (\i,.25) coordinate (B\i); } 
            \filldraw[fill= black!12,draw=black!12,line width=4pt]  (T1) -- (T7) -- (B7) -- (B1) -- (T1);
            \draw[black] (T1)--(B1)--(B2)--(T1);
            \draw[black] (B3)--(B5);
            \draw[black] (T7)--(B7)--(B6)--(T7);
            \draw[black] (T2) .. controls +(0,-.35) and +(0,-.35) .. (T4);
            \draw[black] (T3) .. controls +(0,-.35) and +(0,-.35) .. (T5);
            \colortop{1,...,7}{black}
            \colorbot{1,...,7}{black}
        \end{tikzpicture}
    \end{array}=n^2\begin{array}{c}
        \begin{tikzpicture}[xscale=.5,yscale=.5,line width=1.25pt] 
        \foreach \i in {1,2,3,4,5,6,7}  { \path (\i,1.25) coordinate (T\i); \path (\i,.25) coordinate (B\i); } 
            \filldraw[fill= black!12,draw=black!12,line width=4pt]  (T1) -- (T7) -- (B7) -- (B1) -- (T1);
            \draw[black] (T1)--(T2)--(B2)--(B1)--(T1);
            \draw[black] (T3)--(T5);
            \draw[black] (B3)--(B5);
            \draw[black] (T7)--(B7)--(B6)--(T7);
            \colortop{1,...,7}{black}
            \colorbot{1,...,7}{black}
        \end{tikzpicture}
    \end{array}\]
\end{exam}

In \cite{Halverson2018SetpartitionTA}, the authors construct the irreducible representations $P_r^\lambda$ of $P_r(n)$ for $n\geq 2r$ as a combinatorial action of the set partition diagrams on set partition tableaux. For $\lambda\in\Lambda^{P_r(n)}$, the module $P_r^\lambda$ is $\C\{v_T:T\in\SSPT_{\lambda,r}\}$ with action given as follows. For a set partition $\pi\in\Pi_{2(r)}$ to act on a tableau $T$, first pull out the content of $T$, a set partition of $[r]$, into a single row. Then, put $\pi$ on top of this row and identify the corresponding vertices. Form $T'$ by replacing the content of each box in $T$ with the set of vertices atop $\pi$ that the box is connected to, and for each block entirely in the top of $\pi$, include it as the content of a box in the first row of $T'$. If two blocks above the first row are combined or the content of a box does not connect to the top of the partition diagram, the result is zero.

\begin{exam}
    To illustrate the action, we show three examples of different diagrams acting on the same tableau. In the first example, the blocks $4$ and $5$ can be combined without resulting in zero because one of them is in the first row.
    \begin{align*}
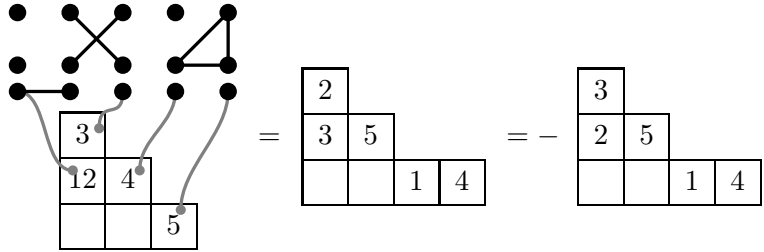
        \begin{tabular}{c}
             \begin{tikzpicture}[line width=1.25pt, xscale=.7, yscale=.7]
                    \path(1.35,1) node {$\begin{ytableau}
        3\\
        12 & 4\\
        \; & \; & 5
\end{ytableau}$};
                    \path (.3,1.2) coordinate (T21);
                    \path (.8,2) coordinate (T31);
                    \path (2.35,.45) coordinate (T13);
                    \path (1.57,1.2) coordinate (T22);
                    \path (-.75,2.7) coordinate (C1);
                    \path (.25,2.7) coordinate (C2);
                    \path (1.25,2.7) coordinate (C3);
                    \path (2.25,2.7) coordinate (C4);
                    \path (3.25,2.7) coordinate (C5);
                    \path (-.75,3.2) coordinate (B1);
                    \path (.25,3.2) coordinate (B2);
                    \path (1.25,3.2) coordinate (B3);
                    \path (2.25,3.2) coordinate (B4);
                    \path (3.25,3.2) coordinate (B5);
                    \path (-.75,4.2) coordinate (T1);
                    \path (.25,4.2) coordinate (T2);
                    \path (1.25,4.2) coordinate (T3);
                    \path (2.25,4.2) coordinate (T4);
                    \path (3.25,4.2) coordinate (T5);
                    \draw[gray] (T21) .. controls +(-.6,0) and +(+.6,0) .. (C1);
                    \draw[gray] (T31) .. controls +(0,+.6) and +(0,-.6) .. (C3);
                    \draw[gray] (T22) .. controls +(0,+.6) and +(0,-.6) .. (C4);
                    \draw[gray] (T13) .. controls +(0,+.8) and +(0,-.8) .. (C5);
                    \draw[black] (C1) -- (C2);
                    \draw[black] (T2)--(B3);
                    \draw[black] (B2)--(T3);
                    \draw[black] (B4)--(B5)--(T5)--(B4);
                    \filldraw[fill=gray,draw=gray,line width = 1pt] (T21) circle (2pt);
                    \filldraw[fill=gray,draw=gray,line width = 1pt] (T31) circle (2pt);
                    \filldraw[fill=gray,draw=gray,line width = 1pt] (T13) circle (2pt);
                    \filldraw[fill=gray,draw=gray,line width = 1pt] (T22) circle (2pt);
                    \foreach \i in {C1, C2, C4, B1, B2, B4, T1, T2, T3, B5, C5, C3, B3, T4, T5} {\filldraw[fill=black,draw=black,line width = 1pt] (\i) circle (4pt);}
            \end{tikzpicture}
        \end{tabular}&=\inline{\begin{ytableau}
            2\\
            3 & 5\\
            \; & \; & 1 & 4
        \end{ytableau}}=-\,\inline{\begin{ytableau}
            3\\
            2 & 5\\
            \; & \; & 1 & 4
        \end{ytableau}}
    \end{align*}
    In the following example, the content of a box above the first row does not connect to the top of $\pi$, so the result is zero.
    \begin{align*}
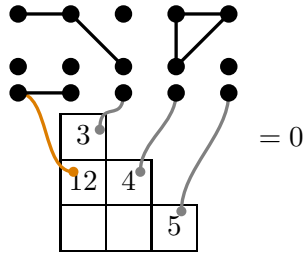
        \begin{tabular}{c}
             \begin{tikzpicture}[line width=1.25pt, xscale=.7, yscale=.7]
                    \path(1.35,1) node {$\begin{ytableau}
        3\\
        12 & 4\\
        \; & \; & 5
\end{ytableau}$};
                    \path (.3,1.2) coordinate (T21);
                    \path (.8,2) coordinate (T31);
                    \path (2.35,.45) coordinate (T13);
                    \path (1.57,1.2) coordinate (T22);
                    \path (-.75,2.7) coordinate (C1);
                    \path (.25,2.7) coordinate (C2);
                    \path (1.25,2.7) coordinate (C3);
                    \path (2.25,2.7) coordinate (C4);
                    \path (3.25,2.7) coordinate (C5);
                    \path (-.75,3.2) coordinate (B1);
                    \path (.25,3.2) coordinate (B2);
                    \path (1.25,3.2) coordinate (B3);
                    \path (2.25,3.2) coordinate (B4);
                    \path (3.25,3.2) coordinate (B5);
                    \path (-.75,4.2) coordinate (T1);
                    \path (.25,4.2) coordinate (T2);
                    \path (1.25,4.2) coordinate (T3);
                    \path (2.25,4.2) coordinate (T4);
                    \path (3.25,4.2) coordinate (T5);
                    \draw[c2] (T21) .. controls +(-.6,0) and +(+.6,0) .. (C1);
                    \draw[gray] (T31) .. controls +(0,+.6) and +(0,-.6) .. (C3);
                    \draw[gray] (T22) .. controls +(0,+.6) and +(0,-.6) .. (C4);
                    \draw[gray] (T13) .. controls +(0,+.8) and +(0,-.8) .. (C5);
                    \draw[black] (C1) -- (C2);
                    \draw[black] (T1)--(T2)--(B3);
                    \draw[black] (T4)--(T5)--(B4)--(T4);
                    \filldraw[fill=c2,draw=c2,line width = 1pt] (T21) circle (2pt);
                    \filldraw[fill=gray,draw=gray,line width = 1pt] (T31) circle (2pt);
                    \filldraw[fill=gray,draw=gray,line width = 1pt] (T13) circle (2pt);
                    \filldraw[fill=gray,draw=gray,line width = 1pt] (T22) circle (2pt);
                    \foreach \i in {C1, C2, C4, B1, B2, B4, T1, T2, T3, B5, C5, C3, B3, T4, T5} {\filldraw[fill=black,draw=black,line width = 1pt] (\i) circle (4pt);}
            \end{tikzpicture}
        \end{tabular}&=0\\
        \intertext{In the following example, the conents of the boxes containing $12$ and $3$ become connected, so the result is zero.}
        \begin{tabular}{c}
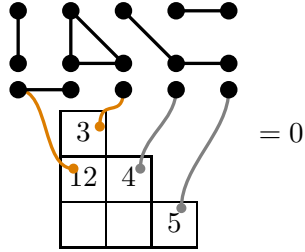
             \begin{tikzpicture}[line width=1.25pt, xscale=.7, yscale=.7]
                    \path(1.35,1) node {$\begin{ytableau}
        3\\
        12 & 4\\
        \; & \; & 5
\end{ytableau}$};
                    \path (.3,1.2) coordinate (T21);
                    \path (.8,2) coordinate (T31);
                    \path (2.35,.45) coordinate (T13);
                    \path (1.57,1.2) coordinate (T22);
                    \path (-.75,2.7) coordinate (C1);
                    \path (.25,2.7) coordinate (C2);
                    \path (1.25,2.7) coordinate (C3);
                    \path (2.25,2.7) coordinate (C4);
                    \path (3.25,2.7) coordinate (C5);
                    \path (-.75,3.2) coordinate (B1);
                    \path (.25,3.2) coordinate (B2);
                    \path (1.25,3.2) coordinate (B3);
                    \path (2.25,3.2) coordinate (B4);
                    \path (3.25,3.2) coordinate (B5);
                    \path (-.75,4.2) coordinate (T1);
                    \path (.25,4.2) coordinate (T2);
                    \path (1.25,4.2) coordinate (T3);
                    \path (2.25,4.2) coordinate (T4);
                    \path (3.25,4.2) coordinate (T5);
                    \draw[c2] (T21) .. controls +(-.6,0) and +(+.6,0) .. (C1);
                    \draw[c2] (T31) .. controls +(0,+.6) and +(0,-.6) .. (C3);
                    \draw[gray] (T22) .. controls +(0,+.6) and +(0,-.6) .. (C4);
                    \draw[gray] (T13) .. controls +(0,+.8) and +(0,-.8) .. (C5);
                    \draw[black] (C1) -- (C2);
                    \draw[black] (T1)--(B1);
                    \draw[black] (B2)--(B3)--(T2)--(B2);
                    \draw[black] (T3)--(B4)--(B5);
                    \draw[black] (T4)--(T5);
                    \filldraw[fill=c2,draw=c2,line width = 1pt] (T21) circle (2pt);
                    \filldraw[fill=c2,draw=c2,line width = 1pt] (T31) circle (2pt);
                    \filldraw[fill=gray,draw=gray,line width = 1pt] (T13) circle (2pt);
                    \filldraw[fill=gray,draw=gray,line width = 1pt] (T22) circle (2pt);
                    \foreach \i in {C1, C2, C4, B1, B2, B4, T1, T2, T3, B5, C5, C3, B3, T4, T5} {\filldraw[fill=black,draw=black,line width = 1pt] (\i) circle (4pt);}
            \end{tikzpicture}
        \end{tabular}&=0
    \end{align*}
\end{exam}

The result $T$ of the above process may not be a standard set partition tableau, so we need to make sense of what it means to write $v_T$ for $T$ nonstandard. The algorithm for writing $v_T$ as a linear combination of standard tableaux is called the \defn{straightening algorithm}. The straightening algorithm for $P_r^\lambda$ is the same as for the Specht modules of $\SG_n$ applied to the rows above the first row of $T$, a complete treatment of which can be found in \cite{sagan}, but we summarize some key features that will be important for our constructions. Given $T\in\SPT_{\lambda,r}$ nonstandard, the straightening algorithm writes $v_T$ as a linear combination \[v_T=\sum_{S\in\SSPT_{\lambda,r}} c_S v_S\] for some coefficients $c_S.$

The relations between the $v_T$ are called Garnir relations and are generally complicated, but one special case will be particularly useful to us: if $T'$ is the result of exchanging two boxes of $T$ above the first row that sit in the same column, then $v_{T'}=-v_T$.

\subsection{Multiset Partition Algebra}

The multiset partition algebra naturally arises by restricting the action of $GL_n$ on $\PP^r(V_{n,k})$ in Howe duality to the permutation matrices. One can think of elements in $\PP^r(V_{n,k})$ as homogeneous polynomials of degree $r$ in indeterminates $x_{ij}$ for $1\leq i\leq n$ and $1\leq j\leq k$. The action of $GL_n$ on the space $\PP^r(V_{n,k})$ can be described on monomials as follows. Given a matrix $M=(m_{ij})\in GL_n$, its inverse acts on an element of $U_\bfa$ as follows:

\begin{align}
    M^{-1}.x_{ij}&=\sum_{\ell=1}^n m_{i\ell}x_{\ell j}\nonumber\\
    M^{-1}.x_{i_1j_1}\dots x_{i_rj_r}&=\left(\sum_{\ell_1=1}^n m_{i_1\ell_1}x_{\ell_1 j_1}\right)\cdots\left(\sum_{\ell_r=1}^n m_{i_r\ell_r}x_{\ell_r j_r}\right)\nonumber\\
    &=\sum_{\ell_1,\dots,\ell_r=1}^n m_{i_1\ell_1}\cdots m_{i_r\ell_r} (x_{\ell_1j_1}\dots x_{\ell_rj_r})\label{eq:Ua_computation}
\end{align}

In \cite{orellana2020howe}, the authors introduce a multiset partition algebra $\MP_{r,k}(x)$ with bases indexed by elements of $\tilde\Pi_{2(r),k}$. When $x$ is specialized to an integer $n\geq 2r$, the algebra $\MP_{r,k}(n)$ is isomorphic to $\End_{\SG_n}(\PP^r(V_{n,k}))$ where $\SG_n$ acts by the restriction of the $GL_n$ action above to the $n\times n$ permutation matrices. The authors obtain a basis analogous to the orbit basis for $P_r(n)$ and prove that for $n\geq 2r$ the irreducible representations of $\MP_{r,k}(n)$ occurring in the decomposition \begin{align}
    \PP^r(V_{n,k})\cong\bigoplus_{\lambda\in\Lambda^{\MP_{r,k}(n)}}W_{\SG_n}^\lambda\otimes W_{\MP_{r,k}(n)}^\lambda \label{eq:mprk_decomp}
\end{align} have dimension $\dim(W_{\MP_{r,k}(n)}^\lambda)=\#\SSMPT_{\lambda,r,k}$.

In \cite{MR4090833}, the authors generalize the Robinson–Schensted–Knuth algorithm to two-line arrays of multisets. This algorithm establishes a correspondence between multiset partitions in $\tilde\Pi_{2(r),k}$ and pairs of elements of $\SSMPT_{\lambda,r,k}$, 

\[\tilde\Pi_{2(r),k}\stackrel{\sim}{\longleftrightarrow}\biguplus_{\lambda\in\Lambda^{\MP_{r,k}(n)}}\SSMPT_{\lambda,r,k}\times \SSMPT_{\lambda,r,k}\]

showing that \begin{align*}
    \dim(\MP_{r,k}(n))=\sum_{\lambda\in\Lambda}\left(\dim(W_{\MP_{r,k}(n)}^\lambda)\right)^2.
\end{align*} Hence, the $\abs{\Lambda^{\MP_{r,k}(n)}}$ irreducible representations occurring in Equation \ref{eq:mprk_decomp} are pairwise nonisomorphic, and each irreducible representation of $\MP_{r,k}(n)$ is isomorphic to one representation in the set.

\section{Orbits Under Young Subgroups}\label{sec:orbits}

The symmetric group algebra $\C\SG_r$ sits naturally inside of $P_r(n)$ as the diagrams whose blocks pair one vertex on top with one on the bottom. For $\sigma\in\SG_r$, we will write $\LL_\sigma$ for the diagram basis element corresponding to the set partition $\{\{\sigma(1),\ov1\},\dots\{\sigma(r),\ov{r}\}\}$. This embedding leads to natural actions on set partitions and set partition tableaux. In this section, we collect up some useful facts about orbits of these actions when they are restricted to Young subgroups. For $\bfa\in W_{r,k}$, write $\SG_\bfa=\SG_{\{1,\dots,\bfa_1\}}\times\dots\times\SG_{\{\bfa_1+\dots+\bfa_{k-1},\dots,\bfa_1+\dots+\bfa_k\}}$ for the corresponding Young subgroup of $\SG_r$.

A pair of permutations $(\sigma_1,\sigma_2)\in \SG_r\times\SG_r$ can act on a set partition $\pi\in\Pi_{2(r)}$ by taking the product $\LL_{\sigma_1}\LL_{\pi}\LL_{\sigma_2}$. The resulting set partition $\sigma_1.\pi.\sigma_2$ can be obtained by replacing each $i$ in $\pi$ with $\sigma_1(i)$ and each $\ov{i}$ in $\pi$ with $\ov{{\sigma_2}^{-1}(i)}$.

Given $\bfa,\bfb\in W_{r,k}$, define the \defn{coloring map} $\kappa_{\bfa,\bfb}:\Pi_{2(r)}\to\tilde\Pi_{2(r)}$ to be the function given by making the following substitutions.

\begin{align*}
    i \mapsto \begin{cases}
        1 & i\leq \bfa_1\\
        2 & \bfa_1<i\leq \bfa_1+\bfa_2\\
        \vdots\\
        k & \bfa_1+\dots+\bfa_{k-1}<i
    \end{cases} && \ov i \mapsto \begin{cases}
        \ov1 & i\leq \bfb_1\\
        \ov2 & \bfb_1<i\leq \bfb_1+\bfb_2\\
        \vdots\\
        \ov k & \bfb_1+\dots+\bfb_{k-1}<i
    \end{cases}
\end{align*}

On diagrams, we can think of $\kappa_{\bfa,\bfb}$ as coloring in the diagram of $\pi$ with colors whose multiplicities are given by $\bfa$ on top and $\bfb$ on bottom.

\begin{exam}\label{ex:kappa} Two set partitions that map to the same multiset partition under the coloring map $\kappa_{(1,2,1),(2,0,2)}$:
    \begin{align*}
    \kappa_{(1,2,1),(2,0,2)}\left(\inlinediagram{4}{\draw[black] (T1)--(T2)--(B2)--(B1)--(T1);
                \draw[black] (T3)--(B3);
                \draw[black] (T4)--(B4);
                \colortop{1,2,3,4}{black}
                \colorbot{1,2,3,4}{black}}\right)&=\inlinediagram{4}{\draw[black] (T1)--(T2)--(B2)--(B1)--(T1);
                \draw[black] (T3)--(B3);
                \draw[black] (T4)--(B4);
                \colortop{1}{c1}
                \colortop{2,3}{c2}
                \colortop{4}{c3}
                \colorbot{1,2}{c1}
                \colorbot{3,4}{c3}}\\
    \kappa_{(1,2,1),(2,0,2)}\left(\inlinediagram{4}{\draw[black] (T1)--(T2)--(B2)--(B1)--(T1);
                \draw[black] (T3)--(B4);
                \draw[black] (T4)--(B3);
                \colortop{1,2,3,4}{black}
                \colorbot{1,2,3,4}{black}}\right)&=\inlinediagram{4}{\draw[black] (T1)--(T2)--(B2)--(B1)--(T1);
                \draw[black] (T3)--(B4);
                \draw[black] (T4)--(B3);
                \colortop{1}{c1}
                \colortop{2,3}{c2}
                \colortop{4}{c3}
                \colorbot{1,2}{c1}
                \colorbot{3,4}{c3}}\\
                \intertext{Then, because the two green vertices at the bottom of the diagram correspond to the same value, they can be exchanged without changing the underling multiset partition.}
                &=\inlinediagram{4}{\draw[black] (T1)--(T2)--(B2)--(B1)--(T1);
                \draw[black] (T3)--(B3);
                \draw[black] (T4)--(B4);
                \colortop{1}{c1}
                \colortop{2,3}{c2}
                \colortop{4}{c3}
                \colorbot{1,2}{c1}
                \colorbot{3,4}{c3}}\\
\end{align*}
\end{exam}

\begin{lemma}\label{lem:msp_correspondence}
    For $\bfa,\bfb\in W_{r,k}$, the map $\kappa_{\bfa,\bfb}$ induces a map \begin{align*}
        \ov\kappa_{\bfa,\bfb}:\Pi_{2(r)}/(\SG_\bfa\times\SG_\bfb)\rightarrow\tilde\Pi_{2(r),\bfa,\bfb}
    \end{align*} from set partitions modulo the action of $\SG_\bfa\times\SG_\bfb$ to multiset partitions whose elements have multiplicities given by $\bfa$ and $\bfb$. These maps taken together give a correspondence \begin{align}
    \tilde\Pi_{2(r),k}\stackrel{\sim}{\longleftrightarrow}\biguplus_{\bfa,\bfb\in W_{r,k}} \Pi_{2(r)}/(\SG_\bfa\times\SG_\bfb)\label{eq:msp_correspondence}
\end{align}
\end{lemma}

\begin{proof}

For $(\sigma_1,\sigma_2)\in \SG_\bfa\times\SG_\bfb$, it is clear that $\kappa_{\bfa,\bfb}(\sigma_1.\pi.\sigma_2)=\kappa_{\bfa,\bfb}(\pi)$ for all $\pi\in\Pi_{2(r)}$. Hence, $\kappa_{\bfa,\bfb}$ induces a map $\ov\kappa_{\bfa,\bfb}:\Pi_{2(r)}/(\SG_\bfa\times\SG_\bfb)\rightarrow\tilde\Pi_{2(r),k}$. Conversely, if $\kappa_{\bfa,\bfb}(\pi)=\kappa_{\bfa,\bfb}(\pi')$, then $\pi$ and $\pi'$ are in the same orbit. Hence, the map $\ov\kappa_{\bfa,\bfb}$ is injective. Given $\tpi\in\tilde\Pi_{2(r),k}$ whose unbarred and barred multiplicities are given by $\bfa$ and $\bfb$ respectively, we can easily create a set partition $\pi\in\Pi_{2(r)}$ such that $\kappa_{\bfa,\bfb}(\pi)=\tpi$ by simply taking any graph representing $\tpi$ and forgetting the data of the colored vertices. Hence, the maps $\ov\kappa_{\bfa,\bfb}$ taken together as a map $\biguplus_{\bfa,\bfb\in W_{r,k}} \Pi_{2(r)}/(\SG_\bfa\times\SG_\bfb)\to\tilde\Pi_{2(r),k}$ give a bijection. This map gives us a correspondence between multiset partitions and orbits of set partitions under an action of a pair of Young subgroups.

\end{proof}

We now obtain a second action from the module structure of $P_r^\lambda$. A permutation $\sigma\in\SG_r$ acts on the set $\SPT_{\lambda,r}$ by replacing each entry $i$ of a tableau $T$ with $\sigma(i)$.

\begin{exam}
    \[(1\,3\,2)(4).\inline{\begin{ytableau}
        23 \\
        1 & 4 \\
        \; & \;
    \end{ytableau}}=\inline{\begin{ytableau}
        12 \\
        3 & 4 \\
        \; & \;
    \end{ytableau}}\]
\end{exam}

Like before, we define a surjective coloring map $\kappa_\bfa:\SPT_{\lambda,r}\to\MPT_{\lambda,r,k}$ for $\bfa\in W_{r,k}$ that replaces the numbers $\{1,\dots,\bfa_1\}$ with 1, $\{\bfa_1+1,\dots,\bfa_1+\bfa_2\}$ with 2, etc.. 

\begin{exam} Two set partition tableaux that are sent to the same multiset partition tableau by the coloring map $\kappa_{(3,1)}$:
\begin{align*}
    \kappa_{(3,1)}\left(\inline{\begin{ytableau}
        23 \\
        1 & 4 \\
        \; & \;
    \end{ytableau}}\right)=\kappa_{(3,1)}\left(\inline{\begin{ytableau}
        12 \\
        3 & 4 \\
        \; & \;
    \end{ytableau}}\right)=\inline{\begin{ytableau}
        11 \\
        1 & 2 \\
        \; & \;
    \end{ytableau}}
\end{align*}
\end{exam}

\begin{lemma}
    There is a correspondence \[\MPT_{\lambda,r,k}\stackrel{\sim}{\longleftrightarrow}\biguplus_{\bfa\in W_{r,k}}\SPT_{\lambda,r}/\SG_\bfa\] where the multiset partition tableau $\tilde{T}$ maps to the set $\kappa_\bfa^{-1}(\tilde{T})$.
\end{lemma}

\begin{proof}
    The proof is completely analogous to that of Lemma \ref{lem:msp_correspondence}
\end{proof}

\begin{rema}\label{rmk:semistandard_tableaux} The orbit of a standard set partition tableau $T$ corresponds to a multiset partition tableau $\tT$ whose rows and columns weakly increase. That is, $\tT$ is semistandard except for possible repeats within columns.
\end{rema}

\section{The Painted Algebra Construction}\label{sec:construction}

This section considers an algebra $B$ with a subset $\{e_1,\dots,e_m\}$ of its idempotents (which we call \defn{distinguished idempotents}) and $M$ a $B$-module. We provide constructions of a new algebra $\tilde{B}$ called the corresponding \defn{painted algebra} and a $\tilde{B}$-module $\tilde{M}$ called the \defn{painted module}. First, we consider the setting in which this construction will naturally arise in Section \ref{sec:diagram_like_basis}.

\begin{lemma}\label{lem:decompose_by_idempotents} Let $A$ be an algebra and $V$ a semisimple $A$-module. Let $e_1,\dots,e_m\in\End_A(V)$ be idempotents. Then \begin{align*}
    \End_A\left(\bigoplus_{i=1}^me_i V\right)\cong \bigoplus_{i,j=1}^m e_i\End_A(V) e_j
\end{align*} where the product on the right hand side of $e_i\varphi e_j\in e_i\End_A(V)e_j$ and $e_k\psi e_\ell\in e_k\End_A(V)e_\ell$ is given by \begin{align*}
    (e_i\varphi e_j)\cdot(e_k\psi e_\ell)=\delta_{jk}e_i(\varphi e_j\psi)e_\ell\in e_i\End_A(V) e_\ell.
\end{align*} where the product to the right of the equal sign is taken in $\End_A(V)$ and $\delta_{jk}$ is the Kronecker delta function.
\end{lemma}

\begin{proof}
First, note that \begin{align*}
    \End_A\left(\bigoplus_{i=1}^me_i V\right)&\cong\bigoplus_{i,j=1}^m \Hom_A(e_j V, e_i V).
\end{align*}

An element $e_i\varphi e_j\in e_i\End_A(V) e_j$ can be viewed as a map $e_j V\to e_i V$, giving rise to an injective linear map $\Phi:e_i\End_A(V)e_j\to \Hom_A(e_jV, e_i V)$. Because $V$ is semisimple, the submodule $e_jV$ has a complementary submodule $U$ so that $V=e_j V\oplus U$. A map $\psi:e_j V\to e_iV\subseteq V$ can be extended to a map $\ov\psi:V\to V$ by setting $\ov\psi(u)=0$ for all $u\in U$. Then for any $e_jv\in e_jV$, we have that $e_i\ov\psi e_j(e_jv)=e_i\ov\psi(e_jv)=e_i\psi(e_jv)$. Hence, the map $\Phi$ is also surjective, so it is an isomorphism.

    \begin{align*}
        \End_A\left(\bigoplus_{i=1}^me_i V\right)&\cong\bigoplus_{i,j=1}^m e_i\End_A(V)e_j
    \end{align*}

    Suppose $j\neq k$. The output of the map $e_k\psi e_\ell$ is in $e_k V$, so the component of its output in $e_j V$ is zero. Hence, $(e_i\varphi e_j)\cdot(e_k\psi e_\ell)=0$. If instead $j=k$, we have that $(e_i\varphi e_j)\cdot(e_k\psi e_\ell)=e_i\varphi e_j e_k\psi e_\ell$. Hence, \begin{align*}
        (e_i\varphi e_j)\cdot(e_k\psi e_\ell)&=\delta_{jk}e_i\varphi (e_j)^2\psi e_\ell\\
        &=\delta_{jk}e_i(\varphi e_j\psi)e_\ell\in e_i\End_A(V) e_\ell
    \end{align*}
\end{proof}

Lemma \ref{lem:decompose_by_idempotents} motivates the following definition.

\begin{defi}
For a semisimple algebra $B$ along with distinguished idempotents $\{e_1,\dots,e_m\}$, the corresponding painted algebra with respect to these idempotents is \[\tB=\bigoplus_{i,j=1}^m e_i B e_j\] with multiplication as in Lemma \ref{lem:decompose_by_idempotents}. For a $B$-module $M$, the corresponding painted module with respect to these idempotents is \[\tilde{M}=\bigoplus_{i=1}^m e_i M\]  where $(e_i b e_j).e_k m=\delta_{jk}(e_ibe_j).m$, the action on the right side of the equality begin that of $B$ on $M$.
\end{defi}

\begin{exam} We now illustrate the painted algebra construction with two extreme examples.
    
    \begin{enumerate}[label={\alph*}]
        \item[(a)] If $e_1,\dots,e_m\in B$ are already orthogonal idempotents which sum to the identity, then $\tilde{B}\cong B$, and the painted algebra simply corresponds to the usual decomposition of the algebra by a system of orthogonal idempotents.
        \item[(b)] If each of $e_1,\dots,e_m\in B$ are the identity, then \[\tilde{B}\cong\underbrace{B\oplus\dots\oplus B}_{m}.\]
    \end{enumerate}

\end{exam}

To conclude this section, we show that the irreducible $\tB$-modules are precisely the painted irreducible $B$-modules.

\begin{lemma}\label{lem:symmetrized_modules}
    Let $B$ be a semisimple algebra along with distinguished idempotents $\{e_1,\dots,e_m\}$. Then \begin{enumerate}
        \item For any simple $B$-module $S$, either $\tS=\{0\}$ or $\tS$ is a simple $\tB$-module.
        \item For any simple $\tB$-module $P$, there is a simple $B$-module $S$ so that $\tS\cong P$.
    \end{enumerate}
\end{lemma}

\begin{proof} \phantom{a}
    \begin{enumerate}
        \item Suppose $\tS\neq\{0\}$, let $\tilde{s}=\sum_{i=1}^m e_is_i\in\tS$ be nonzero, and fix any $j\in[m]$ such that $e_js_j$ is nonzero. We show that any such $\tilde{s}$ generates $\tS$ as a $\tB$-module, and so $\tS$ is simple. Note that $e_j\tilde{s}=e_js_j\in e_j S\subseteq S$. Because $S$ is simple, for any $e_ks\in S$ there exists $b\in B$ such that $be_j\tilde{s}=s$ and so $e_kb e_j\tilde{s}=e_k s$. Because the elements $e_k s$ span $\tS$, we see that $\tS$ is generated by any nonzero element and hence is simple.

        \item The remainder of the proof generalizes an argument for the case of a single idempotent found in the proof of Theorem 1.10.14 of \cite{linckelmann_2018}. Suppose $P$ is a simple $\tB$-module and define a $B$-module \[U=\left(\bigoplus_{i=1}^m Ae_i\right)\otimes_{\tB}P\] where $B$ acts on the direct sum by left-multiplication. Write $S=U/M$ where $M$ is a maximal submodule of $U$. Then $S$ is a simple $B$-module. The goal is now to define a nonzero $\tB$-module map from $P$ to the painted module $\tS$.

    Consider the quotient map $q:U\to S$ and suppose $e_i\otimes p\neq0$. We claim that $e_i\otimes p$ generates $U$ as a $B$-module and hence $q(e_i\otimes p)\neq0$ (else $e_i\otimes t\in M$, which would mean that $M=U$). To show this, we need only demonstrate that for any fixed $\ell\in[m]$ and $p'\in P$, there exists an element $b\in B$ such that $b.(e_i\otimes p)=e_\ell\otimes p'$. Because $P$ is a simple module and $e_i.p\neq0$, there exists an element  $\tilde{b}=\sum_{j,k}e_jb_{j,k}e_k\in \tB$ such that \begin{align*}
        \tilde{b}.e_i\otimes p&=\sum_{j,k}e_jb_{j,k}e_k.e_i\otimes p\\
        &=(e_1+\dots+e_m)\otimes (\sum_{j}e_jb_{j,i}e_i.p)\\
        &=(e_1+\dots+e_m)\otimes e_\ell p'\\
        &=e_\ell\otimes p'.
    \end{align*} Set $b=\sum_{j}e_jb_{j,i}e_i$. Then $b.e_i\otimes p=e_\ell\otimes p'$.

    Let $p\in P$ be nonzero and note that in $\tB$ we have $e_1+\dots+e_m=1$, so $(e_1+\dots+e_m).p=p$. Hence, some $e_i.p\neq0$. Consider the map \[e_i q e_i:e_iU\to e_i S.\] Because $e_i.(e_i\otimes p)=e_i\otimes p \in e_i U$, we know that $e_i q e_i$ is nonzero.  Define a $\tB$-module map $\bigoplus_{i=1}^m e_iU\to\bigoplus_{i=1}^m e_iS$ by $e_jr\mapsto e_j q(r)$. By the above observation, this is a nonzero module map. By Schur's Lemma, it is an isomorphism and hence \begin{align*}
        \tS&\cong \bigoplus_{i=1}^m e_iU\\
        &\cong(\bigoplus_{i,j=1}^m e_iBe_j)\otimes P\\
        &\cong P.
    \end{align*}
    \end{enumerate}
\end{proof}

\section{Painted Diagram Algebras and a Diagram-Like Basis}\label{sec:diagram_like_basis}

In this section, our goal is to decompose $\PP^r(V_{n,k})$ in a way that allows us to use the results of Section \ref{sec:construction}.

Let $U_\bfa$ be the span of monomials of the form $x_{i_1j_1}\cdots x_{i_rj_r}$ where for each $1\leq m\leq k$, exactly $\bfa_m$ of the values $j_1,\dots,j_r$ are equal to $m$. For example, the monomials $x_{11}x_{21}x_{23}$ and $x_{21}x_{21}x_{23}$ are both in $U_{(2,0,1)}\subset \PP^3(V_{2,3})$.  To apply Lemma \ref{lem:decompose_by_idempotents} and Lemma \ref{lem:symmetrized_modules}, we will need to write the subspace $U_\bfa$ as the projection by some idempotent in $P_r(n)$. We define the following idempotent for $\bfa\in W_{r,k}$: \[s_\bfa=\frac{1}{\abs{\SG_\bfa}}\sum_{\sigma\in\SG_\bfa}\LL_\sigma.\]

From Equation \ref{eq:Ua_computation} we see that the subspace $U_\bfa$ is in fact a $GL_n$-submodule of $\PP^r(V_{n,k})$. This gives a decomposition \[\PP^r(V_{n,k})=\bigoplus_{\bfa\in W_{r,k}}U_\bfa\] as a $GL_n$-module. We now use this decomposition to construct a linear isomorphism $\Phi:\bigoplus_{\bfa\in W_{r,k}}s_\bfa {V_n}^{\otimes r}\to \PP^r(V_{n,k})$.

For $\bfa\in W_{r,k}$, define a linear map $\Phi_\bfa:{V_n}^{\otimes r}\to U_\bfa$ by \[\Phi_\bfa(e_{i_1}\otimes\dots\otimes e_{i_r})=\prod_{m=1}^{\bfa_1}x_{i_m1}\prod_{m=\bfa_1+1}^{\bfa_1+\bfa_2}x_{i_m2}\dots \prod_{m=\bfa_1+\dots+\bfa_{k-1}+1}^{r}x_{i_mk}.\]

\begin{exam} Note that two different tensors may map to equal monomials as in the following example.
    \begin{align*}
        \Phi_{(1,2,2)}(e_2\otimes e_2\otimes e_1\otimes e_1\otimes e_2)&=x_{21}x_{22}x_{12}x_{13}x_{23}\\
        \Phi_{(1,2,2)}(e_2\otimes e_1\otimes e_2\otimes e_1\otimes e_2)&=x_{21}x_{12}x_{22}x_{13}x_{23}
    \end{align*}

    Here, the tensors map to the same monomial because the second is obtained from the first by swapping two factors mapping to indeterminates of the same form $x_{i2}$.
\end{exam}

For ease of notation, we will write $e_\bfi=e_{\bfi_1}\otimes\dots\otimes e_{\bfi_r}$ for $\bfi\in[n]^r$. It is clear that $\Phi_\bfa$ is surjective and that $\Phi_\bfa(e_\bfi)=\Phi_\bfa(e_{\bfi'})$ exactly when $e_{\bfi'}$ can be obtained from $e_\bfi$ by rearranging factors grouped into the same product above. That is, $e_{\bfi'}=\sigma(e_\bfi)$ for some $\sigma\in\SG_\bfa$. Hence, $\Phi_\bfa$ restricts to an isomorphism $s_\bfa{V_n}^{\otimes r}\stackrel{\sim}{\longrightarrow} U_\bfa$, so the map $\Phi:\bigoplus_{\bfa\in W_{r,k}}s_\bfa {V_n}^{\otimes r}\to \PP^r(V_{n,k})$ sending $s_\bfa(e_\bfi)$ to $\Phi_\bfa(s_\bfa(e_\bfi))$ is an isomorphism.

\begin{lemma}\label{lem:induced_isomorphism}
    The linear isomorphism $\Phi: \bigoplus_{\bfa\in W_{r,k}} s_\bfa{V_n}^{\otimes r}\to\PP^r(V_{n,k})$ above induces an isomorphism of algebras \[\End_G(\PP^r(V_{n,k}))\stackrel{\sim}{\longrightarrow}\End_G\left(\bigoplus_{\bfa\in W_{r,k}}s_\bfa{V_n}^{\otimes r}\right)\] for each subgroup $G$ of $GL_n$.
\end{lemma}

\begin{proof}
The action of $M\in GL_n$ on $s_\bfa(e_{i_1}\otimes\dots\otimes e_{i_r})$ is given by \begin{align*}
    M.s_\bfa(e_{i_1}\otimes\dots\otimes e_{i_r})&=s_\bfa(Me_{i_1}\otimes\dots\otimes M e_{i_r})\\
    &=\sum_{\ell_1,\dots,\ell_r=1}^n m_{i_1\ell_r}\cdots m_{i_r\ell_r}s_\bfa(e_{\ell_1}\otimes\dots\otimes e_{\ell_r}).
\end{align*}

Comparing this computation with Equation \ref{eq:Ua_computation}, we see that the map $\Phi$ is \textit{nearly} a homomorphism of $GL_n$-modules, but the action of $M\in GL_n$ on one space is the action of $M^{-1}$ on the other. That is, for $M\in GL_n$, \begin{align*}
        \Phi M &= M^{-1}\Phi.\\
        \intertext{Multiplying by $\Phi^{-1}$ on the left and right of both sides yields an analogous statement for $\Phi^{-1}$:}
        M\Phi^{-1}&=\Phi^{-1}M^{-1}.
    \end{align*}

    The linear isomorphism $\Phi$ induces an algebra isomorphism \begin{align*}
        \End_\C(\PP^r(V_{n,k}))&\stackrel{\sim}{\longrightarrow}\End_\C\left(\bigoplus_{\bfa\in W_{r,k}}s_\bfa{V_n}^{\otimes r}\right)\\
        \varphi&\longmapsto \Phi^{-1}\varphi\Phi.
    \end{align*}
    
     Now we make the following observation for $G\subseteq GL_n$ a subgroup. \begin{align*}
        \varphi\in\End_G(\PP^r(V_{n,k}))&\iff \varphi=M^{-1}\varphi M && \forall M\in G\\
        &\iff \Phi^{-1}\varphi\Phi=\Phi^{-1} M^{-1} \varphi M \Phi && \forall M\in G\\
        &\iff \Phi^{-1}\varphi\Phi=M\Phi^{-1}\varphi\Phi M^{-1} && \forall M\in G\\
        &\iff \Phi^{-1}\varphi\Phi\in\End_G\left(\bigoplus_{\bfa\in W_{r,k}}s_\bfa{V_n}^{\otimes r}\right)
    \end{align*}

    Hence, the map $\varphi\mapsto \Phi^{-1}\varphi\Phi$ restricts to an isomorphism\[\End_G(\PP^r{V_{n,k}}^{\otimes r})\stackrel{\sim}{\longrightarrow}\End_G\left(\bigoplus_{\bfa\in W_{r,k}}s_\bfa {V_n}^{\otimes r}\right).\]
\end{proof}

Let $G$ be a subgroup of $GL_n$. Because $G\subseteq GL_n$, we have that\[\C\SG_n\cong\End_{GL_n}({V_n}^{\otimes r})\subseteq \End_G({V_n}^{\otimes r}).\] Hence, the idempotents $s_\bfa$ for $\bfa\in W_{r,k}$ are in $\End_G({V_n}^{\otimes r})$, allowing us to use Lemma \ref{lem:decompose_by_idempotents} to make the following computation. \begin{align*}
    \End_G\left(\PP^r(V_{n,k})\right)&\cong\End_G\left(\bigoplus_{\bfa\in W_{r,k}}s_\bfa {V_n}^{\otimes r}\right)\\
    &\cong\bigoplus_{\bfa,\bfb\in W_{r,k}}s_\bfa\End_G({V_n}^{\otimes r})s_\bfb
\end{align*} where the product is given by \begin{align}
    (s_\bfa\varphi s_\bfb)\cdot(s_\bfc\psi s_\bfd)=\delta_{\bfb,\bfc}s_\bfa\varphi s_\bfb \psi s_\bfd.
\end{align}

If we write $A_r(n)=\End_G({V_n}^{\otimes r})$, this algebra is precisely the painted algebra $\tilde{A}_{r,k}(n)$ with respect to the idempotents $\{s_\bfa:\bfa\in W_{r,k}\}$. We summarize the above analysis in the following theorem.

\begin{theorem}\label{thm:isomorphic_to_painted} Let $G\subseteq GL_n$ be a subgroup and let $A_r(n)=\End_G\left({V_n}^{\otimes r}\right)$. Then\[\End_G\left(\PP^r(V_{n,k})\right)\cong \tilde{A}_{r,k}(n)\] where $\tilde{A}_{r,k}(n)$ is the painted algebra of $A_r(n)$ with respect to the idempotents $\{s_\bfa:\bfa\in W_{r,k}\}$.
\end{theorem}

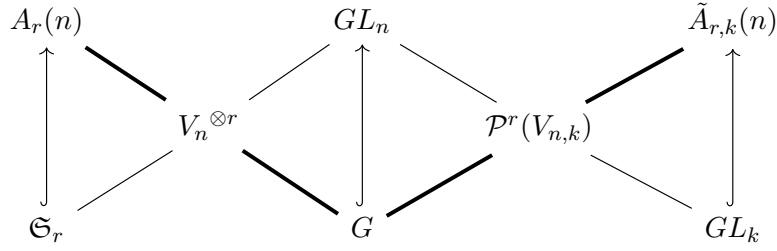
\begin{figure}[h]
    \centering
    \begin{tikzcd}
A_r(n)  \arrow[rd, no head, line width=1.5pt] &          & GL_n   \arrow[ld, no head] \arrow[rd, no head] &    & {\tilde{A}_{r,k}(n)} \arrow[ld, no head, line width=1.5pt]  \\
                       & {V_n}^{\otimes r}  &       & {\PP^r(V_{n,k})}  &    \\
\SG_r \arrow[uu, hook] \arrow[ur, no head] &      & G \arrow[uu, hook] \arrow[ul, no head, line width=1.5pt] \arrow[ur, no head, line width=1.5pt] &       & GL_k \arrow[uu, hook] \arrow[ul, no head]
\end{tikzcd}
    \caption{This diagram illustrates how Theorem \ref{thm:isomorphic_to_painted} relates the centralizers of a group $G$ under Schur-Weyl duality (left) and Howe duality (right). Note that the diagonals should only be taken to signify \textit{mutual} centralizers when the conditions of the double-centralizer theorem are satisfied (e.g. when $G$ is a linear reductive subgroup).}
    \label{fig:seesaw}
\end{figure}

Now, we use this isomorphism to construct a basis for $\MP_{r,k}(n)\cong\End_{\SG_n}(\PP^r(V_{n,k}))$ from the diagram basis of $P_r(n)$. Such a basis can be similarly constructed for other subalgebras of $P_r(n)$ that contain $\SG_r$ (see Remark \ref{rem:other_subalgebras}). In general, the projection $s_\bfa L_\pi s_\bfb$ can be computed as follows.

\begin{align*}
    s_\bfa\LL_\pi s_\bfb&=\frac{1}{\abs{\SG_\bfa\times\SG_\bfb}}\sum_{(\sigma,\sigma')\in \SG_\bfa\times\SG_\bfb} \LL_\sigma\LL_\pi\LL_{\sigma'}\\
    &=\frac{1}{\abs{\SG_\bfa\times\SG_\bfb}}\sum_{(\sigma,\sigma')\in \SG_\bfa\times\SG_\bfb} \LL_{\sigma.\pi.\sigma'}
\end{align*}

Each diagram basis element $\LL_\pi$ projects to the sum of the orbit of $\pi$ under the $\SG_\bfa\times\SG_{\bfb}$-action. Because the orbits are disjoint, the set of distinct projections are linearly independent and hence form a basis of $s_\bfa P_r(n) s_\bfb$. Due to the correspondence between orbits of set partitions under the action of a Young subgroup and multiset partitions given in Section \ref{sec:orbits}, we know that these basis elements are indexed by multiset partitions obtained by applying the map $\kappa_{\bfa,\bfb}$ to set partitions. Diagrammatically, this corresponds to coloring in the vertices of elements in $\Pi_{2(r)}$ with colors whose multiplicities are given by $\bfa$ on the top and $\bfb$ on the bottom.

For $\tpi\in\tilde\Pi_{2(r),k}$ define $\DD_\tpi=s_\bfa\LL_\pi s_\bfb$ where $\pi$ is any set partition in the orbit corresponding to $\tpi$. We then see that $\{\DD_\tpi:\tpi\in\tilde\Pi\}$ is a basis for $\tilde{P}_{r,k}(n)\cong \MP_{r,k}(n)$. Using the formula in Lemma \ref{lem:decompose_by_idempotents}, the product $\DD_\tpi\DD_\tnu$ for $\tpi\in\tilde\Pi_{2(r),k}$ with multiplicities on top and bottom given by $\bfa$ and $\bfb$ respectively and $\tnu\in\tilde\Pi_{2(r),k}$ with multiplicities on top and bottom given by $\bfb'$ and $\bfc$ respectively is the following.

\begin{align*}
    \DD_\tpi\DD_\tnu&=(s_\bfa\LL_\pi s_\bfb)\cdot(s_{\bfb'}\LL_\nu s_\bfc)\\
    &=\delta_{\bfb,\bfb'}s_\bfa\LL_\pi s_\bfb \LL_\nu s_\bfc\\
    &=\frac{\delta_{\bfb,\bfb'}}{\abs{\SG_\bfb}}\sum_{\sigma\in\SG_\bfb}s_\bfa\LL_\pi\LL_\sigma\LL_\nu s_\bfc\\
    &=\frac{\delta_{\bfb,\bfb'}}{\abs{\SG_\bfb}}\sum_{\sigma\in\SG_\bfb}s_\bfa\LL_\pi\LL_{
    \sigma.\nu} s_\bfc
\end{align*}

To interpret this product combinatorially, it will be helpful to assign a combinatorial object to each term in the sum.

\begin{defi}\label{def:snapshot}
    For a pair $\tpi,\tnu\in\tilde\Pi_{2(r),k}$, a \defn{snapshot} is a pair $(\pi,\nu)$ where $\kappa_{\bfa,\bfb}(\pi)=\tpi$ and $\kappa_{\bfb,\bfc}=\tnu$. We can represent these visually as a stack of partition diagrams whose vertices are painted from $k$ colors. To differentiate from the multiset partition diagrams (and to emphasize that the identically colored vertices in this situation are not interchangeable and are instead fixed in place), we draw the vertices as open circles.
\end{defi}

  The formula above can then be thought of as beginning with any snapshot $(\pi,\nu)$ and them summing over the snapshots $\{(\pi,\sigma.\nu):\sigma\in\SG_\bfb\}$. In the summand, $\LL_\pi\LL_{\sigma.\nu}$ is the product of the two set partitions as elements of $P_r(n)$ and multiplying by the idempotents $s_\bfa$ and $s_\bfb$ projects to the diagram-like basis element corresponding to the multiset partition obtained by filling in the vertices.

\begin{exam} The product of $\begin{array}{c}
    \begin{tikzpicture}[xscale=.35,yscale=.35,line width=1.25pt] 
        \foreach \i in {1,2,3,4}  { \path (\i,1.25) coordinate (T\i); \path (\i,.25) coordinate (B\i); } 
        \filldraw[fill= black!12,draw=black!12,line width=4pt]  (T1) -- (T4) -- (B4) -- (B1) -- (T1);
        \draw[black] (T1)--(B1)--(T2)--(T1);
        \draw[black] (B2)--(B3);
        \draw[black] (T4)--(B4);
        \colortop{1,2,3}{c1}
        \colortop{4}{c2}
        \colorbot{1,2}{c1}
        \colorbot{3,4}{c2}
    \end{tikzpicture}
    \end{array}$ and $\begin{array}{c}
    \begin{tikzpicture}[xscale=.35,yscale=.35,line width=1.25pt] 
        \foreach \i in {1,2,3,4}  { \path (\i,1.25) coordinate (T\i); \path (\i,.25) coordinate (B\i); } 
        \filldraw[fill= black!12,draw=black!12,line width=4pt]  (T1) -- (T4) -- (B4) -- (B1) -- (T1);
        \draw[black] (T1)--(B1)--(B2)--(T1);
        \draw[black] (T4)--(B4)--(B3)--(T4);
        \colortop{1,2}{c1}
        \colortop{3,4}{c2}
        \colorbot{1,2,3,4}{c2}
    \end{tikzpicture}
    \end{array}$ is given by choosing a snapshot, such as $\left(\begin{array}{c}
    \begin{tikzpicture}[xscale=.35,yscale=.35,line width=1.25pt] 
        \foreach \i in {1,2,3,4}  { \path (\i,1.25) coordinate (T\i); \path (\i,.25) coordinate (B\i); } 
        \filldraw[fill= black!12,draw=black!12,line width=4pt]  (T1) -- (T4) -- (B4) -- (B1) -- (T1);
        \draw[black] (T1)--(B1)--(T2)--(T1);
        \draw[black] (B2)--(B3);
        \draw[black] (T4)--(B4);
        \colortopfixed{1,2,3}{c1}
        \colortopfixed{4}{c2}
        \colorbotfixed{1,2}{c1}
        \colorbotfixed{3,4}{c2}
    \end{tikzpicture}
    \end{array}, \begin{array}{c}
    \begin{tikzpicture}[xscale=.35,yscale=.35,line width=1.25pt] 
        \foreach \i in {1,2,3,4}  { \path (\i,1.25) coordinate (T\i); \path (\i,.25) coordinate (B\i); } 
        \filldraw[fill= black!12,draw=black!12,line width=4pt]  (T1) -- (T4) -- (B4) -- (B1) -- (T1);
        \draw[black] (T1)--(B1)--(B2)--(T1);
        \draw[black] (T4)--(B4)--(B3)--(T4);
        \colortopfixed{1,2}{c1}
        \colortopfixed{3,4}{c2}
        \colorbotfixed{1,2,3,4}{c2}
    \end{tikzpicture}
    \end{array}\right)$, then acting on the top of the second diagram with each permutation in $\SG_{(2,2)}$.

\newcommand{\topdiagram}{\begin{tikzpicture}[xscale=.5,yscale=.5,line width=1.25pt] 
        \foreach \i in {1,2,3,4}  { \path (\i,1.25) coordinate (T\i); \path (\i,.25) coordinate (B\i); } 
        \filldraw[fill= black!12,draw=black!12,line width=4pt]  (T1) -- (T4) -- (B4) -- (B1) -- (T1);
        \draw[black] (T1)--(B1)--(T2)--(T1);
        \draw[black] (B2)--(B3);
        \draw[black] (T4)--(B4);
        \colortopfixed{1,2,3}{c1}
        \colortopfixed{4}{c2}
        \colorbotfixed{1,2}{c1}
        \colorbotfixed{3,4}{c2}
    \end{tikzpicture}}

\begin{align*}
    \frac{1}{2!2!}\left(
    \begin{array}{c}
    \topdiagram \\
    \begin{tikzpicture}[xscale=.5,yscale=.5,line width=1.25pt] 
    \foreach \i in {1,2,3,4}  { \path (\i,1.25) coordinate (T\i); \path (\i,.25) coordinate (B\i); } 
        \filldraw[fill= black!12,draw=black!12,line width=4pt]  (T1) -- (T4) -- (B4) -- (B1) -- (T1);
        \draw[black] (T1)--(B1)--(B2)--(T1);
        \draw[black] (T4)--(B4)--(B3)--(T4);
        \colortopfixed{1,2}{c1}
        \colortopfixed{3,4}{c2}
        \colorbotfixed{1,2,3,4}{c2}
    \end{tikzpicture}
    \end{array}\!+\!\begin{array}{c}
    \topdiagram \\
    \begin{tikzpicture}[xscale=.5,yscale=.5,line width=1.25pt] 
    \foreach \i in {1,2,3,4}  { \path (\i,1.25) coordinate (T\i); \path (\i,.25) coordinate (B\i); } 
        \filldraw[fill= black!12,draw=black!12,line width=4pt]  (T1) -- (T4) -- (B4) -- (B1) -- (T1);
        \draw[black] (T2)--(B1)--(B2)--(T2);
        \draw[black] (T4)--(B4)--(B3)--(T4);
        \colortopfixed{1,2}{c1}
        \colortopfixed{3,4}{c2}
        \colorbotfixed{1,2,3,4}{c2}
    \end{tikzpicture}
    \end{array}\!+\!\begin{array}{c}
    \topdiagram \\
    \begin{tikzpicture}[xscale=.5,yscale=.5,line width=1.25pt] 
    \foreach \i in {1,2,3,4}  { \path (\i,1.25) coordinate (T\i); \path (\i,.25) coordinate (B\i); } 
        \filldraw[fill= black!12,draw=black!12,line width=4pt]  (T1) -- (T4) -- (B4) -- (B1) -- (T1);
        \draw[black] (T1)--(B1)--(B2)--(T1);
        \draw[black] (T3)--(B4)--(B3)--(T3);
        \colortopfixed{1,2}{c1}
        \colortopfixed{3,4}{c2}
        \colorbotfixed{1,2,3,4}{c2}
    \end{tikzpicture}
    \end{array}
    \!+\!\begin{array}{c}
    \topdiagram \\
    \begin{tikzpicture}[xscale=.5,yscale=.5,line width=1.25pt] 
    \foreach \i in {1,2,3,4}  { \path (\i,1.25) coordinate (T\i); \path (\i,.25) coordinate (B\i); } 
        \filldraw[fill= black!12,draw=black!12,line width=4pt]  (T1) -- (T4) -- (B4) -- (B1) -- (T1);
        \draw[black] (T2)--(B1)--(B2)--(T2);
        \draw[black] (T3)--(B4)--(B3)--(T3);
        \colortopfixed{1,2}{c1}
        \colortopfixed{3,4}{c2}
        \colorbotfixed{1,2,3,4}{c2}
    \end{tikzpicture}
    \end{array}\right)\\
    =\frac{1}{4}\left(n\begin{array}{c}
    \begin{tikzpicture}[xscale=.5,yscale=.5,line width=1.25pt] 
    \foreach \i in {1,2,3,4}  { \path (\i,1.25) coordinate (T\i); \path (\i,.25) coordinate (B\i); } 
    \filldraw[fill= black!12,draw=black!12,line width=4pt]  (T1) -- (T4) -- (B4) -- (B1) -- (T1);
    \draw[black] (T1)--(T2)--(B2)--(B1)--(T1);
    \draw[black] (T4)--(B4)--(B3)--(T4);
    \colortop{1,2,3}{c1}
    \colortop{4}{c2}
    \colorbot{1,2,3,4}{c2}
    \end{tikzpicture}
    \end{array}
    \!+\!\begin{array}{c}
    \begin{tikzpicture}[xscale=.5,yscale=.5,line width=1.25pt] 
    \foreach \i in {1,2,3,4}  { \path (\i,1.25) coordinate (T\i); \path (\i,.25) coordinate (B\i); } 
    \filldraw[fill= black!12,draw=black!12,line width=4pt]  (T1) -- (T4) -- (B4) -- (B1) -- (T1);
    \draw[black] (T1)--(T2);
    \draw[black] (B1)--(B2);
    \draw[black] (T4)--(B4)--(B3)--(T4);
    \colortop{1,2,3}{c1}
    \colortop{4}{c2}
    \colorbot{1,2,3,4}{c2}
    \end{tikzpicture}
    \end{array}
    \!+\!\begin{array}{c}
    \begin{tikzpicture}[xscale=.5,yscale=.5,line width=1.25pt] 
    \foreach \i in {1,2,3,4}  { \path (\i,1.25) coordinate (T\i); \path (\i,.25) coordinate (B\i); } 
    \filldraw[fill= black!12,draw=black!12,line width=4pt]  (T1) -- (T4) -- (B4) -- (B1) -- (T1);
    \draw[black] (T1)--(T2)--(B2)--(B1)--(T1);
    \draw[black] (B3)--(B4);
    \colortop{1,2,3}{c1}
    \colortop{4}{c2}
    \colorbot{1,2,3,4}{c2}
    \end{tikzpicture}
    \end{array}
    \!+\!\begin{array}{c}
    \begin{tikzpicture}[xscale=.5,yscale=.5,line width=1.25pt] 
    \foreach \i in {1,2,3,4}  { \path (\i,1.25) coordinate (T\i); \path (\i,.25) coordinate (B\i); } 
    \filldraw[fill= black!12,draw=black!12,line width=4pt]  (T1) -- (T4) -- (B4) -- (B1) -- (T1);
    \draw[black] (T1)--(T2);
    \draw[black] (B1)--(B4);
    \colortop{1,2,3}{c1}
    \colortop{4}{c2}
    \colorbot{1,2,3,4}{c2}
    \end{tikzpicture}
    \end{array}\right)
\end{align*}

 Another snapshot we could have chosen is $\left(\begin{array}{c}
    \begin{tikzpicture}[xscale=.35,yscale=.35,line width=1.25pt] 
        \foreach \i in {1,2,3,4}  { \path (\i,1.25) coordinate (T\i); \path (\i,.25) coordinate (B\i); } 
        \filldraw[fill= black!12,draw=black!12,line width=4pt]  (T1) -- (T4) -- (B4) -- (B1) -- (T1);
        \draw[black] (T1)--(B1)--(T2)--(T1);
        \draw[black] (B2) .. controls +(0,+0.4) and +(0,+0.4) .. (B4);
        \draw[black] (T4)--(B3);
        \colortopfixed{1,2,3}{c1}
        \colortopfixed{4}{c2}
        \colorbotfixed{1,2}{c1}
        \colorbotfixed{3,4}{c2}
    \end{tikzpicture}
    \end{array}, \begin{array}{c}
    \begin{tikzpicture}[xscale=.35,yscale=.35,line width=1.25pt] 
        \foreach \i in {1,2,3,4}  { \path (\i,1.25) coordinate (T\i); \path (\i,.25) coordinate (B\i); } 
        \filldraw[fill= black!12,draw=black!12,line width=4pt]  (T1) -- (T4) -- (B4) -- (B1) -- (T1);
        \draw[black] (T2)--(B1)--(B2)--(T2);
        \draw[black] (T4)--(B4)--(B3)--(T4);
        \colortopfixed{1,2}{c1}
        \colortopfixed{3,4}{c2}
        \colorbotfixed{1,2,3,4}{c2}
    \end{tikzpicture}
    \end{array}\right)$ because if the vertices were filled in, the resulting multiset partition diagrams would still be identical to the two we are taking the product of.
\end{exam}

We call this basis $\{D_\tpi:\tpi\in\tilde{\Pi}_{2r,k}\}$ for $\MP_{r,k}(n)\cong\End_{\SG_n}(\PP^r(V_{n,k}))$ the \defn{diagram-like basis}.

\begin{rema} This perspective resolves the question of how the multiset partition algebra $\MP_\bfa(n)$ for $\bfa\in W_{r,k}$ defined in \cite{narayanan} sits inside $\MP_{r,k}(n)$:\begin{align*}
    \MP_{\bfa}(n)&\cong\End_{\SG_n}(s_\bfa {V_n}^{\otimes r})\\
    &\cong s_\bfa P_r(n) s_\bfa.
\end{align*} Note that $s_\bfa P_r(n) s_\bfa$ is the span of diagram-like basis elements in $\MP_{r,k}(n)$ whose colors on top and bottom have multiplicity given by $\bfa$.
\end{rema}

\begin{rema}\label{rem:other_subalgebras}
    Let $A_r(n)\subseteq P_r(n)$ be a subalgebra spanned by the diagram-basis elements indexed by the set partitions in some set $\Pi$. If $\Pi$ contains the set partitions corresponding to the symmetric group $\SG_r$, then the construction of the diagram-like basis above restricts to the painted version of this subalgebra. In particular, $\tilde{A}_{r,k}(n)$ has basis given by the diagram-like basis elements indexed by $\biguplus_{\bfa,\bfb\in W_{r,k}}\kappa_{\bfa,\bfb}(\Pi)$.

    For a concrete example, the Brauer algebra $B_r(n)$ is the subalgebra of $P_r(n)$ spanned by the set partition diagrams whose blocks are all of size two. The corresponding painted Brauer algebra $\tilde{B}_{r,k}(n)$ has a basis indexed by the multiset partition diagrams whose blocks are all of size two.

    Note that planar subalgebras of $P_r(n)$ (such as the Temperley-Lieb algebra) do not contain the symmetric group, and hence this painted algebra construction does not extend nicely to them. It may be the case that there are painted analogues of these planar algebras which centralize a group action on $\PP^r(V_{n,k})$, but they do not seem amenable to description with the methods of this paper.
\end{rema}

\section{Irreducible Representations of \texorpdfstring{$\MP_{r,k}(n)$}{MP_r,k(n)}}\label{sec:reps}

In this section, we aim to construct the irreducible representations of $\MP_{r,k}(n)$. Part (2) of Lemma \ref{lem:symmetrized_modules} tells us that in order to construct each of these irreducible representations, we need only consider the irreducible $P_r(n)$ representations painted with respect to the idempotents $\{s_\bfa:\bfa\in W_{r,k}\}$. For $\lambda\in\Lambda^{P_r(n)}$, define \[\MP_{r,k}^\lambda\vcentcolon= \tilde{P_r^\lambda}= \bigoplus_{\bfa\in W_{r,k}}s_\bfa P_r^\lambda.\]

By Lemma \ref{lem:symmetrized_modules}, each module in $\{\MP_{r,k}^\lambda:\lambda\in\Lambda^{P_r(n)}\}$ is either a simple $\MP_{r,k}(n)$-module or the zero module, and each simple $\MP_{r,k}(n)$ module appears in the set. To investigate the structure of these modules, we note that for $T\in\SPT_{\lambda,r}$, the projection $s_\bfa v_T$ is the average over the $\SG_\bfa$-orbit of $T$:

\[s_\bfa v_T=\frac{1}{\abs{\SG_\bfa.T}}\sum_{S\in\SG_\bfa.T}v_S\] This orbit corresponds to $\tT=\kappa_\bfa(T)\in\MPT_{\lambda,r,k}$, and so we define \[w_\tT=s_\bfa v_T=\frac{1}{\abs{\kappa_\bfa^{-1}(\tT)}}\sum_{T\in\kappa_\bfa^{-1}(\tT)} v_T\] where $\kappa_\bfa^{-1}(\tT)$ is the preimage of $\tT$ under the coloring map.

\begin{lemma}\label{lem:painting_standard} If $T\in\SSPT_{\lambda,r}$, then either $w_\tT=0$ or $\tT\in\SSMPT_{\lambda,r,k}$.
\end{lemma}

\begin{proof}
    As observed in Remark \ref{rmk:semistandard_tableaux}, if $T\in\SSPT_{\lambda,r}$, then $\tT$ has rows and columns weakly increasing. If $\tT$ is not semistandard, then it must have two boxes within the same column that have identical contents. For a tableau $T$, write $T'$ for the tableau obtained by swapping the content of these two boxes. Then \[w_{\tT}=w_{\tilde{T'}}=\frac{1}{\abs{\SG_\bfa.T}}\sum_{S\in\SG_\bfa.T}v_{S'}=\frac{1}{\abs{\SG_\bfa.T}}\sum_{S\in\SG_\bfa.T}-v_{S}=-w_{\tT}.\] Hence, if $\tT\notin\SSMPT_{\lambda,r,k}$, then $w_\tT=0$.
\end{proof}

We can now use Lemma \ref{lem:painting_standard} to describe a straightening algorithm for $\MP_{r,k}^\lambda$. Suppose $\tT$ is not semistandard and $w_\tT\neq0$. Then there exists a nonstandard $T$ in the $\SG_\bfa$-orbit corresponding to $\tT$. Then using the straightening algorithm for $P_r^\lambda$, we can write \[v_T=\sum_{S\in\SSPT_{\lambda,r}}c_Sv_S.\]
Then projecting by $s_\bfa$, we obtain \[w_\tT=s_\bfa v_T=\sum_{S\in\SSPT_{\lambda,r}} c_S s_\bfa v_S=\sum_{S\in\SSPT_{\lambda,r}} c_S w_\tS\] where each $\tS$ for which $w_{\tS}\neq0$ is semistandard.

\begin{theorem} The set $\{\MP_{r,k}^\lambda:\lambda\in\Lambda^{\MP_{r,k}(n)}\}$ forms a complete set of irreducible representations for $\MP_{r,k}(n)$ and for each $\lambda\in\Lambda^{\MP_{r,k}(n)}$, the set $\{w_\tT:\tT\in\SSMPT_{\lambda,r,k}\}$ forms a basis of $\MP_{r,k}^\lambda$.
\end{theorem}

\begin{proof} 
    Because $\MP_{r,k}^\lambda$ is the span of $\{w_\tT:\tT\in\SSMPT_{\lambda,r,k}\}$, we know that $\MP_{r,k}^\lambda=0$ unless $\lambda\in\Lambda^{\MP_{r,k}(n)}$. We then have that each of the $\abs{\Lambda^{\MP_{r,k}(n)}}$ irreducible representations appears in the smaller set $\{\MP_{r,k}^\lambda:\lambda\in\Lambda^{\MP_{r,k}(n)}\}$. Because there are only $\abs{\Lambda^{\MP_{r,k}(n)}}$ representations in this set, it must be a complete set of irreducible representations for $\MP_{r,k}(n)$.

    A priori, we do not know that these $w_\tT$ are linearly independent, so we can only conclude that $\dim(\MP_{r,k}^\lambda)\leq\#\SSMPT_{\lambda,r,k}$. However, we do know that
\[\sum_{\lambda\in\Lambda^{\MP_{r,k}(n)}}(\dim(\MP_{r,k}^\lambda))^2=\dim(\MP_{r,k}(n))=\sum_{\lambda\in\Lambda^{\MP_{r,k}(n)}}(\#\SSMPT_{\lambda,r,k})^2.\] Hence $\dim(\MP_{r,k}^\lambda)=\#\SSMPT_{\lambda,r,k}$ and so for each $\lambda\in\Lambda^{\MP_{r,k}(n)}$ the set $\{w_\tT:\tT\in\SSMPT_{\lambda,r,k}\}$ indeed forms a basis of $\MP_{r,k}^\lambda$.
\end{proof}

We now consider the action of an element $\DD_\tpi$ on $w_\tT$. Suppose that the multiplicities of colors in $\tpi$ are given by $\bfa$ and $\bfb$ respectively and that the multiplicities of elements in $\tT$ are given by $\bfc$. Then the definition of a painted module gives us the following formula: \begin{align*}
    \DD_\tpi.w_\tT&=s_\bfa L_\pi s_\bfb . s_\bfc v_T\\
                  &=\delta_{\bfb,\bfc}(s_\bfa L_\pi s_\bfb v_T)\\
                  &=\delta_{\bfb,\bfc}\sum_{\sigma\in\SG_\bfb}s_\bfa L_{\pi.\sigma} v_T
\end{align*}

We can interpret this formula for diagrams as follows.
\begin{enumerate}
    \item[(i)] Pull out the content of $\tT$, a multiset partition with $r$ elements from $[k]$, in a row above and fix the order.
    \item[(ii)] Place $\tpi$ on top and permute the vertices of the same color at the bottom in each possible way.
    \item[(iii)] For each permutation, compute the action as for $P_r^\lambda$.
    \item[(iv)] Sum the resulting tableaux and divide by the number of permutations.
\end{enumerate}

\begin{exam} The action of a multiset partition on a multiset partition tableau. Note that in the latter two permutations of the diagram, the content of the box in the second row does not reach the top of the diagram, so the output is zero.
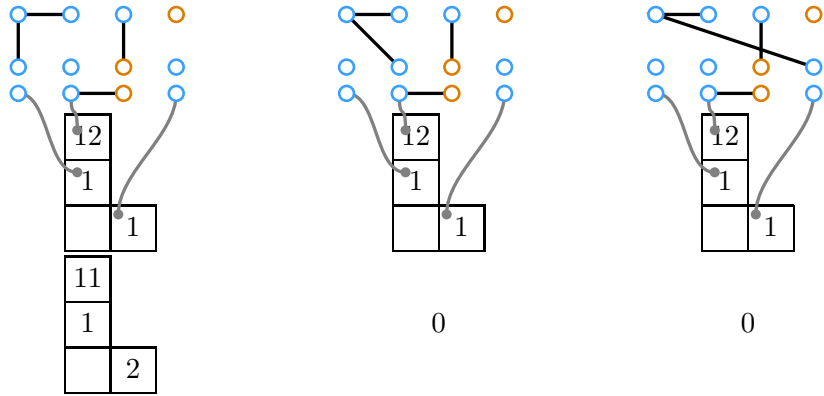
\begin{align*}
    \begin{tabular}{c@{\hskip .8in}c@{\hskip .7in}c}
         \begin{tikzpicture}[line width=1.25pt, xscale=.7, yscale=.7]
            \path(1,1) node {$\inline{\begin{ytableau}
        12 \\
        1 \\
        \; & 1
    \end{ytableau}}$};
            \path(1,-1.7) node {$\inline{\begin{ytableau}
        11 \\
        1 \\
        \; & 2
    \end{ytableau}}$};
            \path (.37,1.2) coordinate (T21);
            \path (.37,2) coordinate (T31);
            \path (1.15,.4) coordinate (T12);
            \path (-.75,2.7) coordinate (C1);
            \path (.25,2.7) coordinate (C2);
            \path (1.25,2.7) coordinate (C3);
            \path (2.25,2.7) coordinate (C4);
            \path (-.75,3.2) coordinate (B1);
            \path (.25,3.2) coordinate (B2);
            \path (1.25,3.2) coordinate (B3);
            \path (2.25,3.2) coordinate (B4);
            \path (-.75,4.2) coordinate (T1);
            \path (.25,4.2) coordinate (T2);
            \path (1.25,4.2) coordinate (T3);
            \path (2.25,4.2) coordinate (T4);
            \draw[gray] (T21) .. controls +(-.6,0) and +(+.6,0) .. (C1);
            \draw[gray] (T31) .. controls +(0,+.6) and +(0,-.6) .. (C2);
            \draw[gray] (T12) .. controls +(0,+.8) and +(0,-.8) .. (C4);
            \draw[black] (C2) -- (C3);
            \draw[black] (T2) -- (T1) -- (B1);
            \draw[black] (T3) -- (B3);
            \filldraw[fill=gray,draw=gray,line width = 1pt] (T21) circle (2pt);
            \filldraw[fill=gray,draw=gray,line width = 1pt] (T31) circle (2pt);
            \filldraw[fill=gray,draw=gray,line width = 1pt] (T12) circle (2pt);
            \foreach \i in {C1, C2, C4, B1, B2, B4, T1, T2, T3} {\filldraw[fill=white,draw=c1,line width = 1pt] (\i) circle (4pt);}
            \foreach \i in {C3, B3, T4} {\filldraw[fill=white,draw=c2,line width = 1pt] (\i) circle (4pt);}
         \end{tikzpicture} &
         \begin{tikzpicture}[line width=1.25pt, xscale=.7, yscale=.7]
            \path(1,1) node {$\inline{\begin{ytableau}
        12 \\
        1 \\
        \; & 1
    \end{ytableau}}$};
            \path(1,-1.7) node {$0\vphantom{\inline{\begin{ytableau}
        12 \\
        1 \\
        \; & 1
    \end{ytableau}}}$};
            \path (.37,1.2) coordinate (T21);
            \path (.37,2) coordinate (T31);
            \path (1.15,.4) coordinate (T12);
            \path (-.75,2.7) coordinate (C1);
            \path (.25,2.7) coordinate (C2);
            \path (1.25,2.7) coordinate (C3);
            \path (2.25,2.7) coordinate (C4);
            \path (-.75,3.2) coordinate (B1);
            \path (.25,3.2) coordinate (B2);
            \path (1.25,3.2) coordinate (B3);
            \path (2.25,3.2) coordinate (B4);
            \path (-.75,4.2) coordinate (T1);
            \path (.25,4.2) coordinate (T2);
            \path (1.25,4.2) coordinate (T3);
            \path (2.25,4.2) coordinate (T4);
            \draw[gray] (T21) .. controls +(-.6,0) and +(+.6,0) .. (C1);
            \draw[gray] (T31) .. controls +(0,+.6) and +(0,-.6) .. (C2);
            \draw[gray] (T12) .. controls +(0,+.8) and +(0,-.8) .. (C4);
            \draw[black] (C2) -- (C3);
            \draw[black] (T2) -- (T1) -- (B2);
            \draw[black] (T3) -- (B3);
            \filldraw[fill=gray,draw=gray,line width = 1pt] (T21) circle (2pt);
            \filldraw[fill=gray,draw=gray,line width = 1pt] (T31) circle (2pt);
            \filldraw[fill=gray,draw=gray,line width = 1pt] (T12) circle (2pt);
            \foreach \i in {C1, C2, C4, B1, B2, B4, T1, T2, T3} {\filldraw[fill=white,draw=c1,line width = 1pt] (\i) circle (4pt);}
            \foreach \i in {C3, B3, T4} {\filldraw[fill=white,draw=c2,line width = 1pt] (\i) circle (4pt);}
         \end{tikzpicture} &
         \begin{tikzpicture}[line width=1.25pt,xscale=.7,yscale=.7]
            \path(1,1) node {$\inline{\begin{ytableau}
        12 \\
        1 \\
        \; & 1
    \end{ytableau}}$};
            \path(1,-1.7) node {$0\vphantom{\inline{\begin{ytableau}
        12 \\
        1 \\
        \; & 1
    \end{ytableau}}}$};
            \path (.37,1.2) coordinate (T21);
            \path (.37,2) coordinate (T31);
            \path (1.15,.4) coordinate (T12);
            \path (-.75,2.7) coordinate (C1);
            \path (.25,2.7) coordinate (C2);
            \path (1.25,2.7) coordinate (C3);
            \path (2.25,2.7) coordinate (C4);
            \path (-.75,3.2) coordinate (B1);
            \path (.25,3.2) coordinate (B2);
            \path (1.25,3.2) coordinate (B3);
            \path (2.25,3.2) coordinate (B4);
            \path (-.75,4.2) coordinate (T1);
            \path (.25,4.2) coordinate (T2);
            \path (1.25,4.2) coordinate (T3);
            \path (2.25,4.2) coordinate (T4);
            \draw[gray] (T21) .. controls +(-.6,0) and +(+.6,0) .. (C1);
            \draw[gray] (T31) .. controls +(0,+.6) and +(0,-.6) .. (C2);
            \draw[gray] (T12) .. controls +(0,+.8) and +(0,-.8) .. (C4);
            \draw[black] (C2) -- (C3);
            \draw[black] (T2) -- (T1) -- (B4);
            \draw[black] (T3) -- (B3);
            \filldraw[fill=gray,draw=gray,line width = 1pt] (T21) circle (2pt);
            \filldraw[fill=gray,draw=gray,line width = 1pt] (T31) circle (2pt);
            \filldraw[fill=gray,draw=gray,line width = 1pt] (T12) circle (2pt);
            \foreach \i in {C1, C2, C4, B1, B2, B4, T1, T2, T3} {\filldraw[fill=white,draw=c1,line width = 1pt] (\i) circle (4pt);}
            \foreach \i in {C3, B3, T4} {\filldraw[fill=white,draw=c2,line width = 1pt] (\i) circle (4pt);}
         \end{tikzpicture}
    \end{tabular}
\end{align*}
\[\tpi.\, \inline{\begin{ytableau}
        12 \\
        1 \\
        \; & 1
    \end{ytableau}}=\frac{1}{3}\,\inline{\begin{ytableau}
        1 \\
        11 \\
        \; & 2
    \end{ytableau}}=\frac{-1}{3}\inline{\begin{ytableau}
        11 \\
        1 \\
        \; & 2
    \end{ytableau}}\]
\end{exam}

\section{Generators}\label{sec:generators}

To describe a generating set for $\MP_{r,k}(n)$, we will give an algorithm for factoring out certain blocks from a diagram. 

\begin{defi}
    We call a block of the form $\multi{i,\ov i}$ a \defn{vertical bar}. A block of a multiset partition $\tpi$ that is not a vertical bar or a singleton is called a \defn{nonbasic block}. For a set partition $\pi$ with $\kappa_{\bfa,\bfb}(\pi)=\tpi$, we call a vertex in the diagram of $\pi$ \defn{nonbasic} if it is mapped to a nonbasic block under $\kappa_{\bfa,\bfb}$.
\end{defi}

We now define a statistic on multiset partitions and prove a lemma about how this statistic interacts with the diagram-like product.

\begin{defi}
    Write $N(\tpi)$ for the multiset of nonbasic blocks of $\tpi$ and define the \defn{nonbasic weight} of $\tpi$ to be \[\nbw(\tpi)=\sum_{\tB\in N(\tpi)}\abs{\tB}.\]
\end{defi}

\begin{exam} We compute the nonbasic weight of some multiset partitions. The nonbasic vertices have their vertices highlighted in green.
\begin{align*}
    \nbw\left(\begin{array}{c}\begin{tikzpicture}[xscale=.7,yscale=.7,line width=1.4pt] 
        \foreach \i in {1,2,3,4,5,6}  { \path (\i,1.25) coordinate (T\i); \path (\i,.25) coordinate (B\i);} 
        \filldraw[fill= black!12,draw=black!12,line width=4pt]  (T1) -- (T6) -- (B6) -- (B1) -- (T1);
        \draw[black] (T1)--(B1);
        \draw[black] (T2)--(T3)--(B3)--(B2)--(T2);
        \draw[black] (T4)--(T6)--(B6)--(T4);
        \draw[black] (B4)--(B5);
        \foreach \l in {T2,T3,T4,T5,T6,B2,B3,B4,B5,B6} {\draw[c3] (\l) circle (7pt);}
        \foreach \l in {T1,T2,T3,B1} {\filldraw[fill=white,draw=c1,line width = 1pt] (\l) circle (4pt);}
        \foreach \l in {T4,T5,T6,B2,B3,B4,B5,B6} {\filldraw[fill=white,draw=c2,line width = 1pt] (\l) circle (4pt);}
    \end{tikzpicture}\end{array}\right)&=2+4+4=10\\
    \nbw\left(\begin{array}{c}\begin{tikzpicture}[xscale=.7,yscale=.7,line width=1.4pt] 
        \foreach \i in {1,2,3,4,5,6}  { \path (\i,1.25) coordinate (T\i); \path (\i,.25) coordinate (B\i);} 
        \filldraw[fill= black!12,draw=black!12,line width=4pt]  (T1) -- (T6) -- (B6) -- (B1) -- (T1);
        \draw[black] (T1);
        \draw[black] (T2)--(B1);
        \draw[black] (B2)--(T3)--(B3)--(B2);
        \draw[black] (T4)--(T5);
        \draw[black] (B4) --(B6) -- (T6) -- (B4);
        \foreach \l in {T3,T4,T5,T6,B2,B3,B4,B5,B6} {\draw[c3] (\l) circle (7pt);}
        \foreach \l in {T1,T2,T3,B1} {\filldraw[fill=white,draw=c1,line width = 1pt] (\l) circle (4pt);}
        \foreach \l in {T4,T5,T6,B2,B3,B4,B5,B6} {\filldraw[fill=white,draw=c2,line width = 1pt] (\l) circle (4pt);}
    \end{tikzpicture}\end{array}\right)&=2+3+4=9
\end{align*}
\end{exam}

\begin{lemma} \label{lem:nonbasic_weight_and_product}
    If $\DD_\tnu$ appears with nonzero coefficient in the product $\DD_{\tpi_1}\DD_{\tpi_2}$, then \[\nbw(\tnu)\leq \nbw(\tpi_1)+\nbw(\tpi_2)\] with equality if and only if $N(\tnu)=N(\tpi_1)\uplus N(\tpi_2)$.
\end{lemma}

\begin{proof}
    Consider a snapshot $(\pi_1,\pi_2)$ in the product $\DD_{\tpi_1}\DD_{\tpi_2}$ and suppose $\LL_{\pi_1}\LL_{\pi_2}=n^c\LL_\nu$. Suppose $\tpi_1=\kappa_{\bfa,\bfb}(\pi_1)$ and $\tpi_2=\kappa_{\bfb,\bfc}(\pi_2)$. Write $\tnu=\kappa_{\bfa,\bfc}(\nu)$. Our goal is now to construct an injective map $\varphi$ from the set of nonbasic vertices of $\nu$ to the nonbasic vertices of $\pi_1$ and $\pi_2$. 
    
    Consider a nonbasic vertex $v$ of $\nu$ labeled $i$. There is a corresponding vertex $v'$ of $\pi_1$ also labeled $i$. If $v'$ is the only element in its block, the same is true for $v$. Hence, $v'$ must either be a nonbasic vertex or one end of a vertical bar. In the first case, set $\varphi(v)$ equal to $v'$, which is nonbasic. In the latter case, let $\ov j$ be the label of the other vertex in the vertical bar and let $\varphi(v)$ be the vertex of $\pi_2$ labeled $j$. If the diagram of $\pi_1$ were set atop that of $\pi_2$, this would be the vertex on the top of $\pi_2$ that the vertical bar lands on. If $v$ is labeled $\ov i$, the same process is followed swapping which elements are barred (see Figure \ref{fig:injective_map_on_nonbasics} for an illustration of this map).

    \begin{figure}[h]
        \centering
        $\begin{array}{cc}
    \begin{tikzpicture}[xscale=.7,yscale=.7,line width=1.4pt] 
        \foreach \i in {1,2,3,4,5,6}  { \path (\i,1.25) coordinate (T\i); \path (\i,.25) coordinate (B\i); \path (\i,-.35) coordinate (TP\i); \path (\i,-1.35) coordinate (BP\i); } 
        \filldraw[fill= black!12,draw=black!12,line width=4pt]  (T1) -- (T6) -- (B6) -- (B1) -- (T1);
        \filldraw[fill= black!12,draw=black!12,line width=4pt]  (TP1) -- (TP6) -- (BP6) -- (BP1) -- (TP1);
        \draw[black] (T1)--(B1)--(B2)--(T1);
        \draw[black] (T2)--(T3)--(B3)--(T2);
        \draw[black] (T4)--(T6)--(B6)--(T4);
        \draw[black] (TP1)--(TP2)--(BP1)--(TP1);
        \draw[black] (TP3)--(BP3)--(BP2)--(TP3);
        \draw[black] (TP4)--(TP5)--(BP5)--(BP4)--(TP4);
        \draw[black] (TP6)--(BP6);
        \foreach \l in {T2,T3,T4,T5,T6,B6,BP2,BP3,BP4,BP5} {\draw[c3] (\l) circle (7pt);}
        \foreach \l in {T1,T2,T3,B1,B2,TP1,TP2,BP1} {\filldraw[fill=white,draw=c1,line width = 1pt] (\l) circle (4pt);}
        \foreach \l in {T3,T4,T5,T6,B3,B4,B5,B6,TP3,TP4,TP5,TP6,BP2,BP3,BP4,BP5,BP6} {\filldraw[fill=white,draw=c2,line width = 1pt] (\l) circle (4pt);}
        \end{tikzpicture} & \raisebox{2.7em}{$\rightarrow$}\,\,\raisebox{1.6em}{\begin{tikzpicture}[xscale=.7,yscale=.7,line width=1.4pt] 
        \foreach \i in {1,2,3,4,5,6}  { \path (\i,1.25) coordinate (T\i); \path (\i,.25) coordinate (B\i);} 
        \filldraw[fill= black!12,draw=black!12,line width=4pt]  (T1) -- (T6) -- (B6) -- (B1) -- (T1);
        \draw[black] (T1)--(B1);
        \draw[black] (T2)--(T3)--(B3)--(B2)--(T2);
        \draw[black] (T4)--(T6)--(B6)--(T4);
        \draw[black] (B4)--(B5);
        \foreach \l in {T2,T3,T4,T5,T6,B2,B3,B4,B5,B6} {\draw[c3] (\l) circle (7pt);}
        \foreach \l in {T1,T2,B1} {\filldraw[fill=white,draw=c1,line width = 1pt] (\l) circle (4pt);}
        \foreach \l in {T3,T4,T5,T6,B2,B3,B4,B5,B6} {\filldraw[fill=white,draw=c2,line width = 1pt] (\l) circle (4pt);}
    \end{tikzpicture}}
\end{array}$
        \caption{Illustration of the injective map $\varphi$ constructed in Lemma \ref{lem:nonbasic_weight_and_product}. The nonbasic vertices on the right and the corresponding nonbasic vertices in the image of $\varphi$ on the left are highlighted in green.}
        \label{fig:injective_map_on_nonbasics}
    \end{figure}
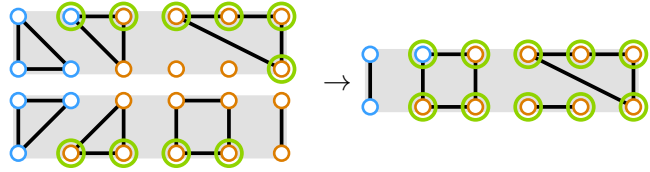

    In the latter case, if $\varphi(v)$ is basic, it is either the only element of its block (in which case, $v$ would be a singleton) or part of a vertical bar (in which case, $v$ is in a vertical bar). Either way, this contradicts the assumption that $\nu$ is nonbasic, so we have constructed a map $\varphi$ from the set of nonbasic vertices of $\nu$ to the nonbasic vertices of $\pi_1$ and $\pi_2$. It is clear that this map is injective, and so the number of nonbasic vertices of $\nu$ is less than the total number of nonbasic vertices of 
    $\pi_1$ and $\pi_2$, giving us the desired inequality.

    Now to investigate the case of equality we consider how the map $\varphi$ interacts with the set partition structure. Suppose that $\varphi(v)$ and $\varphi(w)$ are in the same block. Without loss of generality, assume they are in the same block of $\pi_1$. We then need to consider the following cases: (see Figure \ref{fig:generator_lemma_cases}(i)-(iii) for illustrations of these cases)

    \begin{enumerate}
        \item[(i)] $\varphi(v)$ and $\varphi(w)$ are both on the top of the block.

        The vertex $\varphi(v)$ has the same label as $v$ and $\varphi(w)$ has the same label as $w$. Because the vertices with these labels are connected, the vertices with the same labels must be connected in the product, so $v$ and $w$ are in the same block.

        \item[(ii)] $\varphi(v)$ and $\varphi(w)$ are both on the bottom of the block.

        The vertices $\varphi(v)$ and $\varphi(w)$ each meet a vertical bar whose other end is labeled the same as $v$ and $w$ respectively. Hence, $v$ and $w$ are joined in the product.

        \item[(iii)] Without loss of generality $\varphi(v)$ is on the top of the block and $\varphi(w)$ is on the bottom.

        The vertex $\varphi(v)$ is labeled the same as $v$ and the vertex on the other end of the vertical bar meeting $\varphi(w)$ is labeled the same as $w$. Hence $v$ and $w$ are again joined in the product.
    \end{enumerate}

    Hence, if $\varphi(v)$ and $\varphi(w)$ are in the same block, then $u$ and $v$ are in the same block.

    Suppose that $v$ and $w$ are in the same block but $\varphi(v)$ and $\varphi(w)$ are not (see Figure \ref{fig:generator_lemma_cases}(iv)-(v)). Then two nonbasic blocks in the product must have been combined, and any vertex where the two nonbasic blocks meet must not be in the image of $\varphi$, meaning $\varphi$ is not a surjection in this case.

    When equality holds, the map $\varphi$ is a bijection and $\varphi(v)$ in the same block as $\varphi(w)$ if and only if $v$ and $w$ are in the same block. The map $\varphi$ then induces a bijection of nonbasic \emph{blocks}, hence $N(\tnu)=N(\tpi_1)\uplus N(\tpi_2)$.
\end{proof}

\begin{figure}
    \centering
    \begin{tabular}{cc}
        (i) $\begin{array}{cc}
    \begin{tikzpicture}[xscale=.5,yscale=.5,line width=1.25pt] 
        \foreach \i in {1,2,3}  { \path (\i,1.25) coordinate (T\i); \path (\i,.25) coordinate (B\i); \path (\i,-.35) coordinate (TP\i); \path (\i,-1.35) coordinate (BP\i); } 
        \filldraw[fill= black!12,draw=black!12,line width=4pt]  (T1) -- (T3) -- (B3) -- (B1) -- (T1);
        \filldraw[fill= black!12,draw=black!12,line width=4pt]  (TP1) -- (TP3) -- (BP3) -- (BP1) -- (TP1);
        \draw[black] (T1)--(T3)--(B2)--(B1)--(T1);
        \draw[black] (TP1)--(TP2)--(BP1)--(TP1);
        \draw[black] (TP3)--(BP3)--(BP2)--(TP3);
        \draw[c3] (T1) circle (7pt);
        \draw[c3] (T2) circle (7pt);
        \foreach \l in {T1,T2,B1,B2,B3,TP1,TP2,TP3,BP1,BP2,BP3} {\filldraw[fill=white,draw=c1,line width = 1pt] (\l) circle (4pt);}
        \foreach \l in {T3} {\filldraw[fill=white,draw=c2,line width = 1pt] (\l) circle (4pt);}
        \end{tikzpicture} & \raisebox{1.5em}{$\rightarrow$}\,\,\raisebox{1em}{\begin{tikzpicture}[xscale=.5,yscale=.5,line width=1.25pt] 
        \foreach \i in {1,2,3}  { \path (\i,1.25) coordinate (T\i); \path (\i,.25) coordinate (B\i);} 
        \filldraw[fill= black!12,draw=black!12,line width=4pt]  (T1) -- (T3) -- (B3) -- (B1) -- (T1);
        \draw[black] (T1)--(T3)--(B1)--(T1);
        \draw[black] (B2)--(B3);
        \draw[c3] (T1) circle (7pt);
        \draw[c3] (T2) circle (7pt);
        \foreach \l in {T1,T2,B1,B2,B3} {\filldraw[fill=white,draw=c1,line width = 1pt] (\l) circle (4pt);}
        \foreach \l in {T3} {\filldraw[fill=white,draw=c2,line width = 1pt] (\l) circle (4pt);}
    \end{tikzpicture}}
\end{array}$
    & (ii) $\begin{array}{cc}
    \begin{tikzpicture}[xscale=.5,yscale=.5,line width=1.25pt] 
        \foreach \i in {1,2,3}  { \path (\i,1.25) coordinate (T\i); \path (\i,.25) coordinate (B\i); \path (\i,-.35) coordinate (TP\i); \path (\i,-1.35) coordinate (BP\i); } 
        \filldraw[fill= black!12,draw=black!12,line width=4pt]  (T1) -- (T3) -- (B3) -- (B1) -- (T1);
        \filldraw[fill= black!12,draw=black!12,line width=4pt]  (TP1) -- (TP3) -- (BP3) -- (BP1) -- (TP1);
        \draw[black] (T1)--(T2)--(B2)--(B1)--(T1);
        \draw[black] (T3)--(B3);
        \draw[black] (TP1)--(BP1);
        \draw[black] (TP2)--(BP2);
        \draw[black] (TP3)--(BP3);
        \draw[c3] (B1) circle (7pt);
        \draw[c3] (B2) circle (7pt);
        \foreach \l in {T1,T2,B1,B2,B3,TP1,TP2,TP3,BP1,BP2,BP3} {\filldraw[fill=white,draw=c1,line width = 1pt] (\l) circle (4pt);}
        \foreach \l in {T3} {\filldraw[fill=white,draw=c2,line width = 1pt] (\l) circle (4pt);}
        \end{tikzpicture} & \raisebox{1.5em}{$\rightarrow$}\,\,\raisebox{1em}{\begin{tikzpicture}[xscale=.5,yscale=.5,line width=1.25pt] 
        \foreach \i in {1,2,3}  { \path (\i,1.25) coordinate (T\i); \path (\i,.25) coordinate (B\i);} 
        \filldraw[fill= black!12,draw=black!12,line width=4pt]  (T1) -- (T3) -- (B3) -- (B1) -- (T1);
        \draw[black] (T1)--(T2)--(B2)--(B1)--(T1);
        \draw[black] (T3)--(B3);
        \draw[c3] (B1) circle (7pt);
        \draw[c3] (B2) circle (7pt);
        \foreach \l in {T1,T2,B1,B2,B3} {\filldraw[fill=white,draw=c1,line width = 1pt] (\l) circle (4pt);}
        \foreach \l in {T3} {\filldraw[fill=white,draw=c2,line width = 1pt] (\l) circle (4pt);}
    \end{tikzpicture}}
\end{array}$ \\
        (iii) $\begin{array}{cc}
    \begin{tikzpicture}[xscale=.5,yscale=.5,line width=1.25pt] 
        \foreach \i in {1,2,3}  { \path (\i,1.25) coordinate (T\i); \path (\i,.25) coordinate (B\i); \path (\i,-.35) coordinate (TP\i); \path (\i,-1.35) coordinate (BP\i); } 
        \filldraw[fill= black!12,draw=black!12,line width=4pt]  (T1) -- (T3) -- (B3) -- (B1) -- (T1);
        \filldraw[fill= black!12,draw=black!12,line width=4pt]  (TP1) -- (TP3) -- (BP3) -- (BP1) -- (TP1);
        \draw[black] (T1)--(T2)--(B2)--(B1)--(T1);
        \draw[black] (T3)--(B3);
        \draw[black] (TP1)--(BP1);
        \draw[black] (TP2)--(TP3)--(BP3)--(TP2);
        \draw[c3] (T1) circle (7pt);
        \draw[c3] (B1) circle (7pt);
        \foreach \l in {T1,T2,B1,B2,B3,TP1,TP2,TP3,BP1,BP2,BP3} {\filldraw[fill=white,draw=c1,line width = 1pt] (\l) circle (4pt);}
        \foreach \l in {T3} {\filldraw[fill=white,draw=c2,line width = 1pt] (\l) circle (4pt);}
        \end{tikzpicture} & \raisebox{1.5em}{$\rightarrow$}\,\,\raisebox{1em}{\begin{tikzpicture}[xscale=.5,yscale=.5,line width=1.25pt] 
        \foreach \i in {1,2,3}  { \path (\i,1.25) coordinate (T\i); \path (\i,.25) coordinate (B\i);} 
        \filldraw[fill= black!12,draw=black!12,line width=4pt]  (T1) -- (T3) -- (B3) -- (B1) -- (T1);
        \draw[black] (T1)--(T3)--(B3) .. controls +(0,+.35) and +(0,+.35) .. (B1)--(T1);
        \draw[c3] (T1) circle (7pt);
        \draw[c3] (B1) circle (7pt);
        \foreach \l in {T1,T2,B1,B2,B3} {\filldraw[fill=white,draw=c1,line width = 1pt] (\l) circle (4pt);}
        \foreach \l in {T3} {\filldraw[fill=white,draw=c2,line width = 1pt] (\l) circle (4pt);}
    \end{tikzpicture}}
\end{array}$
& (iv) $\begin{array}{cc}
    \begin{tikzpicture}[xscale=.5,yscale=.5,line width=1.25pt] 
        \foreach \i in {1,2,3}  { \path (\i,1.25) coordinate (T\i); \path (\i,.25) coordinate (B\i); \path (\i,-.35) coordinate (TP\i); \path (\i,-1.35) coordinate (BP\i); } 
        \filldraw[fill= black!12,draw=black!12,line width=4pt]  (T1) -- (T3) -- (B3) -- (B1) -- (T1);
        \filldraw[fill= black!12,draw=black!12,line width=4pt]  (TP1) -- (TP3) -- (BP3) -- (BP1) -- (TP1);
        \draw[black] (T1)--(B1);
        \draw[black] (T2)--(T3)--(B2)--(T2);
        \draw[black] (TP1)--(TP2)--(BP2)--(BP1)--(TP1);
        \draw[black] (TP3)--(BP3);
        \draw[c4] (TP1) circle (7pt);
        \draw[c4] (T2) circle (7pt);
        \foreach \l in {T1,T2,B1,B2,B3,TP1,TP2,TP3,BP1,BP2,BP3} {\filldraw[fill=white,draw=c1,line width = 1pt] (\l) circle (4pt);}
        \foreach \l in {T3} {\filldraw[fill=white,draw=c2,line width = 1pt] (\l) circle (4pt);}
        \end{tikzpicture} & \raisebox{1.5em}{$\rightarrow$}\,\,\raisebox{1em}{\begin{tikzpicture}[xscale=.5,yscale=.5,line width=1.25pt] 
        \foreach \i in {1,2,3}  { \path (\i,1.25) coordinate (T\i); \path (\i,.25) coordinate (B\i);} 
        \filldraw[fill= black!12,draw=black!12,line width=4pt]  (T1) -- (T3) -- (B3) -- (B1) -- (T1);
        \draw[black] (T1)--(T3)--(B2)--(B1)--(T1);
        \draw[c4] (T1) circle (7pt);
        \draw[c4] (T2) circle (7pt);
        \foreach \l in {T1,T2,B1,B2,B3} {\filldraw[fill=white,draw=c1,line width = 1pt] (\l) circle (4pt);}
        \foreach \l in {T3} {\filldraw[fill=white,draw=c2,line width = 1pt] (\l) circle (4pt);}
    \end{tikzpicture}}
\end{array}$\\
        (v) $\begin{array}{cc}
    \begin{tikzpicture}[xscale=.5,yscale=.5,line width=1.25pt] 
        \foreach \i in {1,2,3}  { \path (\i,1.25) coordinate (T\i); \path (\i,.25) coordinate (B\i); \path (\i,-.35) coordinate (TP\i); \path (\i,-1.35) coordinate (BP\i); } 
        \filldraw[fill= black!12,draw=black!12,line width=4pt]  (T1) -- (T3) -- (B3) -- (B1) -- (T1);
        \filldraw[fill= black!12,draw=black!12,line width=4pt]  (TP1) -- (TP3) -- (BP3) -- (BP1) -- (TP1);
        \draw[black] (T1)--(T2)--(B1)--(T1);
        \draw[black] (T3)--(B3)--(B2)--(T3);
        \draw[black] (TP1)--(TP2);
        \draw[black] (TP3)--(BP3);
        \draw[c4] (T1) circle (7pt);
        \draw[c4] (T3) circle (7pt);
        \foreach \l in {T1,T2,B1,B2,B3,TP1,TP2,TP3,BP1,BP2,BP3} {\filldraw[fill=white,draw=c1,line width = 1pt] (\l) circle (4pt);}
        \foreach \l in {T3} {\filldraw[fill=white,draw=c2,line width = 1pt] (\l) circle (4pt);}
        \end{tikzpicture} & \raisebox{1.5em}{$\rightarrow$}\,\,\raisebox{1em}{\begin{tikzpicture}[xscale=.5,yscale=.5,line width=1.25pt] 
        \foreach \i in {1,2,3}  { \path (\i,1.25) coordinate (T\i); \path (\i,.25) coordinate (B\i);} 
        \filldraw[fill= black!12,draw=black!12,line width=4pt]  (T1) -- (T3) -- (B3) -- (B1) -- (T1);
        \draw[black] (T1)--(T3)--(B3)--(T1);
        \draw[c4] (T1) circle (7pt);
        \draw[c4] (T3) circle (7pt);
        \foreach \l in {T1,T2,B1,B2,B3} {\filldraw[fill=white,draw=c1,line width = 1pt] (\l) circle (4pt);}
        \foreach \l in {T3} {\filldraw[fill=white,draw=c2,line width = 1pt] (\l) circle (4pt);}
    \end{tikzpicture}}
\end{array}$ &
    \end{tabular}
    \caption{Illustrations of cases when $\varphi(u)$ and $\varphi(w)$ are in the same block.}
    \label{fig:generator_lemma_cases}
\end{figure}
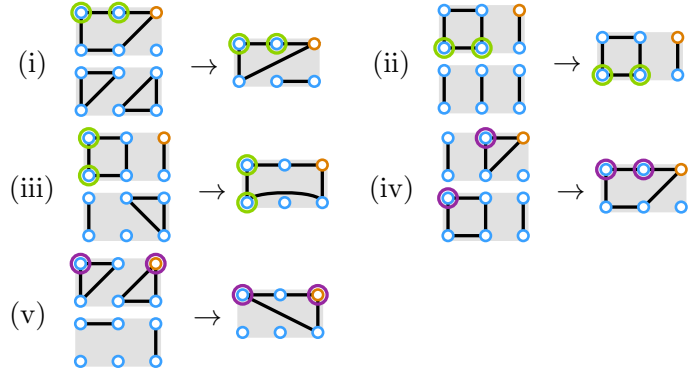

We will introduce a sort of factorization of a diagram $\tpi$ with a nonbasic block $\tB$ into diagrams with fewer nonbasic blocks, and to that end we define two diagrams $\tpi/\tB$ and $\tpi|_{\tB}$. Informally, the diagram $\tpi/\tB$ is the result of removing the block $\tB$ and replacing it with basic blocks, and the diagram $\tpi|_{\tB}$ is a diagram whose only nonbasic block is $\tB$. 

\begin{exam}\label{ex:factored_diagrams} Here we show how the diagram $\tpi$ can be factored at the nonbasic block $\tB$.
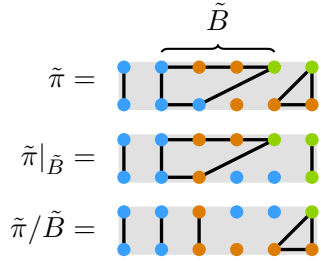
    \begin{align*}
        \tpi&=\raisebox{.75em}{$\begin{array}{c}
                \begin{tikzpicture}[xscale=.5,yscale=.5,line width=1.25pt] 
                    \foreach \i in {1,...,6}  { \path (\i,1.25) coordinate (T\i); \path (\i,.25) coordinate (B\i);  } 
                    \filldraw[fill= black!12,draw=black!12,line width=4pt]  (T1) -- (T6) -- (B6) -- (B1) -- (T1);
                    \draw[black] (T1)--(B1);
                    \draw[black] (T2)--(T5)--(B3)--(B2)--(T2);
                    \draw[black] (T6)--(B6)--(B5)--(T6);
                    \colortop{1,2}{c1}
                    \colortop{3,4}{c2}
                    \colortop{5,6}{c3}
                    \colorbot{1,2,3}{c1}
                    \colorbot{4,5,6}{c2}
                    \draw [
                thick,
                decoration={
                    brace,
                    raise=0.2cm
                },
                decorate
            ] (T2) -- (T5) 
            node  [pos=0.5,anchor=north,yshift=1cm] {$\tB$}; 
                \end{tikzpicture}
            \end{array}$}\\
            \tpi|_{\tB}&=\inlinediagram{6}{
            \draw[black] (T1)--(B1);
            \draw[black] (T2)--(T5)--(B3)--(B2)--(T2);
            \draw[black] (T6)--(B6);
            \colortop{1,2}{c1}
            \colortop{3,4}{c2}
            \colortop{5,6}{c3}
            \colorbot{1,2,4,5}{c1}
            \colorbot{3}{c2}
            \colorbot{6}{c3}}\\
            \tpi/\tB&=\inlinediagram{6}{
            \draw[black] (T1)--(B1);
            \draw[black] (T2)--(B2);
            \draw[black] (T3)--(B3);
            \draw[black] (T6)--(B6)--(B5)--(T6);
            \colortop{1,2,4,5}{c1}
            \colortop{3}{c2}
            \colortop{6}{c3}
            \colorbot{1,2}{c1}
            \colorbot{3,4,5,6}{c2}}
    \end{align*}
    Note that in $\tpi|_{\tB}$, the block $\tB$ is joined by vertical bars matching the remaining vertices atop $\tpi$ as well as two singletons $\multi{\ov1}$ at the bottom so that the number of vertices on top and on bottom match. In $\tpi/\tB$, the block $\tB$ is removed from $\tpi$, leaving behind a vertical bar matching each vertex at the bottom of $\tB$ as well as two singletons $\multi1$ so that the number of vertices on top and on bottom again match.
\end{exam}

We now define this factorization more precisely.

\begin{defi}
    Let $\tpi\in\tilde{\Pi}_{2r,k}$ have a nonbasic block $\tB$. Without loss of generality, assume $\tB$ has more unbarred entries than barred entries. Write $\tpi/\tB$ for the multiset partition obtained by replacing $\tB$ with vertical bars $\multi{i,\ov i}$ for each barred entry $\ov i$ of $\tB$ and a number of singletons $\multi{\ov 1}$ making up the difference in the number of barred and unbarred entries in $\tB$. Write $\tpi|_{\tB}$ for the multiset partition consisting of $\tB$, a vertical bar $\multi{i,\ov i}$ for each vertex labeled $i$ in $\tpi$ not in $\tB$, and enough $\multi{\ov1}$ to make it an element of $\tilde\Pi_{2(r),k}$.
\end{defi}

 We see immediately in Example \ref{ex:factored_diagrams} that although the element $\DD_\tpi$ will appear with nonzero coefficient in $\DD_{\tpi|_{\tB}}\DD_{\tpi/\tB}$, many diagrams other than $\DD_\tpi$ also appear. To systematically account for these extra diagrams, we now define a partial order in which they are smaller than the desired diagram $\DD_\tpi$. This will allow us to use the factorization recursively to write $\DD_\tpi$ as a polynomial in simpler diagrams.

 Write $\vb(\tpi)$ for the number of vertical bars in $\tpi$. Define a partial order on $\tilde{\Pi}_{r,k}$ by saying that $\tpi\prec\ttau$ if either \begin{align*}
    \nbw(\tpi)&<\nbw(\ttau)\\
    &\text{or}\\
    \nbw(\tpi)=\nbw(\ttau)&\text{ and }\vb(\tpi)<\vb(\ttau).
\end{align*}

\begin{lemma}\label{lem:factorize}
    Let $\tpi$ be a multiset partition and $\tB$ a nonbasic block of $\tpi$ with more unbarred entries than barred entries. Then there is a constant $c\in\C$ so that \[c\DD_{\tpi|_{\tB}}\DD_{\tpi/\tB}-\DD_\tpi\in\Span_\C\{\DD_\tnu:\tnu\prec\tpi\}.\]
\end{lemma}

\begin{proof}

    Consider a snapshot in the product $\DD_{\tpi|_{\tB}}\DD_{\tpi/\tB}$ in which each vertex at the bottom of the block $\tB$ in $\tpi|_{\tB}$ meets a vertical bar in $\tpi/\tB$ and each singleton $\multi{\ov1}$ in $\tpi|_{\tB}$ created in the factorization meets a singleton $\multi{1}$ in $\tpi/\tB$ created during the factorization. Such a snapshot exists because there is a vertical bar in $\tpi|_{\tB}$ matching each vertex at the bottom of $\tB$ and the number of singletons added to each factor is precisely the difference between the number of vertices in the top and in the bottom of $\tB$. The resulting diagram from this snapshot is $\tpi$, and so $\DD_\tpi$ appears with nonzero coefficient in the product. Let $c$ be the reciprocal of the coefficient it appears with.

    By Lemma \ref{lem:nonbasic_weight_and_product} any $\DD_\tnu$ appearing in the product must have $\nbw(\tnu)\leq\nbw(\tpi|_{\tB})+\nbw(\tpi/\tB)=\nbw(\tpi)$ with equality only if \[N(\tnu)=N(\tpi|_{\tB})\biguplus N(\tpi/\tB)=N(\tpi).\]

    If any vertical bars in $\tnu$ came from nonbasic blocks of $\tpi|_{\tB}$ and $\tpi/\tB$ meeting, then $\tnu$ would necessarily have a smaller nonbasic weight. Hence, in the case that $\nbw(\tnu)=\nbw(\tpi)$, it must either be the case that $\tnu=\tpi$ or $\tnu$ has fewer vertical bars. So, every $\tnu\neq\tpi$ that appears in the product is smaller in $(\tilde\Pi_{2(r),k},\preceq)$.
\end{proof}

\begin{exam} \label{ex:factorize} Lemma \ref{lem:factorize} can be used recursively to write a diagram as a polynomial in diagrams with a single nonbasic block. At each step, the nonbasic block $\tB$ that the diagram will next be factored at is highlighted. Notice that example $(ii)$ ends where example $(i)$ begins.
\begin{align*}
    (i) \begin{array}{c}
        \begin{tikzpicture}[xscale=.4,yscale=.4,line width=1.25pt] 
                \foreach \i in {1,...,4}  { \path (\i,1.25) coordinate (T\i); \path (\i,.25) coordinate (B\i);  } 
                \filldraw[fill= black!12,draw=black!12,line width=4pt]  (T1) -- (T4) -- (B4) -- (B1) -- (T1);
                \draw[black] (T2)--(T3);
                \draw[black] (T4)--(B4);
                \draw[black] (B1)--(B3);
                \draw [draw=none,fill=yellow, opacity=0.5] (1.75,1.45) -- (3.25,1.45) -- (3.25,1.05) -- (1.75,1.05)-- cycle;
                \colortop{1,2}{c1}
                \colortop{3,4}{c2}
                \colorbot{1,2}{c1}
                \colorbot{3}{c2}
                \colorbot{4}{c3}
            \end{tikzpicture}
        \end{array}&=\frac1{n^2}\left(\begin{array}{c}
            \begin{tikzpicture}[xscale=.4,yscale=.4,line width=1.25pt] 
                \foreach \i in {1,...,4}  { \path (\i,1.25) coordinate (T\i); \path (\i,.25) coordinate (B\i);  } 
                \filldraw[fill= black!12,draw=black!12,line width=4pt]  (T1) -- (T4) -- (B4) -- (B1) -- (T1);
                \draw[black] (T1)--(B1);
                \draw[black] (T2)--(T3);
                \draw[black] (T4)--(B4);
                \colortop{1,2}{c1}
                \colortop{3,4}{c2}
                \colorbot{1,2,3}{c1}
                \colorbot{4}{c2}
            \end{tikzpicture}
        \end{array}\right)\left(\begin{array}{c}
            \begin{tikzpicture}[xscale=.4,yscale=.4,line width=1.25pt] 
                \foreach \i in {1,...,4}  { \path (\i,1.25) coordinate (T\i); \path (\i,.25) coordinate (B\i);  } 
                \filldraw[fill= black!12,draw=black!12,line width=4pt]  (T1) -- (T4) -- (B4) -- (B1) -- (T1);
                \draw[black] (T4)--(B4);
                \draw[black] (B1)--(B3);
                \draw [draw=none,fill=yellow, opacity=0.5] (3.75,1.45) -- (4.25,1.45) -- (4.25,.05) -- (3.75,.05)-- cycle;
                \colortop{1,2,3}{c1}
                \colortop{4}{c2}
                \colorbot{1,2}{c1}
                \colorbot{3}{c2}
                \colorbot{4}{c3}
            \end{tikzpicture}
        \end{array}\right)\\
        &=\frac1{n^2}\left(\begin{array}{c}
            \begin{tikzpicture}[xscale=.4,yscale=.4,line width=1.25pt] 
                \foreach \i in {1,...,4}  { \path (\i,1.25) coordinate (T\i); \path (\i,.25) coordinate (B\i);  } 
                \filldraw[fill= black!12,draw=black!12,line width=4pt]  (T1) -- (T4) -- (B4) -- (B1) -- (T1);
                \draw[black] (T1)--(B1);
                \draw[black] (T2)--(T3);
                \draw[black] (T4)--(B4);
                \colortop{1,2}{c1}
                \colortop{3,4}{c2}
                \colorbot{1,2,3}{c1}
                \colorbot{4}{c2}
            \end{tikzpicture}
        \end{array}\right)\left(\begin{array}{c}
            \begin{tikzpicture}[xscale=.4,yscale=.4,line width=1.25pt] 
                \foreach \i in {1,...,4}  { \path (\i,1.25) coordinate (T\i); \path (\i,.25) coordinate (B\i);  } 
                \filldraw[fill= black!12,draw=black!12,line width=4pt]  (T1) -- (T4) -- (B4) -- (B1) -- (T1);
                \draw[black] (T1)--(B1);
                \draw[black] (T2)--(B2);
                \draw[black] (T3)--(B3);
                \draw[black] (T4)--(B4);
                \colortop{1,2,3}{c1}
                \colortop{4}{c2}
                \colorbot{1,2,3}{c1}
                \colorbot{4}{c3}
            \end{tikzpicture}
        \end{array}\right)\left(\begin{array}{c}
            \begin{tikzpicture}[xscale=.4,yscale=.4,line width=1.25pt] 
                \foreach \i in {1,...,4}  { \path (\i,1.25) coordinate (T\i); \path (\i,.25) coordinate (B\i);  } 
                \filldraw[fill= black!12,draw=black!12,line width=4pt]  (T1) -- (T4) -- (B4) -- (B1) -- (T1);
                \draw[black] (T4)--(B4);
                \draw[black] (B1)--(B3);
                \colortop{1,2,3}{c1}
                \colortop{4}{c3}
                \colorbot{1,2}{c1}
                \colorbot{3}{c2}
                \colorbot{4}{c3}
            \end{tikzpicture}
        \end{array}\right)\\
        (ii)\begin{array}{c}
            \begin{tikzpicture}[xscale=.4,yscale=.4,line width=1.25pt] 
                \foreach \i in {1,...,4}  { \path (\i,1.25) coordinate (T\i); \path (\i,.25) coordinate (B\i);  } 
                \filldraw[fill= black!12,draw=black!12,line width=4pt]  (T1) -- (T4) -- (B4) -- (B1) -- (T1);
                \draw[black] (T4)--(B4);
                \draw[black] (T2)--(T3);
                \draw[black] (T1)--(B1)--(B3)--(T1);
                \draw [draw=none,fill=yellow, opacity=0.5] (.75,1.55) -- (3.5,.05) -- (.75,.05) -- cycle;
                \colortop{1,2}{c1}
                \colortop{3,4}{c2}
                \colorbot{1,2}{c1}
                \colorbot{3}{c2}
                \colorbot{4}{c3}
            \end{tikzpicture}
        \end{array}&=\frac{3}{n^2}\left(\begin{array}{c}
            \begin{tikzpicture}[xscale=.4,yscale=.4,line width=1.25pt] 
                \foreach \i in {1,...,4}  { \path (\i,1.25) coordinate (T\i); \path (\i,.25) coordinate (B\i);  } 
                \filldraw[fill= black!12,draw=black!12,line width=4pt]  (T1) -- (T4) -- (B4) -- (B1) -- (T1);
                \draw [draw=none,fill=yellow, opacity=0.5] (3.75,1.45) -- (4.25,1.45) -- (4.25,.05) -- (3.75,.05)-- cycle;
                \draw[black] (T1)--(B1);
                \draw[black] (T4)--(B4);
                \draw[black] (T2)--(T3);
                \colortop{1,2}{c1}
                \colortop{3,4}{c2}
                \colorbot{1,2,3}{c1}
                \colorbot{4}{c3}
            \end{tikzpicture}
        \end{array}\right)\left(\begin{array}{c}
            \begin{tikzpicture}[xscale=.4,yscale=.4,line width=1.25pt] 
                \foreach \i in {1,...,4}  { \path (\i,1.25) coordinate (T\i); \path (\i,.25) coordinate (B\i);  } 
                \filldraw[fill= black!12,draw=black!12,line width=4pt]  (T1) -- (T4) -- (B4) -- (B1) -- (T1);
                \draw[black] (T4)--(B4);
                \draw[black] (T1)--(B1)--(B3)--(T1);
                \colortop{1,2,3}{c1}
                \colortop{4}{c3}
                \colorbot{1,2,3}{c1}
                \colorbot{3}{c2}
                \colorbot{4}{c3}
            \end{tikzpicture}
        \end{array}\right)-\frac2n\left(\begin{array}{c}
            \begin{tikzpicture}[xscale=.4,yscale=.4,line width=1.25pt] 
                \foreach \i in {1,...,4}  { \path (\i,1.25) coordinate (T\i); \path (\i,.25) coordinate (B\i);  } 
                \filldraw[fill= black!12,draw=black!12,line width=4pt]  (T1) -- (T4) -- (B4) -- (B1) -- (T1);
                \draw[black] (T4)--(B4);
                \draw[black] (T2)--(T3);
                \draw[black] (B1)--(B3);
                \colortop{1,2}{c1}
                \colortop{3,4}{c2}
                \colorbot{1,2}{c1}
                \colorbot{3}{c2}
                \colorbot{4}{c3}
            \end{tikzpicture}
        \end{array}\right)\\
        &=\frac{3}{n^4}\left(\begin{array}{c}
            \begin{tikzpicture}[xscale=.4,yscale=.4,line width=1.25pt] 
                \foreach \i in {1,...,4}  { \path (\i,1.25) coordinate (T\i); \path (\i,.25) coordinate (B\i);  } 
                \filldraw[fill= black!12,draw=black!12,line width=4pt]  (T1) -- (T4) -- (B4) -- (B1) -- (T1);
                \draw[black] (T1)--(B1);
                \draw[black] (T2)--(B2);
                \draw[black] (T3)--(B3);
                \draw[black] (T4)--(B4);
                \colortop{1,2}{c1}
                \colortop{3,4}{c2}
                \colorbot{1,2}{c1}
                \colorbot{3}{c2}
                \colorbot{4}{c3}
            \end{tikzpicture}
        \end{array}\right)\left(\begin{array}{c}
            \begin{tikzpicture}[xscale=.4,yscale=.4,line width=1.25pt] 
                \foreach \i in {1,...,4}  { \path (\i,1.25) coordinate (T\i); \path (\i,.25) coordinate (B\i);  } 
                \filldraw[fill= black!12,draw=black!12,line width=4pt]  (T1) -- (T4) -- (B4) -- (B1) -- (T1);
                \draw[black] (T1)--(B1);
                \draw[black] (T4)--(B4);
                \draw[black] (T2)--(T3);
                \colortop{1,2}{c1}
                \colortop{3}{c2}
                \colortop{4}{c3}
                \colorbot{1,2,3}{c1}
                \colorbot{4}{c3}
            \end{tikzpicture}
        \end{array}\right)\left(\begin{array}{c}
            \begin{tikzpicture}[xscale=.4,yscale=.4,line width=1.25pt] 
                \foreach \i in {1,...,4}  { \path (\i,1.25) coordinate (T\i); \path (\i,.25) coordinate (B\i);  } 
                \filldraw[fill= black!12,draw=black!12,line width=4pt]  (T1) -- (T4) -- (B4) -- (B1) -- (T1);
                \draw[black] (T4)--(B4);
                \draw[black] (T1)--(B1)--(B3)--(T1);
                \colortop{1,2,3}{c1}
                \colortop{4}{c3}
                \colorbot{1,2,3}{c1}
                \colorbot{3}{c2}
                \colorbot{4}{c3}
            \end{tikzpicture}
        \end{array}\right)-\frac2n\left(\begin{array}{c}
            \begin{tikzpicture}[xscale=.4,yscale=.4,line width=1.25pt] 
                \foreach \i in {1,...,4}  { \path (\i,1.25) coordinate (T\i); \path (\i,.25) coordinate (B\i);  } 
                \filldraw[fill= black!12,draw=black!12,line width=4pt]  (T1) -- (T4) -- (B4) -- (B1) -- (T1);
                \draw [draw=none,fill=yellow, opacity=0.5] (1.75,1.45) -- (3.25,1.45) -- (3.25,1.05) -- (1.75,1.05)-- cycle;
                \draw[black] (T4)--(B4);
                \draw[black] (T2)--(T3);
                \draw[black] (B1)--(B3);
                \colortop{1,2}{c1}
                \colortop{3,4}{c2}
                \colorbot{1,2}{c1}
                \colorbot{3}{c2}
                \colorbot{4}{c3}
            \end{tikzpicture}
        \end{array}\right)
\end{align*}
\end{exam}

We now introduce our generators. For $i,j\in[k]$ and $\bfa\in W_{r-1,k}$, write $P_{i,j,\bfa}$ for the diagram-like basis element indexed by the set partition with singleton blocks $\multi{i}$ and $\multi{\ov j}$ and vertical bars whose colors have multiplicity given by $\bfa$. Now fix $i\in[r]$, $\bfa,\bfb\in W_{i,k}$ and $\bfc\in W_{r-i,k}$. Write $R_{\bfa,\bfb,\bfc}$ for the diagram-like basis element for the set partition with a block whose colors on top and bottom are given by $\bfa$ and $\bfb$ respectively as well as vertical bars with multiplicities given by $\bfc$.

\begin{exam} Here we show an example of each type of generator. 
     \[
        P_{2,1,(2,0,1,2)}=\begin{tikzpicture}[xscale=.5,yscale=.5,line width=1.25pt,baseline=1.8ex] 
    \foreach \i in {1,2,3,4,5,6}  { \path (\i,1.25) coordinate (T\i); \path (\i,.25) coordinate (B\i); } 
    \filldraw[fill= black!12,draw=black!12,line width=4pt]  (T1) -- (T6) -- (B6) -- (B1) -- (T1);

    \foreach \i in {2,...,6} {\draw[black] (T\i)--(B\i);}

    \colortop{2,3}{c1}
    \colortop{1}{c2}
    \colortop{4}{c3}
    \colortop{5,6}{c4}

    \colorbot{1,2,3}{c1}
    \colorbot{4}{c3}
    \colorbot{5,6}{c4}
    \end{tikzpicture}\in\MP_{6,4}(n)\]
     \[
        R_{(2,0,1),(0,2,1),(0,0,2)}=\begin{tikzpicture}[xscale=.5,yscale=.5,line width=1.25pt,baseline=1.8ex] 
    \foreach \i in {1,2,3,4,5}  { \path (\i,1.25) coordinate (T\i); \path (\i,.25) coordinate (B\i); } 
    \filldraw[fill= black!12,draw=black!12,line width=4pt]  (T1) -- (T5) -- (B5) -- (B1) -- (T1);

    \draw[black] (T1)--(T3)--(B3)--(B1)--(T1);
    \draw[black] (T4)--(B4);
    \draw[black] (T5)--(B5);

    \colortop{1,2}{c1}
    \colortop{3,4,5}{c3}

    \colorbot{1,2}{c2}
    \colorbot{3,4,5}{c3}
    \end{tikzpicture}\in\MP_{5,3}(n)\]
\end{exam}

As a base case, we show how the elements $\{P_{i,j,\bfa}:i,j\in[k],\bfa\in W_{r-1,k}\}$ generate the diagrams with no nonbasic blocks.

\begin{lemma}\label{lem:no_nonbasics} The elements $\{P_{i,j,\bfa}:i,j\in[k],\bfa\in W_{r-1,k}\}$ generate each $\DD_\tpi$ where $\tpi$ has no nonbasic blocks.
\end{lemma}

\begin{proof}
    Fix $\bfb\in W_{r,k}$. For $m\leq\bfb_1$, write $Q_m$ for the diagram-like basis element indexed by the multiset partition with $m$ pairs of singletons $\multi{1}$ and $\multi{\ov1}$ along with vertical bars $\multi{1,\ov1}^{\bfb_1-m},\multi{2,\ov2}^{\bfb_2},\dots,\multi{k,\ov k}^{\bfb_k}$. Notice that $Q_1=P_{1,1,\bfb'}$ where $\bfb'$ is $\bfb$ with the first entry decremented by one.

    Now consider the product $Q_1Q_m$. The singleton at the bottom of $Q_1$ will meet one of the $m$ singletons at the top of $Q_m$ in $\frac{m}{\bfb_1}$ of the snapshots. In the remaining snapshots, the singleton meets a vertical bar and breaks it into a singleton, resulting in $Q_{m+1}$: \[Q_1Q_m=\frac{m}{\bfb_1}nQ_m+\frac{\bfb_1-m}{\bfb_1}Q_{m+1}.\] Hence, the elements $Q_m$ for $1\leq m\leq \bfb_1$ are generated by the elements $P_{i,j,\bfa}$ (see Figure \ref{fig:Q_and_P}(i)).

    Suppose $\tpi$ is a multiset partition with no nonbasic blocks and a singleton $\multi{i}$ with $i\neq 1$. Let $\tpi'$ be the result of replacing that $\multi{i}$ with $\multi{1}$. Then for $\bfc\in W_{r-1,k}$ chosen so that $P_{i,1,\bfc}\DD_{\tpi'}$ is nonzero, this product includes $\DD_\tpi$ along with diagrams with fewer vertical bars (see Figure \ref{fig:Q_and_P}(ii)). Via this process and the corresponding process for singletons $\multi{\ov i}$, we can write any basic diagram with a non-one singleton as a polynomial in diagrams with fewer non-one singletons or fewer vertical bars. Repeating this process for any diagram in the resulting polynomial with a non-one singleton terminates in a polynomial in diagrams with all basic blocks and singletons of the form $\multi{1}$ or $\multi{\ov 1}$. These are just the $Q_m$ above for different choices of $\bfb$, so the $\{P_{i,j,\bfa}:i,j\in[k],\bfa\in W_{r-1,k}\}$ generate the diagrams with all basic blocks.
\end{proof}

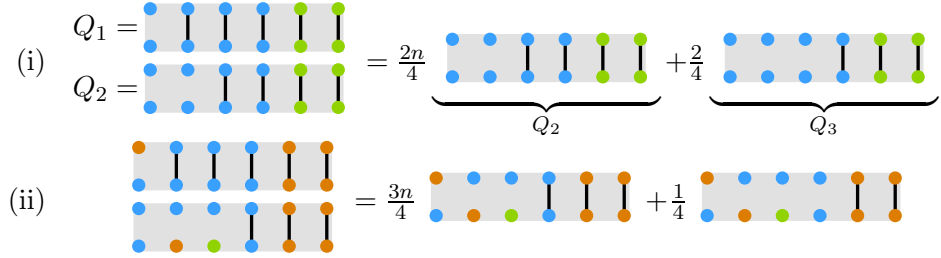
\begin{figure}[h]
    \centering
    \begin{enumerate}
\item[(i)] $\begin{array}{c}
        \raisebox{.6em}{$Q_1=$}\,\begin{tikzpicture}[xscale=.5,yscale=.5,line width=1.25pt] 
        \foreach \i in {1,2,3,4,5,6}  { \path (\i,1.25) coordinate (T\i); \path (\i,.25) coordinate (B\i); } 
            \filldraw[fill= black!12,draw=black!12,line width=4pt]  (T1) -- (T6) -- (B6) -- (B1) -- (T1);
            \draw[black] (T2)--(B2);
            \draw[black] (T3)--(B3);
            \draw[black] (T4)--(B4);
            \draw[black] (T5)--(B5);
            \draw[black] (T6)--(B6);
            \colortop{1,2,3,4}{c1}
            \colortop{5,6}{c3}
            \colorbot{1,2,3,4}{c1}
            \colorbot{5,6}{c3}
        \end{tikzpicture} \\
        \raisebox{.6em}{$Q_2=$}\,\begin{tikzpicture}[xscale=.5,yscale=.5,line width=1.25pt] 
        \foreach \i in {1,2,3,4,5,6}  { \path (\i,1.25) coordinate (T\i); \path (\i,.25) coordinate (B\i); } 
            \filldraw[fill= black!12,draw=black!12,line width=4pt]  (T1) -- (T6) -- (B6) -- (B1) -- (T1);
            \draw[black] (T3)--(B3);
            \draw[black] (T4)--(B4);
            \draw[black] (T5)--(B5);
            \draw[black] (T6)--(B6);
            \colortop{1,2,3,4}{c1}
            \colortop{5,6}{c3}
            \colorbot{1,2,3,4}{c1}
            \colorbot{5,6}{c3}
        \end{tikzpicture}
    \end{array}=\frac{2n}{4}\underbrace{\begin{array}{c}
        \begin{tikzpicture}[xscale=.5,yscale=.5,line width=1.25pt] 
        \foreach \i in {1,2,3,4,5,6}  { \path (\i,1.25) coordinate (T\i); \path (\i,.25) coordinate (B\i); } 
            \filldraw[fill= black!12,draw=black!12,line width=4pt]  (T1) -- (T6) -- (B6) -- (B1) -- (T1);
            \draw[black] (T3)--(B3);
            \draw[black] (T4)--(B4);
            \draw[black] (T5)--(B5);
            \draw[black] (T6)--(B6);
            \colortop{1,2,3,4}{c1}
            \colortop{5,6}{c3}
            \colorbot{1,2,3,4}{c1}
            \colorbot{5,6}{c3}
        \end{tikzpicture}
    \end{array}}_{Q_2}+\frac24\underbrace{\begin{array}{c}
        \begin{tikzpicture}[xscale=.5,yscale=.5,line width=1.25pt] 
        \foreach \i in {1,2,3,4,5,6}  { \path (\i,1.25) coordinate (T\i); \path (\i,.25) coordinate (B\i); } 
            \filldraw[fill= black!12,draw=black!12,line width=4pt]  (T1) -- (T6) -- (B6) -- (B1) -- (T1);
            \draw[black] (T4)--(B4);
            \draw[black] (T5)--(B5);
            \draw[black] (T6)--(B6);
            \colortop{1,2,3,4}{c1}
            \colortop{5,6}{c3}
            \colorbot{1,2,3,4}{c1}
            \colorbot{5,6}{c3}
        \end{tikzpicture}
    \end{array}}_{Q_3}$

\item[(ii)] \hspace{.31in}$\begin{array}{c}
        \begin{tikzpicture}[xscale=.5,yscale=.5,line width=1.25pt] 
        \foreach \i in {1,2,3,4,5,6}  { \path (\i,1.25) coordinate (T\i); \path (\i,.25) coordinate (B\i); } 
            \filldraw[fill= black!12,draw=black!12,line width=4pt]  (T1) -- (T6) -- (B6) -- (B1) -- (T1);
            \draw[black] (T2)--(B2);
            \draw[black] (T3)--(B3);
            \draw[black] (T4)--(B4);
            \draw[black] (T5)--(B5);
            \draw[black] (T6)--(B6);
            \colortop{2,3,4}{c1}
            \colortop{1,5,6}{c2}
            \colorbot{1,2,3,4}{c1}
            \colorbot{5,6}{c2}
        \end{tikzpicture} \\
        \begin{tikzpicture}[xscale=.5,yscale=.5,line width=1.25pt] 
        \foreach \i in {1,2,3,4,5,6}  { \path (\i,1.25) coordinate (T\i); \path (\i,.25) coordinate (B\i); } 
            \filldraw[fill= black!12,draw=black!12,line width=4pt]  (T1) -- (T6) -- (B6) -- (B1) -- (T1);
            \draw[black] (T4)--(B4);
            \draw[black] (T5)--(B5);
            \draw[black] (T6)--(B6);
            \colortop{1,2,3,4}{c1}
            \colortop{5,6}{c2}
            \colorbot{1,4}{c1}
            \colorbot{2,5,6}{c2}
            \colorbot{3}{c3}
        \end{tikzpicture}
    \end{array}=\frac{3n}{4}\begin{array}{c}
        \begin{tikzpicture}[xscale=.5,yscale=.5,line width=1.25pt] 
        \foreach \i in {1,2,3,4,5,6}  { \path (\i,1.25) coordinate (T\i); \path (\i,.25) coordinate (B\i); } 
            \filldraw[fill= black!12,draw=black!12,line width=4pt]  (T1) -- (T6) -- (B6) -- (B1) -- (T1);
            \draw[black] (T4)--(B4);
            \draw[black] (T5)--(B5);
            \draw[black] (T6)--(B6);
            \colortop{3,2,4}{c1}
            \colortop{1,5,6}{c2}
            \colorbot{1,4}{c1}
            \colorbot{2,5,6}{c2}
            \colorbot{3}{c3}
        \end{tikzpicture}
    \end{array}+\frac14\begin{array}{c}
        \begin{tikzpicture}[xscale=.5,yscale=.5,line width=1.25pt] 
        \foreach \i in {1,2,3,4,5,6}  { \path (\i,1.25) coordinate (T\i); \path (\i,.25) coordinate (B\i); } 
            \filldraw[fill= black!12,draw=black!12,line width=4pt]  (T1) -- (T6) -- (B6) -- (B1) -- (T1);
            \draw[black] (T5)--(B5);
            \draw[black] (T6)--(B6);
            \colortop{3,2,4}{c1}
            \colortop{1,5,6}{c2}
            \colorbot{1,4}{c1}
            \colorbot{2,5,6}{c2}
            \colorbot{3}{c3}
        \end{tikzpicture}
    \end{array}$
\end{enumerate}
    \caption{Examples of the processes employed in Lemma \ref{lem:no_nonbasics}.}
    \label{fig:Q_and_P}
\end{figure}

Finally, we use Lemma \ref{lem:factorize} and Lemma \ref{lem:no_nonbasics} to prove that the elements $P_{i,j,\bfa}$ and $R_{\bfa,\bfb,\bfc}$ defined above generate the algebra $\MP_{r,k}(n)$.

\begin{theorem}\label{thm:generators} The algebra $\MP_{r,k}(n)$ is generated by the set $\Theta=\{P_{i,j,\bfa}:i,j\in[k],\bfa\in W_{r-1,k}\}\cup\{R_{\bfa,\bfb,\bfc}:\bfa,\bfb\in W_{i,k},\bfc\in W_{r-i,k}\text{ for some }i\in[r]\}$.
\end{theorem}

\begin{proof}
    If $\tpi$ has more than one nonbasic block, then we apply Lemma \ref{lem:factorize} to write $\DD_\tpi$ as a polynomial in elements $\DD_\tnu$ where $\tnu\prec\tpi$. We can iterate this process on each $\DD_\tnu$ in this polynomial where $\tnu$ has more than one nonbasic block (see Example \ref{ex:factorize}). Because the poset $(\tilde\Pi_{2(r),k},\preceq)$ is finite, this iteration terminates with $\DD_\tpi$ written as a polynomial in elements $\DD_\tnu$ where each $\tnu$ has at most one nonbasic block. Hence, it suffices to show that $\Theta$ generates the elements $\DD_\tnu$ where $\tnu$ has at most one nonbasic block.

    By Lemma \ref{lem:no_nonbasics}, $\Theta$ generates the elements $\DD_\tpi$ where $\tpi$ has no nonbasic blocks. We now prove that $\Theta$ generates the diagrams with a single nonbasic block by induction on the number of vertical bars. If $\tpi$ has a single nonbasic block and no vertical bars, it has a straightforward factorization as $\DD_\tpi=\DD_{\tpi_1}\DD_{\tpi_2}\DD_{\tpi_3}$ where $\tpi_2$ is obtained from $\tpi$ by connecting all vertices into a single block, $\tpi_1$ has a vertical bar for each vertex at the top of the nonbasic block of $\tpi$ and a pair of identically colored singletons for each singleton atop $\tpi$, and $\tpi_3$ is obtained similarly from the bottom of $\tpi$ (see Figure \ref{fig:gen_thm_examples}(i)). Notice that $\DD_{\tpi_1},\DD_{\tpi_2},\DD_{\tpi_3}\in\Theta$.

    For $\tpi$ with a single nonbasic block and $s$ vertical bars, one can try a modified version of the above factorization in which a copy of each vertical bar in $\tpi$ is put in $\tpi_1$, $\tpi_2$, and $\tpi_3$ (see Figure \ref{fig:gen_thm_examples}(ii)). The element $\DD_\tpi$ appears in the product $\DD_{\tpi_1}\DD_{\tpi_2}\DD_{\tpi_3}$ when each singleton in $\tpi_1$ and $\tpi_3$ meets the nonbasic block in $\tpi_2$. When these singletons instead meet vertical bars in $\tpi_2$, the resulting diagram has fewer than $s$ vertical bars. By induction on the number of vertical bars, the set $\Theta$ generates the diagrams with at most one nonbasic block and hence the algebra $\MP_{r,k}(x)$.
\end{proof}

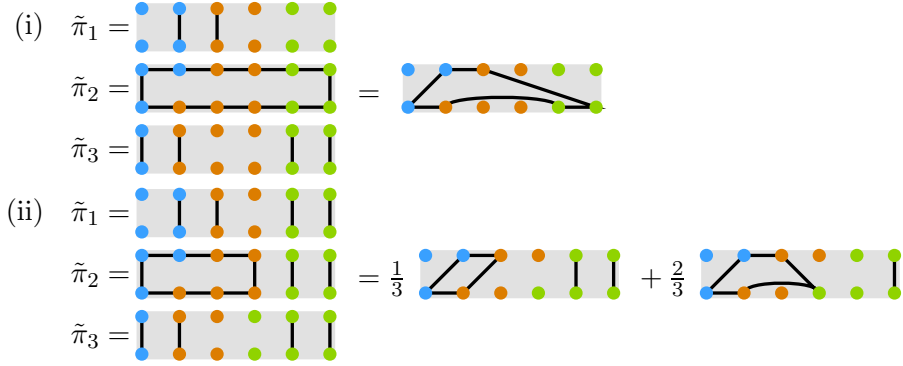
\begin{figure}
    \centering
    \begin{enumerate}
\item[(i)] \raisebox{-.35in}{$\begin{array}{c}
        \raisebox{.6em}{$\tpi_1=$}\,\begin{tikzpicture}[xscale=.5,yscale=.5,line width=1.25pt] 
        \foreach \i in {1,2,3,4,5,6}  { \path (\i,1.25) coordinate (T\i); \path (\i,.25) coordinate (B\i); } 
            \filldraw[fill= black!12,draw=black!12,line width=4pt]  (T1) -- (T6) -- (B6) -- (B1) -- (T1);
            \draw[black] (T2)--(B2);
            \draw[black] (T3)--(B3);
            \colortop{1,2}{c1}
            \colortop{3,4}{c2}
            \colortop{5,6}{c3}
            \colorbot{1,2}{c1}
            \colorbot{3,4}{c2}
            \colorbot{5,6}{c3}
        \end{tikzpicture} \\
        \raisebox{.6em}{$\tpi_2=$}\,\begin{tikzpicture}[xscale=.5,yscale=.5,line width=1.25pt] 
        \foreach \i in {1,2,3,4,5,6}  { \path (\i,1.25) coordinate (T\i); \path (\i,.25) coordinate (B\i); } 
            \filldraw[fill= black!12,draw=black!12,line width=4pt]  (T1) -- (T6) -- (B6) -- (B1) -- (T1);
            \draw[black] (T1)--(T6)--(B6)--(B1)--(T1);
            \colortop{1,2}{c1}
            \colortop{3,4}{c2}
            \colortop{5,6}{c3}
            \colorbot{1}{c1}
            \colorbot{2,3,4}{c2}
            \colorbot{5,6}{c3}
        \end{tikzpicture} \\
        \raisebox{.6em}{$\tpi_3=$}\,\begin{tikzpicture}[xscale=.5,yscale=.5,line width=1.25pt] 
        \foreach \i in {1,2,3,4,5,6}  { \path (\i,1.25) coordinate (T\i); \path (\i,.25) coordinate (B\i); } 
            \filldraw[fill= black!12,draw=black!12,line width=4pt]  (T1) -- (T6) -- (B6) -- (B1) -- (T1);
            \draw[black] (T1)--(B1);
            \draw[black] (T2)--(B2);
            \draw[black] (T5)--(B5);
            \draw[black] (T6)--(B6);
            \colortop{1}{c1}
            \colortop{2,3,4}{c2}
            \colortop{5,6}{c3}
            \colorbot{1}{c1}
            \colorbot{2,3,4}{c2}
            \colorbot{5,6}{c3}
        \end{tikzpicture}
    \end{array}=\begin{array}{c}
        \begin{tikzpicture}[xscale=.5,yscale=.5,line width=1.25pt] 
        \foreach \i in {1,2,3,4,5,6}  { \path (\i,1.25) coordinate (T\i); \path (\i,.25) coordinate (B\i); } 
            \filldraw[fill= black!12,draw=black!12,line width=4pt]  (T1) -- (T6) -- (B6) -- (B1) -- (T1);
            \draw[black] (T2)--(B1) -- (B2) .. controls +(0,+.35) and +(0,+.35) .. (B5) -- (B6) -- (T3) -- (T2);
            \colortop{1,2}{c1}
            \colortop{3,4}{c2}
            \colortop{5,6}{c3}
            \colorbot{1}{c1}
            \colorbot{2,3,4}{c2}
            \colorbot{5,6}{c3}
        \end{tikzpicture}
    \end{array}$}

\item[(ii)] \raisebox{-.35in}{$\begin{array}{c}
        \raisebox{.6em}{$\tpi_1=$}\,\begin{tikzpicture}[xscale=.5,yscale=.5,line width=1.25pt] 
        \foreach \i in {1,2,3,4,5,6}  { \path (\i,1.25) coordinate (T\i); \path (\i,.25) coordinate (B\i); } 
            \filldraw[fill= black!12,draw=black!12,line width=4pt]  (T1) -- (T6) -- (B6) -- (B1) -- (T1);
            \draw[black] (T2)--(B2);
            \draw[black] (T3)--(B3);
            \draw[black] (T5)--(B5);
            \draw[black] (T6)--(B6);
            \colortop{1,2}{c1}
            \colortop{3,4}{c2}
            \colortop{5,6}{c3}
            \colorbot{1,2}{c1}
            \colorbot{3,4}{c2}
            \colorbot{5,6}{c3}
        \end{tikzpicture} \\
        \raisebox{.6em}{$\tpi_2=$}\,\begin{tikzpicture}[xscale=.5,yscale=.5,line width=1.25pt] 
        \foreach \i in {1,2,3,4,5,6}  { \path (\i,1.25) coordinate (T\i); \path (\i,.25) coordinate (B\i); } 
            \filldraw[fill= black!12,draw=black!12,line width=4pt]  (T1) -- (T6) -- (B6) -- (B1) -- (T1);
            \draw[black] (T1)--(T4)--(B4)--(B1)--(T1);
            \draw[black] (T5)--(B5);
            \draw[black] (T6)--(B6);
            \colortop{1,2}{c1}
            \colortop{3,4}{c2}
            \colortop{5,6}{c3}
            \colorbot{1}{c1}
            \colorbot{2,3,4}{c2}
            \colorbot{5,6}{c3}
        \end{tikzpicture} \\
        \raisebox{.6em}{$\tpi_3=$}\,\begin{tikzpicture}[xscale=.5,yscale=.5,line width=1.25pt] 
        \foreach \i in {1,2,3,4,5,6}  { \path (\i,1.25) coordinate (T\i); \path (\i,.25) coordinate (B\i); } 
            \filldraw[fill= black!12,draw=black!12,line width=4pt]  (T1) -- (T6) -- (B6) -- (B1) -- (T1);
            \draw[black] (T1)--(B1);
            \draw[black] (T2)--(B2);
            \draw[black] (T5)--(B5);
            \draw[black] (T6)--(B6);
            \colortop{1}{c1}
            \colortop{2,3}{c2}
            \colortop{4,5,6}{c3}
            \colorbot{1}{c1}
            \colorbot{2,3}{c2}
            \colorbot{4,5,6}{c3}
        \end{tikzpicture}
    \end{array}=\frac13\begin{array}{c}
        \begin{tikzpicture}[xscale=.5,yscale=.5,line width=1.25pt] 
        \foreach \i in {1,2,3,4,5,6}  { \path (\i,1.25) coordinate (T\i); \path (\i,.25) coordinate (B\i); } 
            \filldraw[fill= black!12,draw=black!12,line width=4pt]  (T1) -- (T6) -- (B6) -- (B1) -- (T1);
            \draw[black] (T2)--(B1)--(B2)--(T3)--(T2);
            \draw[black] (T5)--(B5);
            \draw[black] (T6)--(B6);
            \colortop{1,2}{c1}
            \colortop{3,4}{c2}
            \colortop{5,6}{c3}
            \colorbot{1}{c1}
            \colorbot{2,3}{c2}
            \colorbot{4,5,6}{c3}
        \end{tikzpicture}
    \end{array}+\frac23\begin{array}{c}
        \begin{tikzpicture}[xscale=.5,yscale=.5,line width=1.25pt] 
        \foreach \i in {1,2,3,4,5,6}  { \path (\i,1.25) coordinate (T\i); \path (\i,.25) coordinate (B\i); } 
            \filldraw[fill= black!12,draw=black!12,line width=4pt]  (T1) -- (T6) -- (B6) -- (B1) -- (T1);
            \draw[black] (T2)--(B1)--(B2) .. controls +(0,+.35) and +(0,+.35) .. (B4) --(T3)--(T2);
            \draw[black] (T6)--(B6);
            \colortop{1,2}{c1}
            \colortop{3,4}{c2}
            \colortop{5,6}{c3}
            \colorbot{1}{c1}
            \colorbot{2,3}{c2}
            \colorbot{4,5,6}{c3}
        \end{tikzpicture}
    \end{array}$}
\end{enumerate}
    \caption{Examples of the products in Theorem \ref{thm:generators}.}
    \label{fig:gen_thm_examples}
\end{figure}

\section{Change-of-Basis Formula}\label{app:changeOfBasis}

In this section, we give a change-of-basis formula from Orellana and Zabrocki's orbit basis to the diagram-like basis. While the earlier sections dealt with the centralizer algebras $P_r(n)$ and $\MP_{r,k}(n)$, the combinatorics of the change-of-basis formula do not depend on the semisimplicity of the algebras. Hence, this section considers the abstract algebras $P_r(x)$ and $\MP_{r,k}(x)$ over $\C(x)$ for $x$ an indeterminate. These results can all be applied to the centralizer algebra case by specializing $x$ to an integer $n\geq 2r$.

In analogy with the construction of the diagram-like basis as a projection of the diagram basis of $P_r(x)$, we can define the \defn{orbit-like basis} by projecting the orbit basis of $P_r(x)$: \[\OO_{\tpi}=s_\bfa \TT_\pi s_\bfb\] where $\pi\in\Pi_{2(r)}$ is any set partition so that $\kappa_{\bfa,\bfb}(\pi)=\tpi$. For a multiset partition $\tpi$, define $m_{\tB}(\tpi)$ to be the multiplicity of the block $\tB$ in $\tpi$ and write \[m(\tpi)!=\prod_{\substack{\tB\in\tpi\\\text{distinct}}} m_\tB(\tpi)\] where the product is over \textit{distinct} blocks of $\tpi$.

\begin{theorem}\label{thm:change_of_basis} For $\tpi\in\tilde\Pi_{2(r),k}$ whose unbarred entries have multiplicity given by $\bfa\in W_{r,k}$, write \[\omega(\tpi)=\frac{m(\tpi)!}{\abs{\SG_\bfa}}\prod_{\tB\in\tpi}m(\tB|_{\ovb{k}})!.\] Then the map \begin{align*}
    \varphi:\tilde{P}_{r,k}(x)&\to\MP_{r,k}(x)\\
    \OO_\tpi&\mapsto \omega(\tpi)\XX_\tpi
\end{align*} where $\{\XX_\tpi:\tpi\in\tilde\Pi_{2(r),k}\}$ is the orbit basis of Orellana and Zabrocki is an isomorphism of algebras.
\end{theorem}

Because each $\omega(\tpi)$ is nonzero, it is clear that this map is an isomorphism of vector spaces--the remainder of this section is devoted to proving that it respects the multiplication. First, we observe that such an isomorphism gives us the following change-of-basis formula from Orellana and Zabrocki's orbit basis to the diagram-like basis:

\begin{align*}
    \DD_\tpi&=\varphi(s_\bfa L_\pi s_\bfb)\\
            &=\sum_{\tnu\leq\tpi} c_{\tnu,\tpi}\omega(\tnu)\XX_\tnu
\end{align*} where for a fixed $\pi$ such that $\kappa_{\bfa,\bfb}(\pi)=\tpi$, $c_{\tnu,\tpi}$ is the number of $\nu\leq\pi$ such that $\kappa_{\bfa,\bfb}(\nu)=\tnu$.

\begin{exam} We expand the following diagram-like basis element in the orbit basis of Orellana and Zabrocki:
    \newcommand{\subdiag}[1]{\hspace{-.08in}\begin{tabular}{c}
    \begin{tikzpicture}[xscale=.25,yscale=.25,line width=1pt] 
        \foreach \i in {1,2,3,4}  { \path (\i,1.25) coordinate (T\i); \path (\i,.25) coordinate (B\i);} 
        \filldraw[fill= black!12,draw=black!12,line width=4pt]  (T1) -- (T4) -- (B4) -- (B1) -- (T1);
        #1
        \colortop{1,2,3,4}{black}
        \colorbot{1,2,3,4}{black}
    \end{tikzpicture}
\end{tabular}}

\newcommand{\submultidiag}[1]{\hspace{-.08in}\begin{tabular}{c}
    \begin{tikzpicture}[xscale=.25,yscale=.25,line width=1pt] 
        \foreach \i in {1,2,3,4}  { \path (\i,1.25) coordinate (T\i); \path (\i,.25) coordinate (B\i);} 
        \filldraw[fill= black!12,draw=black!12,line width=4pt]  (T1) -- (T4) -- (B4) -- (B1) -- (T1);
        #1
        \colortop{1,2}{c1}
        \colortop{3,4}{c2}
        \colorbot{1,2,3}{c1}
        \colorbot{4}{c2}
    \end{tikzpicture}
\end{tabular}}

\begin{align*}
    \DD_{\submultidiag{\draw[black] (T1)--(B1);\draw[black] (T2)--(B2);\draw[black] (T3)--(T4)--(B4)--(B3)--(T3);}}&=s_{(2,2)}\LL_{\subdiag{\draw[black] (T1)--(B1);\draw[black] (T2)--(B2);\draw[black] (T3)--(T4)--(B4)--(B3)--(T3);}}s_{(3,1)}\\
    \hspace{-.1in}&=s_{(2,2)}\left(
    \TT_{\subdiag{\draw[black] (T1)--(B1);\draw[black] (T2)--(B2);\draw[black] (T3)--(T4)--(B4)--(B3)--(T3);}}
    \hspace{-.1in}+\TT_{\subdiag{\draw[black] (T1)--(B1)--(B2)--(T2)--(T1);\draw[black] (T3)--(T4)--(B4)--(B3)--(T3);}}
    \hspace{-.1in}+\TT_{\subdiag{\draw[black] (T1)--(B1);\draw[black] (T2)--(T4)--(B4)--(B2)--(T2);}}
    \hspace{-.1in}+\TT_{\subdiag{\draw[black] (T1)  .. controls +(0,-.4) and +(0,-.4) .. (T3) -- (T4) -- (B4) -- (B3) .. controls +(0,+.4) and +(0,+.4) .. (B1) -- (T1);\draw[black] (T2) -- (B2);}}
    \hspace{-.1in}+\TT_{\subdiag{\draw[black] (T1)--(T4)--(B4)--(B1)--(T1);}}\hspace{-.1in}\right)s_{(3,1)}
\end{align*}
\begin{align*}
    \DD_{\submultidiag{\draw[black] (T1)--(B1);\draw[black] (T2)--(B2);\draw[black] (T3)--(T4)--(B4)--(B3)--(T3);}}&=\hspace{.35in}\OO_{\submultidiag{\draw[black] (T1)--(B1);\draw[black] (T2)--(B2);\draw[black] (T3)--(T4)--(B4)--(B3)--(T3);}}
    \hspace{-.1in}+\hspace{.35in}\OO_{\submultidiag{\draw[black] (T1)--(B1)--(B2)--(T2)--(T1);\draw[black] (T3)--(T4)--(B4)--(B3)--(T3);}}
    \hspace{-.1in}+\hspace{.3in}2\OO_{\submultidiag{\draw[black] (T1)--(B1);\draw[black] (T2)--(T4)--(B4)--(B2)--(T2);}}
    \hspace{-.1in}+\hspace{.35in}\OO_{\submultidiag{\draw[black] (T1)--(T4)--(B4)--(B1)--(T1);}}\\
    &\mapsto\omega(\tpi_1)\XX_{\submultidiag{\draw[black] (T1)--(B1);\draw[black] (T2)--(B2);\draw[black] (T3)--(T4)--(B4)--(B3)--(T3);}}
    \hspace{-.1in}+\omega(\tpi_2)\XX_{\submultidiag{\draw[black] (T1)--(B1)--(B2)--(T2)--(T1);\draw[black] (T3)--(T4)--(B4)--(B3)--(T3);}}
    \hspace{-.1in}+2\omega(\tpi_3)\XX_{\submultidiag{\draw[black] (T1)--(B1);\draw[black] (T2)--(T4)--(B4)--(B2)--(T2);}}
    \hspace{-.1in}+\omega(\tpi_4)\XX_{\submultidiag{\draw[black] (T1)--(T4)--(B4)--(B1)--(T1);}}
    \\
    &=\hspace{.23in}\frac{2!}{3!}\XX_{\submultidiag{\draw[black] (T1)--(B1);\draw[black] (T2)--(B2);\draw[black] (T3)--(T4)--(B4)--(B3)--(T3);}}
    \hspace{-.1in}+\hspace{.2in}\frac{2!}{3!}\XX_{\submultidiag{\draw[black] (T1)--(B1)--(B2)--(T2)--(T1);\draw[black] (T3)--(T4)--(B4)--(B3)--(T3);}}
    \hspace{-.1in}+\hspace{.2in}2\frac{2!}{3!}\XX_{\submultidiag{\draw[black] (T1)--(B1);\draw[black] (T2)--(T4)--(B4)--(B2)--(T2);}}
    \hspace{-.1in}+\hspace{.21in}\frac{3!}{3!}\XX_{\submultidiag{\draw[black] (T1)--(T4)--(B4)--(B1)--(T1);}}\\
    &=\hspace{.26in}\frac13\XX_{\submultidiag{\draw[black] (T1)--(B1);\draw[black] (T2)--(B2);\draw[black] (T3)--(T4)--(B4)--(B3)--(T3);}}
    \hspace{-.1in}+\hspace{.24in}\frac13\XX_{\submultidiag{\draw[black] (T1)--(B1)--(B2)--(T2)--(T1);\draw[black] (T3)--(T4)--(B4)--(B3)--(T3);}}
    \hspace{-.1in}+\hspace{.32in}\frac23\XX_{\submultidiag{\draw[black] (T1)--(B1);\draw[black] (T2)--(T4)--(B4)--(B2)--(T2);}}
    \hspace{-.1in}+\hspace{.34in}\XX_{\submultidiag{\draw[black] (T1)--(T4)--(B4)--(B1)--(T1);}}
\end{align*}
\end{exam}

\subsection{Preliminary Definitions and Enumerative Results}

The product formula for the orbit basis $\{\TT_\pi:\pi\in\Pi_{2(r)}\}$ for $P_r(x)$ use set partitions that include unbarred, barred, and double-barred elements. To that end, we make the following definitions. For $\pi,\nu\in\Pi_{2(r)}$, write $\Gamma^\pi_\nu$ for the set of set partitions $\gamma$ of $[r]\cup\ovb{r}\cup\ovovb{r}$ such that $\gamma|_{[r]\cup\ovb{r}}=\pi$ and $\gamma|_{\ovb{r}\cup\ovovb{r}}=\ov\nu$ where $\ov\nu$ is the result of adding a bar to each element in $\nu$. For such a $\gamma$, write $\beta_\gamma=\{S\in\gamma:\forall i\in S, i\in\ovb{r}\}$ for the set of blocks of $\gamma$ containing only barred numbers. Write \[b_\gamma(x)=(x-\ell(\gamma|_{[r]\cup\ovovb{r}}))_{\ell(\beta_\gamma)}\]  where $(a)_n=a(a-1)\cdots(a-n+1)$. The product formula for the orbit basis is

\begin{align*}
    \TT_\pi \TT_\nu&=\sum_{\gamma\in\Gamma^\pi_\nu}b_\gamma(x)\TT_\gamma
\end{align*} where the $\gamma$ in the subscript is understood to be an element of $\Pi_{2(r)}$ by taking the restriction $\gamma|_{[r]\cup\ovovb{r}}$ and removing a bar from each double-barred entry (see Theorem 4.14 of \cite{benkart2017partition} for details).

For the orbit basis $\{\XX_\tpi:\tpi\in\tilde\Pi_{2(r),k}\}$ of $\MP_{r,k}(x)$ we make similar definitions. For $\tpi,\tnu\in\tilde\Pi_{2(r),k}$, write $\tilde\Gamma^\tpi_\tnu$ for the set of multiset partitions $\tgamma$ with $r$ elements each from $[k]$, $\ovb{k}$, and $\ovovb{k}$ such that $\tgamma|_{[k]\cup\ovb{k}}=\tpi$ and $\tgamma|_{\ovb{k}\cup\ovovb{k}}=\ov\tnu$. For such a $\tgamma$, write $\beta_\tgamma=\multi{\tS\in\tgamma:\forall i\in \tS, i\in\ovb{k}}$. For any multiset partition $\trho$, write $m_\tS(\trho)$ for the multiplicity of $\tS$ in $\trho$ and \[m(\trho)!=\prod_{\substack{\tS\in\trho\\\text{distinct}}}m_\tS(\trho)!.\] Finally, write \begin{align*}
    \tilde{b}_\tgamma(x)&=\frac{(x-\ell(\tgamma|_{[k]\cup\ovovb{k}}))_{\ell(\beta_\tgamma)}}{m(\beta_\tgamma)!}\\
    a_\tgamma&=\prod_{\substack{\tS\in\tgamma|_{[k]\cup\ovovb{k}}\\\text{distinct}}}\frac{\ell(\tgamma_\tS)!}{m(\tgamma_\tS)!}
\end{align*} where $\tgamma_\tS=\multi{\tT\in\tgamma:\tT|_{[k]\cup\ovovb{k}}=\tS}$. Then, the product for the orbit basis of $\MP_{r,k}(x)$ is given by \begin{align*}
    \XX_\tpi \XX_\tnu &=\sum_{\tgamma\in\tilde\Gamma^\tpi_\tnu} a_\tgamma \tilde{b}_\tgamma(x) \XX_\tgamma
\end{align*} where the $\tgamma$ in the subscript is understood to be an element of $\tilde{\Pi}_{2r,k}$ by taking the restriction $\gamma|_{[k]\cup\ovovb{k}}$ and removing a bar from each double-barred entry (see Section 3 of \cite{orellana2020howe} for details).

To handle these set and multiset partitions on three alphabets combinatorially, we will want to extend the notation of our painting function $\kappa_{\bfa,\bfb}$ to them. In particular, $\kappa_{\bfa,\bfb,\bfc}(\gamma)$ will be the result of replacing the unbarred, barred, and double-barred elements according to $\bfa,\bfb$, and $\bfc$ respectively. It will be useful later to write $\Gamma_\nu^\pi(\tmu)$ for the set of $\gamma\in\Gamma_\nu^\pi$ such that $\kappa_{\bfa,\bfb,\bfc}(\gamma)=\tmu$.

For $\pi,\nu\in\Pi_{2(r)}$ with $\pi|_{\ovb{r}}=\ov\nu|_{\ov r}$, let $\pi\ast\nu\in\Gamma^\pi_\nu$ be the set partition obtained by placing the diagram of $\pi$ atop the diagram of $\nu$ and identifying the corresponding vertices in the center. This set partition plays a central role because any $\gamma\in\Gamma^\pi_\nu$ only differs from $\pi\ast\nu$ by connecting some blocks on the very top and very bottom. That is, if we denote by \defn{$\Break(\gamma)$} the result of splitting each block of $\gamma$ that does not contain a vertex in the middle row of the diagram into its restriction to $[r]$ and restriction to $\ovb{r}$, then $\gamma\in\Gamma^\pi_\nu$ if and only if $\Break(\gamma)=\pi\ast\nu$.

\begin{exam} We compute $\Break(\gamma)$ of a set partition $\gamma$ of $[6]\cup\ovb{6}\cup\ovovb{6}$.
    \begin{align*}
        \gamma=\begin{array}{c}
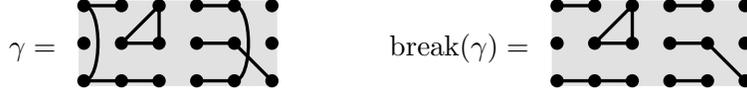
            \begin{tikzpicture}[xscale=.5,yscale=.5,line width=1.25pt] 
                \foreach \i in {1,2,3,4,5,6}  { \path (\i,1.25) coordinate (T\i); \path (\i,.25) coordinate (M\i); \path (\i,-.75) coordinate (B\i); } 
                \filldraw[fill= black!12,draw=black!12,line width=4pt]  (T1) -- (T6) -- (B6) -- (B1) -- (T1);
                \draw[black] (T1) .. controls +(+0.50,0) and +(+0.50,0) .. (B1);
                \draw[black] (T1) -- (T2);
                \draw[black] (B1) -- (B2) -- (B3);
                \draw[black] (T3) -- (M3) -- (M2) -- (T3);
                \draw[black] (T4) -- (T5);
                \draw[black] (T5) .. controls +(+0.50,0) and +(+0.50,0) .. (B5);
                \draw[black] (B4) -- (B5);
                \draw[black] (M4) -- (M5) -- (B6);
                \colortop{1,2,3,4,5,6}{black};
                \colormid{1,2,3,4,5,6}{black};
                \colorbot{1,2,3,4,5,6}{black};
            \end{tikzpicture}
        \end{array} && \Break(\gamma)= \begin{array}{c}
            \begin{tikzpicture}[xscale=.5,yscale=.5,line width=1.25pt] 
                \foreach \i in {1,2,3,4,5,6}  { \path (\i,1.25) coordinate (T\i); \path (\i,.25) coordinate (M\i); \path (\i,-.75) coordinate (B\i); } 
                \filldraw[fill= black!12,draw=black!12,line width=4pt]  (T1) -- (T6) -- (B6) -- (B1) -- (T1);
                \draw[black] (T1) -- (T2);
                \draw[black] (B1) -- (B2) -- (B3);
                \draw[black] (T3) -- (M3) -- (M2) -- (T3);
                \draw[black] (T4) -- (T5);
                \draw[black] (B4) -- (B5);
                \draw[black] (M4) -- (M5) -- (B6);
                \colortop{1,2,3,4,5,6}{black};
                \colormid{1,2,3,4,5,6}{black};
                \colorbot{1,2,3,4,5,6}{black};
            \end{tikzpicture}
        \end{array}
    \end{align*}
\end{exam}

To analyze the orbit basis of $\MP_{r,k}(x)$, we want to investigate how $\pi\ast\nu$ acts when $\nu$ is acted upon by some permutation. This inspires the definition of several subgroups of permutations. For $\rho$ a set partition of $[r]$ and $\bfa\in W_{r,k}$, define \begin{align*}
    \SG_\bfa^\rho=\{\sigma\in\SG_\bfa:\sigma.\rho=\rho\}
\end{align*} where the action $\sigma.\rho$ applies $\sigma$ to each element of each block of $\rho$. This subgroup factors as a semidirect product \begin{align*}
    \SG_\bfa^\rho=X_\bfa^\rho Y_\bfa^\rho
\end{align*} where the permutations in $X_\bfa^\rho$ permute whole blocks and the permutations in $Y_\bfa^\rho$ permute only within blocks of $\rho$.

\begin{exam} For clarity, we illustrate the above factorization in the case that all elements are the same color (where the corresponding young subgroup $\SG_{(r)}$ is equal to the full symmetric group $\SG_r$). The permutation \[\sigma=(3\,4)(1\,5)(2\,6)(8)(7\,9)\] fixes the set partition $\rho=\{\{1,2,4\},\{3,5,6\},\{8\},\{7,9\}\}$ and hence $\sigma\in \SG_{(r)}^\rho$. It then factors as $\sigma=\sigma_X\sigma_Y$ where \[\sigma_X=(1\,4\,2)(3\,5\,6)(7\,9)\] permutes elements within each block and \[\sigma_Y=(1\,3)(2\,5)(4\,6)\] swaps the two blocks of $\rho$ of size three. 
\end{exam}

Given $\pi\in\Pi_{2(r)}$, consider the subgroup $A_{\bfa,\bfb}^\pi$ of $X_\bfb^{\pi|_{\ovb{r}}}$ that only permutes blocks in $\pi|_{\ovb{r}}$ if they are part of blocks of $\pi$ that are painted identically by $\kappa_{\bfa,\bfb}$ (see Example \ref{ex:kappa}). Let $B_{\bfa,\bfb}^\nu$ be the analogous subgroup of $X_\bfb^{\nu|_{\ovb{r}}}$ for the restrictions to the top. More precisely,
\begin{align*}
    A^\pi_{\bfa,\bfb}&=\{\sigma\in X_\bfb^{\pi|_{\ovb{r}}}:\forall S,T\in\pi, S|_{\ovb{r}}=\sigma(T|_{\ovb{r}})\implies \kappa_{\bfa,\bfb}(S)=\kappa_{\bfa,\bfb}(T)\}\\
    B^\nu_{\bfb,\bfc}&=\{\sigma\in X_\bfb^{\nu|_{[r]}}:\forall S,T\in\nu, S|_{[r]}=\sigma(T|_{[r]})\implies \kappa_{\bfb,\bfc}(S)=\kappa_{\bfb,\bfc}(T)\}
\end{align*}

Put another way, these permutations $\sigma$ of the blocks in the middle row of $\pi\ast\nu$ do not change the resulting multiset partition. That is, $\kappa_{\bfa,\bfb,\bfc}(\pi\ast\nu)=\kappa_{\bfa,\bfb,\bfc}(\pi\ast\sigma.\nu)$ for $\sigma\in A_{\bfa,\bfb}^\pi$ or $\sigma\in B_{\bfb,\bfc}^\nu$.

Finally, we collect up formulas for the sizes of these subgroups. To that end, it will be useful to consider the following multiset partitions obtained by restricting to particular blocks.

\begin{align*}
    \tpi_+&=\{\tS\in\tpi:\forall i\in S,i\in[k]\}\\
    \tpi_-&=\{\tS\in\tpi:\forall i\in S,i\in\ovb{k}\}\\
    \tgamma_\pm&=\{\tS\in\tpi:\forall i\in S,i\in[k]\cup\ovovb{k}\}
\end{align*}

We think of $\tpi_+$ (resp. $\tpi_-$) as the blocks contained entirely in the top (resp. bottom) of $\tpi$ and $\tgamma_\pm$ as the blocks of $\tgamma$ that have no vertex in the middle row.

\begin{lemma}\label{lem:subgroup_sizes}
    Let $\bfa,\bfb,\bfc\in W_{r,k}$, $\pi,\nu\in\Pi_{2(r)}$, $\tpi=\kappa_{\bfa,\bfb}(\pi),$ and $\tnu=\kappa_{\bfb,\bfc}(\nu)$. Then, the subgroups defined above have sizes given by the following.

    \begin{align*}
        \begin{array}{ccc}
            \abs{X^\rho_\bfa}=m(\kappa_\bfa(\rho))! & \hspace{.5in} & \abs{Y^\rho_\bfa}=\prod_{\tB\in\tpi}m(\tB|_{\ovb{k}})! \\
            \abs{A^\pi_{\bfa,\bfb}}=\frac{m(\tpi)!}{m(\tpi_+)!} & & \abs{B^\nu_{\bfb,\bfc}}=\frac{m(\tnu)!}{m(\tnu_-)!}
        \end{array}
    \end{align*}

    Furthermore, for any $\tgamma\in\tilde\Gamma^\tpi_\tnu$, the size of the intersection of $A^\pi_{\bfa,\bfb}$ and $B^\nu_{\bfb,\bfc}$ is given by \begin{align*}
        \abs{A^\pi_{\bfa,\bfb}\cap B^\nu_{\bfb,\bfc}}&=\frac{m(\tgamma)!}{m(\tgamma_\pm)!}.
    \end{align*}
\end{lemma}

\begin{proof}
    The equalities in the first row are clear. The next two follow from the observation that the only blocks of $\tpi$ that contribute to $A_{\bfa,\bfb}^\pi$ are the ones that touch the bottom, so we cancel out the contribution of those contained entirely in the top.

    For the last equality, we can think of $A_{\bfa,\bfb}^\pi\cap B^\nu_{\bfb,\bfc}$ as the permutations of the vertices in the middle row of $\pi\ast\nu$ that only permute blocks that are the restrictions of blocks painted the same in $\kappa_{\bfa,\bfb,\bfc}(\pi\ast\nu)$. The number of such permutations is $\frac{m(\kappa_{\bfa,\bfb,\bfc}(\pi\ast\nu))!}{m(\kappa_{\bfa,\bfb,\bfc}(\pi\ast\nu)_\pm)!}$. Because $\tgamma$ only differs from $\kappa_{\bfa,\bfb,\bfc}(\pi\ast\nu)$ by blocks that do not touch the center (whose contributions are all canceled) we can make the substitution of $\tgamma$ for $\kappa_{\bfa,\bfb,\bfc}(\pi\ast\nu)$.
\end{proof}

\subsection{Proof of the Isomorphism}

\begin{proof}[Proof of Theorem \ref{thm:change_of_basis}]
    Let $\tpi,\tnu\in\tilde\Pi_{2(r)}$ and let $\bfa,\bfb,\bfb',\bfc\in W_{r,k}$ be such that there exist $\pi,\nu\in\Pi_{2(r)}$ so that $\kappa_{\bfa,\bfb}(\pi)=\tpi$ and $\kappa_{\bfb',\bfc}(\nu)=\tnu$. Note that when $\bfb\neq\bfb'$, we have $\OO_\tpi\OO_\tnu=0$ and $\XX_\tpi \XX_\tnu=0$, so we need only address the case when $\bfb=\bfb'$:

    \begin{align*}
        \OO_\tpi\OO_\tnu&=\frac{1}{\abs{\SG_\bfb}}\sum_{\sigma\in\SG_\bfb}s_\bfa\TT_\pi\TT_{\sigma.\nu}s_\bfc\\
        &=\frac{1}{\abs{\SG_\bfb}}\sum_{\sigma\in\SG_\bfb}\sum_{\gamma\in\Gamma^\pi_{\sigma.\nu}} b_\gamma(x) s_\bfa \TT_\gamma s_\bfb\\
        \intertext{To simplify notation, we will write $\tgamma=\kappa_{\bfa,\bfb,\bfc}(\gamma)$. Note that $b_\gamma(x)=\tilde{b}_{\tgamma}(x)m(\beta_\tgamma)!$, so we can rewrite the expression as follows.}
        \OO_\tpi\OO_\tnu&=\frac{1}{\abs{\SG_\bfb}}\sum_{\sigma\in\SG_\bfb}\sum_{\gamma\in\Gamma^\pi_{\sigma.\nu}} \tilde{b}_\tgamma(x)m(\beta_\tgamma)! \OO_{\tgamma}\\
        \intertext{We then partition the sum over the possible multiset partitions $\tmu$ that could arise from $\gamma$ in this sum, noting that $\kappa_{\bfb,\bfc}(\sigma.\nu)=\tnu$, and then we swap the order of summation to obtain}
        \OO_\tpi\OO_\tnu&=\frac{1}{\abs{\SG_\bfb}}\sum_{\tmu\in\tilde\Gamma_{\tnu}^\tpi} \tilde{b}_\tmu(x)m(\beta_\tmu)!\left(\sum_{\sigma\in\SG_\bfb}\sum_{\substack{\gamma\in\Gamma^\pi_{\sigma.\nu}\\\tgamma=\tmu}}1\right)\OO_\tmu\\
        &=\frac{1}{\abs{\SG_\bfb}}\sum_{\tmu\in\tilde\Gamma_{\tnu}^\tpi} \tilde{b}_\tmu(x)m(\beta_\tmu)!\left(\sum_{\sigma\in\SG_\bfb}\abs{\Gamma^\pi_{\sigma.\nu}(\tmu)}\right)\OO_\tmu.
    \end{align*}
    \begin{align}
        \intertext{Recall that $\varphi(\OO_\tpi)=\omega(\tpi)X_\tpi$. Applying $\varphi$ to both sides of the above equation and writing $\varphi(\OO_\tpi\OO_\tnu)|_{\XX_\ttau}$ for the coefficient of $\XX_\ttau$ in $\varphi(\OO_\tpi\OO_\tnu)$ then yields}
        \varphi(\OO_\tpi\OO_\tnu)|_{\XX_\ttau}&=\frac{\omega(\ttau)}{\abs{\SG_\bfb}}\sum_{\substack{\tmu\in\Gamma^\tpi_\tnu\\\tmu|_{[k]\cup\ovovb{k}}=\ttau}}\tilde{b}_\tmu(x)m(\beta_\tmu)!\left(\sum_{\sigma\in\SG_\bfb}\abs{\Gamma^\pi_{\sigma.\nu}(\tmu)}\right).\label{eq:phi}\\
        \intertext{We now compare this to the same coefficient in $\varphi(\OO_\tpi)\varphi(\OO_\tnu)$:}
        \varphi(\OO_\tpi)\varphi(\OO_\tnu)|_{\XX_\ttau}&=\omega(\tpi)\omega(\tnu)\sum_{\substack{\tmu\in\tilde\Gamma^\tpi_\tnu\\\tmu|_{[k]\cup\ovovb{k}=\ttau}}}a_\tmu \tilde{b}_\tmu(x)\label{eq:phiphi}
    \end{align}

    The goal is now to show that the rather unpleasant expressions given in Equation \ref{eq:phi} for $\varphi(\OO_\tpi\OO_\tnu)$ and Equation \ref{eq:phiphi} for $\varphi(\OO_\tpi)\varphi(\OO_\tnu)$ are equal. We will leverage their similarities---namely that they sum over the same objects and each include a factor of $\tilde{b}_\tmu(x)$ in each summand---to simplify the task. It would suffice to show the following equality:

    \begin{align*}
        \frac{\omega(\ttau)}{\abs{\SG_\bfb}}m(\beta_\tmu)!\left(\sum_{\sigma\in\SG_\bfb}\abs{\Gamma^\pi_{\sigma.\nu}(\tmu)}\right)&=\omega(\tpi)\omega(\tnu)a_\tmu
    \end{align*}

    By using the definition of $\omega(\tpi)$ and noticing that $\ttau|_{\ovb{k}}=\tnu|_{\ovb{k}}$, we are able to rearrange the above equality to

    \begin{align*}
        \sum_{\sigma\in\SG_\bfb}\abs{\Gamma^\pi_{\sigma.\nu}(\tmu)}&=a_\tmu\frac{m(\tpi)!m(\tnu)!}{m(\ttau)!m(\beta_\tmu)!}\prod_{\tB\in\tpi}m(\tB|_{\ovb{k}})!.
    \end{align*}
    Then using the formulas in Lemma \ref{lem:subgroup_sizes} for the sizes of the subgroups, we rewrite the right-hand side to obtain
    \begin{align*}
        \sum_{\sigma\in\SG_\bfb}\abs{\Gamma^\pi_{\sigma.\nu}(\tmu)}&=a_\tmu\frac{m(\tpi_+)!m(\tnu_-)!}{m(\ttau)!m(\beta_\tmu)!}\frac{m(\tmu)!}{m(\tmu_\pm)!}\frac{\abs{A^\pi_{\bfa,\bfb}}\abs{B^\nu_{\bfb,\bfc}}\abs{Y_\bfb^{\pi|_{\ovb{r}}}}}{\abs{A_{\bfa,\bfb}^\pi\cap B_{\bfb,\bfc}^\nu}}\\
        &=a_\tmu\frac{m(\tpi_+)!m(\tnu_-)!}{m(\ttau)!m(\beta_\tmu)!}\frac{m(\tmu)!}{m(\tmu_\pm)!}\abs{A_{\bfa,\bfb}^\pi B_{\bfb,\bfc}^\nu Y_\bfb^{\pi|_{\ovb{r}}}}\\
        &=\left(a_\tmu\frac{m(\tmu)!}{m(\ttau)!m(\beta_\tmu)!}\right)\frac{m(\tpi_+)!m(\tnu_-)!}{m(\tmu_\pm)!}
    \end{align*} where the second equality follows from the fact that $A_{\bfa,\bfb}^\pi$ and $B_{\bfb,\bfc}^\nu$ are subgroups of $X_\bfb^{\pi|_{\ovb{r}}}$, which intersects trivially with $Y_\bfb^{\pi|_{\ovb{r}}}$.
    
    Finally, we turn our attention to the factor in parentheses in the above expression. Because $\ttau=\tmu|_{\ovb{k}\cup\ovovb{k}}$, we have for each block $\tS\in\tmu|_{\ovb{k}\cup\ovovb{k}}$ that $\ell(\tmu_{\tS})=m_{\tS}(\ttau)$. This fact allows us to cancel the $\ell(\tmu_\tS)!$ terms in the definition of $a_\tmu$ with $m(\ttau)!$ to get
    \begin{align*}
        a_\tmu\frac{m(\tmu)!}{m(\ttau)!m(\beta_\tmu)!}&=\left(\prod_{\substack{\tS\in\tmu|_{[k]\cup\ovovb{k}}\\\text{distinct}}}\frac{\ell(\tmu_\tS)!}{m(\tmu_\tS)!}\right)\frac{m(\tmu)!}{m(\ttau)!m(\beta_\tmu)!}\\
        &=\left(\prod_{\substack{\tS\in\tmu|_{[k]\cup\ovovb{k}}\\\text{distinct}}}\frac{1}{m(\tmu_\tS)!}\right)\frac{m(\tmu)!}{m(\beta_\tmu)!}.
    \end{align*} 

    Note that the terms $m(\tmu_\tS)!$ contribute $m_\tU(\tmu)!$ to the denominator for each block $\tU\in\tmu$ which contains either an unbarred or a double-barred element. Furthermore, $m(\beta_\tmu)!$ contributes a $m_\tU(\tmu)!$ to the denominator for each block $\tU\in\tmu$ consisting only of barred elements. In total, this cancels the contribution of $m(\tmu)!$, so \begin{align*}
        a_\tmu\frac{m(\tmu)!}{m(\ttau)!m(\beta_\tmu)!}&=1
    \end{align*}

    Hence, $\varphi$ is an isomorphism of algebras if the equality \[\sum_{\sigma\in\SG_\bfb}\abs{\Gamma^\pi_{\sigma.\nu}(\tmu)}=\frac{m(\tpi_+)m(\tnu_-)}{m(\tmu_\pm)}{\abs{A_{\bfa,\bfb}^\pi B_{\bfb,\bfc}^\nu Y_{\bfb}^{\pi|_{\ovb{r}}}}}\] holds. The proof of this equality will be carried out in two lemmas. First, Lemma \ref{lem:number_of_sigma} will show that the set of $\sigma$ such that $\abs{\Gamma_{\sigma.\nu}^\pi(\tmu)}\neq0$ is given by a translation of the product $A_{\bfa,\bfb}^\pi B_{\bfb,\bfc}^{\nu'} Y_\bfb^{\pi|_{\ovb{r}}}$ where $\kappa_{\bfb,\bfc}(\nu')=\kappa_{\bfb,\bfc}(\nu)$ so $B_{\bfb,\bfc}^{\nu'}\cong B_{\bfb,\bfc}^{\nu}$. Finally, Lemma \ref{lem:size_of_gamma} will show that for each such $\sigma$, \begin{align*}
        \abs{\Gamma_{\sigma.\nu}^\pi(\tmu)}=\frac{m_+(\tpi)m_-(\tnu)}{m_\pm(\tmu)}.
    \end{align*} Because this quantity is independent of $\sigma$, the value of the sum is simply the product of the quantity $\frac{m_+(\tpi)m_-(\tnu)}{m_\pm(\tmu)}$ with the number of nonzero summands, given by $\abs{A_{\bfa,\bfb}^\pi B_{\bfb,\bfc}^{\nu'} Y_\bfb^{\pi|_{\ovb{r}}}}$.
\end{proof}

To set the stage for the final two lemmas, we first need to investigate a particular class of permutations in $\sigma\in\SG_\bfb^\rho$.

\begin{lemma}\label{moveMiddleWithoutChangingMSP} Let $\pi, \nu\in\Pi_{2(r)}$ such that $\pi|_{\ovb{r}}=\ov\nu|_{\ovb{r}}=\rho$ and fix $\bfa,\bfb,\bfc\in W_{r,k}$. Then \[\{\sigma\in \SG_\bfb^\rho:\kappa_{\bfa,\bfb,\bfc}(\pi\ast \sigma.\nu)=\kappa_{\bfa,\bfb,\bfc}(\pi\ast\nu)\}=A_{\bfa,\bfb}^\pi B_{\bfb,\bfc}^{\nu}Y_{\bfb}^\rho.\]
\end{lemma}

\begin{proof}
First, note that $\sigma$ factors as $\sigma=\sigma_X\sigma_Y$ for $\sigma_X\in X_\bfm^\rho$ and $\sigma_Y\in Y_\bfm^\rho$. Because $\sigma_Y.\nu=\nu$ for all $\nu$, we need only determine which $\sigma_X$ can be factored into a product of an element of $A_{\bfa,\bfb}^\pi$ and an element of $B_{\bfb,\bfc}^\nu$.

One containment is straightforward. Suppose $\sigma_X=\sigma_A\sigma_B$ with $\sigma_A\in A_{\bfa,\bfb}^{\pi}$ and $\sigma_B\in B_{\bfb,\bfc}^{\nu}$. Consider a block in $\pi\ast\nu$. Although the bottom half of this block may be different in $\pi\ast \sigma_B.\nu$, the condition that $\sigma_B\in B_{\bfb,\bfc}^{\nu}$ guarantees that it is not different in $\kappa_{\bfa,\bfb,\bfc}(\pi\ast \sigma_B.\nu)$. Hence, $\kappa_{\bfa,\bfb,\bfc}(\pi\ast \sigma_B.\nu)=\kappa_{\bfa,\bfb,\bfc}(\pi\ast\nu)$ for any $\pi\in\Pi_{2(r)}$ with $\pi|_{\ovb{r}}=\ov\nu|_{\ovb{r}}$. Analogously, we see that $\kappa_{\bfa,\bfb,\bfc}(\pi.{\sigma_A}^{-1}\ast \nu)=\kappa_{\bfa,\bfb,\bfc}(\pi\ast\nu)$ for any $\nu\in\Pi_{2(r)}$ with $\pi|_{\ovb{r}}=\ov\nu|_{\ovb{r}}$. Hence,
\begin{align*}
    \kappa_{\bfa,\bfb,\bfc}(\pi\ast \sigma_A\sigma_B.\nu)&=\kappa_{\bfa,\bfb,\bfc}(\pi.{\sigma_A}^{-1}\ast \sigma_B.\nu)\\
    &=\kappa_{\bfa,\bfb,\bfc}(\pi.{\sigma_A}^{-1}\ast \nu)\\
    &=\kappa_{\bfa,\bfb,\bfc}(\pi\ast\nu).
\end{align*}
For the other containment, suppose that $\kappa_{\bfa,\bfb,\bfc}(\pi\ast\nu)=\kappa_{\bfa,\bfb,\bfc}(\pi\ast \sigma_X.\nu)=\tmu$. We use the following convention for indexing the blocks of $\pi\ast\nu$. Write $\rho=\{M_1<\dots< M_\ell\}$ for the restrictions of the blocks of $\pi\ast\nu$ to $\ovb{r}$ in last-letter order, and write $S_i$ for the block of $\pi\ast\nu$ with $M_i\subseteq S_i$. An element $\sigma\in X_\bfm^\rho$ permutes the blocks of $\rho$ and hence the indices $[\ell]$. For $S_i\in\pi\ast\nu$, it will be helpful to write $\sigma(S_i)$ for the block in $\pi\ast \sigma.\nu$ such that $S_i|_{[r]\cup\ovb{r}}=\sigma(S_i)|_{[r]\cup\ovb{r}}$. Equivalently, $\sigma(S_i)$ is obtained by replacing the bottom row $S_i|_{\ovovb{r}}$ with $S_{\sigma^{-1}(i)}|_{\ovovb{r}}$.

For a fixed $\sigma\in X_\bfb^\rho$ and $\tS$ a multiset from $[k]\cup\ovb{k}\cup\ovovb{k}$, write \[W_{\sigma,\tS}=\{i:\kappa_{\bfa,\bfb,\bfc}(\sigma(S_i))=\tS\}.\]

For an example of these sets, see Example \ref{ex:final_lemmas}. The following two properties can be observed in this example---we show that they hold in general.

\begin{enumerate}
    \item[1.] For a fixed $\tR\in\kappa_{\bfa,\bfb}(\pi)$ such that $\tR\cap \ovb{k}\neq\emptyset$,\begin{align*}
    \biguplus_{\substack{\tS\in\tmu\\\tS|_{[k]\cup\ovb{k}}=\tR}}W_{1,\tS}=\biguplus_{\substack{\tS\in\tmu\\\tS|_{[k]\cup\ovb{k}}=\tR}}W_{\sigma_X,\tS}.
\end{align*}

This follows just about immediately from the fact that $\sigma(S_i)|_{[r]\cup\ovb{r}}=S_i|_{[r]\cup\ovb{r}}$ for any $\sigma\in \SG_\bfb^\rho$. Furthermore,

\begin{align*}
    \biguplus_{\substack{\tS\in\tmu\\\tS|_{[k]\cup\ovb{k}}=\tR}}W_{\sigma_X,\tS}&=\left\{i:\kappa_{\bfa,\bfb,\bfc}(\sigma_X(S_i))|_{[k]\cup\ovb{k}}=\tR\right\}\\
    &=\left\{i:\kappa_{\bfa,\bfb,\bfc}(S_i)|_{[k]\cup\ovb{k}}=\tR\right\}\\
    &=\biguplus_{\substack{\tS\in\tmu\\\tS|_{[k]\cup\ovb{k}}=\tR}}W_{1,\tS}.
\end{align*}

Note that the assumption that $\kappa_{\bfa,\bfb,\bfc}(\pi\ast\nu)=\kappa_{\bfa,\bfb,\bfc}(\pi\ast \sigma_X.\nu)=\tmu$ is necessary here so that the unions on either side of the equality are over the same set of $\tS$.

\item[2.] For a fixed $\tS\in\tmu$ with $\tS\cap\ovb{k}\neq\emptyset$, \begin{align*}
    |W_{\sigma_X,\tS}|=|W_{1,\tS}|.
\end{align*}

For $\sigma\in \SG_\bfb^\rho$ and $\tS\in\kappa_{\bfa,\bfb,\bfc}(\pi\ast\sigma.\nu)$ with $\tS\cap\ovb{k}\neq\emptyset$, we have
\begin{align*}
    \abs{W_{\sigma,\tS}}&=\abs{\left\{i:\kappa_{\bfa,\bfb,\bfc}(\sigma(S_i))=\tS\right\}}\\
    &=m_{\tS}(\kappa_{\bfa,\bfb,\bfc}(\pi\ast\sigma.\nu)).
\end{align*}

The statement then follows from the assumption that $\kappa_{\bfa,\bfb,\bfc}(\pi\ast\nu)=\kappa_{\bfa,\bfb,\bfc}(\pi\ast \sigma_X.\nu)=\tmu$.
\end{enumerate}

These two facts allow us to construct a permutation in the following way. Fixing $\tR\in\kappa_{\bfa,\bfb}(\pi)$ such that $\tR\cap \ovb{k}\neq\emptyset$, there exists a permutation $\eta_{\tR}$ of $\{i:\kappa_{\bfa,\bfb,\bfc}(S_i)|_{[k]\cup\ovb{k}}=\tR\}$ such that $\eta_\tR(W_{\sigma_X,\tS})=W_{1,\tS}$ for all $\tS$ with $\tS|_{[k]\cup\ovb{k}}=\tR$. Because $\eta_\tR$ by definition permutes only blocks that restrict to the same block in $\kappa_{\bfa,\bfb}(\pi)$, we see $\eta_\tR$ is an element of $A_{\bfa,\bfb}^\pi$. Now define $\eta\in A_{\bfa,\bfb}^\pi$ by \begin{align*}
    \eta:=\prod_{\substack{\tR\in\kappa_{\bfa,\bfb}(\pi)\\\tR\cap\ovb{k}\neq\emptyset}}\eta_\tR
\end{align*} with the product taken in any order.

It remains only to show that $\eta\sigma_X\in B_{\bfm,\bfb}^\nu$. Fix $i\in[\ell]$ and let $\tS=\kappa_{\bfa,\bfb,\bfc}(\eta\sigma_X(S_i))$. Then,

\begin{align*}
    \tS&=\kappa_{\bfa,\bfb,\bfc}(\eta\sigma_X(S_i))\\
    &=\kappa_{\bfa,\bfb}(S_i|_{[r]\cup\ovb{r}})\cup\kappa_\bfc(S_{(\eta\sigma_X)^{-1}(i)}).\\
    \intertext{Then by the fact that $\eta\in A_{\bfa,\bfb}^\pi$,}
    \tS&=\kappa_{\bfa,\bfb}(S_{\eta^{-1}(i)}|_{[r]\cup\ovb{r}})\cup\kappa_\bfc(S_{{\sigma_X}^{-1}(\eta^{-1}(i))})\\
    &=\kappa_{\bfa,\bfb,\bfc}(\sigma_X(S_{\eta^{-1}(i)})).
\end{align*}

Then $\eta^{-1}(i)\in W_{\sigma_X,\tS}$, so $i\in \eta(W_{\sigma_X,\tS})=W_{1,\tS}$. Hence, $\tS_i=\tS=\kappa_{\bfa,\bfb,\bfc}(\eta\sigma_X(S_i))$ for each $i$ and so $\eta\sigma_X\in B^{\pi}_{\bft,\bfm}$, meaning $\sigma_X\in A_{\bft,\bfm}^\pi B_{\bfm,\bfb}^\nu$. Thus $\sigma=\sigma_X\sigma_Y\in A_{\bft,\bfm}^\pi B_{\bfm,\bfb}^\nu Y_\bfm^\rho$.
\end{proof}

\begin{exam}\label{ex:final_lemmas} Here, we provide an example of the sets $W_{\sigma,\tS}$ defined in Lemma \ref{moveMiddleWithoutChangingMSP}. This definition relies on an (arbitrary) ordering of the blocks of $\pi\ast\nu$ which contain a barred element (i.e. touch the middle row in the diagram). In this example, we order these blocks by their restriction to the middle row, and we label them as-such in the following diagram.

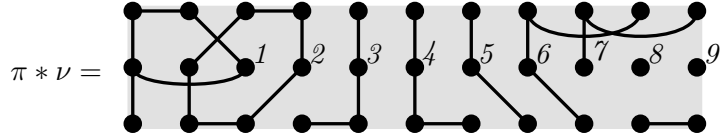
\begin{align*}
    \pi\ast\nu&=\begin{array}{c}
            \begin{tikzpicture}[xscale=.75,yscale=.75,line width=1.25pt] 
                \foreach \i in {1,2,3,4,5,6,7,8,9,10,11}  { \path (\i,1.25) coordinate (T\i); \path (\i,.25) coordinate (M\i); \path (\i,-.75) coordinate (B\i); } 
                \filldraw[fill= black!12,draw=black!12,line width=4pt]  (T1) -- (T11) -- (B11) -- (B1) -- (T1);
                \draw[black] (M1)--(T1)--(T2)--(M3);
                \draw[black] (M2)--(T3)--(T4)--(M4);
                \draw[black] (T5)--(M5);
                \draw[black] (T6)--(M6);
                \draw[black] (T7)--(M7);
                \draw[black] (M8)--(T8) .. controls +(0,-0.6) and +(0,-0.6) .. (T10);
                \draw[black] (M9)--(T9) .. controls +(0,-0.6) and +(0,-0.6) .. (T11);
                \draw[black] (B1)--(M1) .. controls +(0,-0.4) and +(0,-0.4) .. (M3);
                \draw[black] (M2)--(B2)--(B3)--(M4);
                \draw[black] (M5)--(B5)--(B4);
                \draw[black] (M6)--(B6)--(B7);
                \draw[black] (M7)--(B8);
                \draw[black] (M8)--(B9);
                \draw[black] (B10)--(B11);
                \colortop{1,2,3,4,5,6,7,8,9,10,11}{black};
                \colormid{1,2,3,4,5,6,7,8,9,10,11}{black};
                \colorbot{1,2,3,4,5,6,7,8,9,10,11}{black};
                \node at (3.25,.5) {1};
                \node at (4.25,.5) {2};
                \node at (5.25,.5) {3};
                \node at (6.25,.5) {4};
                \node at (7.25,.5) {5};
                \node at (8.25,.5) {6};
                \node at (9.25,.5) {7};
                \node at (10.25,.5) {8};
                \node at (11.25,.5) {9};
            \end{tikzpicture}
        \end{array}
\end{align*}

Letting $\sigma_X=(1\,2)(3\,5\,6)$, we similarly label the diagram for $\pi\ast\sigma_X.\nu$.

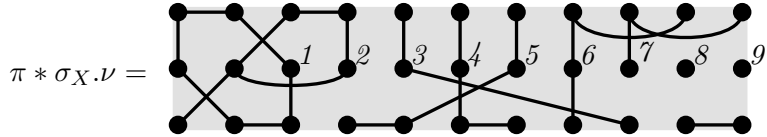
\begin{align*}
    \pi\ast\sigma_X.\nu&=\begin{array}{c}
            \begin{tikzpicture}[xscale=.75,yscale=.75,line width=1.25pt] 
                \foreach \i in {1,2,3,4,5,6,7,8,9,10,11}  { \path (\i,1.25) coordinate (T\i); \path (\i,.25) coordinate (M\i); \path (\i,-.75) coordinate (B\i); } 
                \filldraw[fill= black!12,draw=black!12,line width=4pt]  (T1) -- (T11) -- (B11) -- (B1) -- (T1);
                \draw[black] (M1)--(T1)--(T2)--(M3);
                \draw[black] (M2)--(T3)--(T4)--(M4);
                \draw[black] (T5)--(M5);
                \draw[black] (T6)--(M6);
                \draw[black] (T7)--(M7);
                \draw[black] (M8)--(T8) .. controls +(0,-0.6) and +(0,-0.6) .. (T10);
                \draw[black] (M9)--(T9) .. controls +(0,-0.6) and +(0,-0.6) .. (T11);
                \draw[black] (B1)--(M2) .. controls +(0,-0.4) and +(0,-0.4) .. (M4);
                \draw[black] (M1)--(B2)--(B3)--(M3);
                \draw[black] (M7)--(B5)--(B4);
                \draw[black] (M6)--(B6)--(B7);
                \draw[black] (M8)--(B8);
                \draw[black] (M5)--(B9);
                \draw[black] (B10)--(B11);
                \colortop{1,2,3,4,5,6,7,8,9,10,11}{black};
                \colormid{1,2,3,4,5,6,7,8,9,10,11}{black};
                \colorbot{1,2,3,4,5,6,7,8,9,10,11}{black};
                \node at (3.25,.5) {1};
                \node at (4.25,.5) {2};
                \node at (5.25,.5) {3};
                \node at (6.25,.5) {4};
                \node at (7.25,.5) {5};
                \node at (8.25,.5) {6};
                \node at (9.25,.5) {7};
                \node at (10.25,.5) {8};
                \node at (11.25,.5) {9};
            \end{tikzpicture}
        \end{array}
\end{align*}

We can now read off the blocks $S_2=\{3,4,\ov2,\ov4,\ovov2,\ovov3\}$ and $\sigma_X(S_2)=\{3,4,\ov2,\ov4,\ovov1\}$ as the blocks labeled two in the first and second diagram respectively. Now we apply the above labels to the diagrams of $\kappa_{\bfa,\bfb,\bfc}(\pi\ast\nu)$ and $\kappa_{\bfa,\bfb,\bfc}(\pi\ast \sigma_X.\nu)$ where $\bfa=(7,2,2)$, $\bfb=(2,7,2)$, and $\bfc=(2,5,4)$.

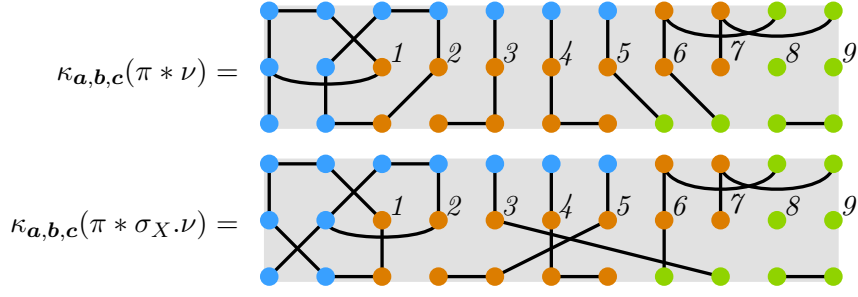
\begin{align*}
    \kappa_{\bfa,\bfb,\bfc}(\pi\ast\nu)&=\begin{array}{c}
            \begin{tikzpicture}[xscale=.75,yscale=.75,line width=1.25pt] 
                \foreach \i in {1,2,3,4,5,6,7,8,9,10,11}  { \path (\i,1.25) coordinate (T\i); \path (\i,.25) coordinate (M\i); \path (\i,-.75) coordinate (B\i); } 
                \filldraw[fill= black!12,draw=black!12,line width=4pt]  (T1) -- (T11) -- (B11) -- (B1) -- (T1);
                \draw[black] (M1)--(T1)--(T2)--(M3);
                \draw[black] (M2)--(T3)--(T4)--(M4);
                \draw[black] (T5)--(M5);
                \draw[black] (T6)--(M6);
                \draw[black] (T7)--(M7);
                \draw[black] (M8)--(T8) .. controls +(0,-0.6) and +(0,-0.6) .. (T10);
                \draw[black] (M9)--(T9) .. controls +(0,-0.6) and +(0,-0.6) .. (T11);
                \draw[black] (B1)--(M1) .. controls +(0,-0.4) and +(0,-0.4) .. (M3);
                \draw[black] (M2)--(B2)--(B3)--(M4);
                \draw[black] (M5)--(B5)--(B4);
                \draw[black] (M6)--(B6)--(B7);
                \draw[black] (M7)--(B8);
                \draw[black] (M8)--(B9);
                \draw[black] (B10)--(B11);
                \colortop{1,2,3,4,5,6,7}{c1};
                \colortop{8,9}{c2};
                \colortop{10,11}{c3};
                \colormid{1,2}{c1};
                \colormid{3,4,5,6,7,8,9}{c2};
                \colormid{10,11}{c3};
                \colorbot{1,2}{c1};
                \colorbot{3,4,5,6,7}{c2};
                \colorbot{8,9,10,11}{c3};
                \node at (3.25,.5) {1};
                \node at (4.25,.5) {2};
                \node at (5.25,.5) {3};
                \node at (6.25,.5) {4};
                \node at (7.25,.5) {5};
                \node at (8.25,.5) {6};
                \node at (9.25,.5) {7};
                \node at (10.25,.5) {8};
                \node at (11.25,.5) {9};
            \end{tikzpicture}
        \end{array}\\
    \kappa_{\bfa,\bfb,\bfc}(\pi\ast \sigma_X.\nu)&= \begin{array}{c}
            \begin{tikzpicture}[xscale=.75,yscale=.75,line width=1.25pt] 
                \foreach \i in {1,2,3,4,5,6,7,8,9,10,11}  { \path (\i,1.25) coordinate (T\i); \path (\i,.25) coordinate (M\i); \path (\i,-.75) coordinate (B\i); } 
                \filldraw[fill= black!12,draw=black!12,line width=4pt]  (T1) -- (T11) -- (B11) -- (B1) -- (T1);
                \draw[black] (M1)--(T1)--(T2)--(M3);
                \draw[black] (M2)--(T3)--(T4)--(M4);
                \draw[black] (T5)--(M5);
                \draw[black] (T6)--(M6);
                \draw[black] (T7)--(M7);
                \draw[black] (M8)--(T8) .. controls +(0,-0.6) and +(0,-0.6) .. (T10);
                \draw[black] (M9)--(T9) .. controls +(0,-0.6) and +(0,-0.6) .. (T11);
                \draw[black] (B1)--(M2) .. controls +(0,-0.4) and +(0,-0.4) .. (M4);
                \draw[black] (M1)--(B2)--(B3)--(M3);
                \draw[black] (M7)--(B5)--(B4);
                \draw[black] (M6)--(B6)--(B7);
                \draw[black] (M8)--(B8);
                \draw[black] (M5)--(B9);
                \draw[black] (B10)--(B11);
                \colortop{1,2,3,4,5,6,7}{c1};
                \colortop{8,9}{c2};
                \colortop{10,11}{c3};
                \colormid{1,2}{c1};
                \colormid{3,4,5,6,7,8,9}{c2};
                \colormid{10,11}{c3};
                \colorbot{1,2}{c1};
                \colorbot{3,4,5,6,7}{c2};
                \colorbot{8,9,10,11}{c3};
                \node at (3.25,.5) {1};
                \node at (4.25,.5) {2};
                \node at (5.25,.5) {3};
                \node at (6.25,.5) {4};
                \node at (7.25,.5) {5};
                \node at (8.25,.5) {6};
                \node at (9.25,.5) {7};
                \node at (10.25,.5) {8};
                \node at (11.25,.5) {9};
            \end{tikzpicture}
        \end{array}
\end{align*}

Now we can read off that $\kappa_{\bfa,\bfb,\bfc}(\sigma_X(S_2))=\multi{1,1,\ov1,\ov2,\ovov1}$ by looking at the block labeled $2$ in the above diagram.

Consider $\tR=\multi{1,\ov2}\in\kappa_{\bfa,\bfb}(\pi)$. There are two distinct blocks in $ \kappa_{\bfa,\bfb,\bfc}(\pi\ast\nu)$ that restrict to this $\tR$: $\tS=\multi{1,\ov2,\ovov2,\ovov2}$ and $\tS'=\multi{1,\ov2,\ovov3}$. The first corresponds to the blocks labeled $3$ and $4$, while the second corresponds to just the block labeled $5$. From this, we can conclude the following:

\begin{align*}
    W_{1,\tS}=\{3,4\} & & W_{1,\tS'}=\{5\}\\
\intertext{Applying the same logic to $\kappa_{\bfa,\bfb,\bfc}(\pi\ast\nu)$ yields the following:}
    W_{\sigma_X,\tS}=\{4,5\} & & W_{\sigma_X,\tS'}=\{3\}
\end{align*}

Notice that the sets in each column have the same size and the union across rows is always $\{3,4,5\}$.
\end{exam}

\begin{lemma}\label{lem:number_of_sigma} Fix $\pi,\nu\in\Pi_{2(r)}$ such that $\pi|_{\ovb{r}}=\ov\nu_{\ovb{r}}=\rho$ and $\bfa,\bfb,\bfc\in W_{r,k}$. Let $\tmu\in\tilde{\Gamma}_{\kappa_{\bfb,\bfc}(\nu)}^{\kappa_{\bfa,\bfb}(\pi)}$. The set of $\sigma\in \SG_\bfm$ for which there exists a $\gamma\in\Gamma_{\sigma.\nu}^\pi(\tmu)$ is given by
\begin{align*}
    A_{\bfa,\bfb}^\pi B_{\bfb,\bfc}^{\sigma_0.\nu} Y_\bfb^\rho\sigma_0
\end{align*}
for some $\sigma_0\in \SG_\bfb$.
\end{lemma}

\begin{proof}

Let $\pi,\nu\in\Pi_{2(r)}$. If there exists $\gamma\in\Gamma^\pi_\nu(\tmu)$, then \[\Break(\tmu)=\kappa_{\bfa,\bfb,\bfc}(\pi\ast\nu).\] Conversely if $\Break(\tmu)=\kappa_{\bfa,\bfb,\bfc}(\pi\ast\nu)$ we can construct a $\gamma\in\Gamma_\nu^\pi(\tmu)$ as follows. For each block of $\tmu$ broken into $\tilde T$ in the top and $\tilde B$ in the bottom, find blocks $T$ and $B$ in $\pi$ and $\nu$ for which $\kappa_{\bfa,\bfb}(T)=\tilde T$ and $\kappa_{\bfb,\bfc}(B)=\tilde B$ and connect these blocks in $\pi\ast\nu$. After connecting such a pair for each block broken in $\mu$, we have constructed the desired $\gamma$ (see Example \ref{ex:form_gamma}).

Hence, we are looking for the set of $\sigma\in\SG_\bfb$ for which $\kappa_{\bfa,\bfb,\bfc}(\pi\ast\sigma.\nu)=\Break(\tmu)$. Note that $\pi\ast\sigma.\nu$ only makes sense when $\sigma.\nu|_{[r]}=\rho$, so we need only consider $\sigma\in \SG_\bfb^\rho$.

Choose $\sigma_0$ so that $\Break(\tmu)=\kappa_{\bfa,\bfb,\bfc}(\pi\ast\sigma_0.\nu)$ and write $\nu'=\sigma_0.\nu$. Then the desired set of $\sigma$ is precisely the permutations $\sigma\in \SG_\bfm^\rho$ such that \begin{align*}
     \kappa_{\bfa,\bfb,\bfc}(\pi\ast(\sigma\sigma_0^{-1}).\nu')=\kappa_{\bfa,\bfb,\bfc}(\pi\ast\nu').
\end{align*}

Lemma \ref{moveMiddleWithoutChangingMSP} tells us that this set is precisely those $\sigma$ where \begin{align*}
    \sigma{\sigma_0}^{-1}\in A_{\bft,\bfm}^\pi B_{\bfm,\bfb}^{\nu'}Y_\bfm^\rho
\end{align*} as desired.

\end{proof}

\begin{exam}\label{ex:form_gamma} Here we have an example of a $\tmu, \pi,$ and $\nu$ such that $\Break(\tmu)=\kappa_{\bfa,\bfb,\bfc}(\pi\ast\nu)$. In the diagram of $\tmu$, we represent with dotted lines the connections that must be severed to form $\Break(\tmu)$. To form $\gamma\in\Gamma_\nu^\pi(\tmu)$ we add the corresponding connections to $\pi\ast\nu$, again represented by dotted lines in the diagram of $\gamma$.
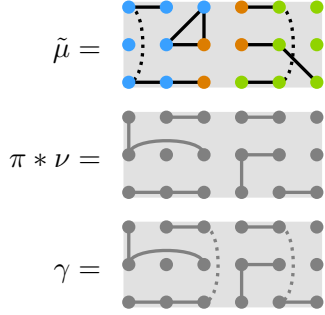
    \begin{align*}
    \tmu&=\begin{array}{c}
            \begin{tikzpicture}[xscale=.5,yscale=.5,line width=1.25pt] 
                \foreach \i in {1,2,3,4,5,6}  { \path (\i,1.25) coordinate (T\i); \path (\i,.25) coordinate (M\i); \path (\i,-.75) coordinate (B\i); } 
                \filldraw[fill= black!12,draw=black!12,line width=4pt]  (T1) -- (T6) -- (B6) -- (B1) -- (T1);
                \draw[black, dotted] (T1) .. controls +(+0.50,0) and +(+0.50,0) .. (B1);
                \draw[black] (T1) -- (T2);
                \draw[black] (B1) -- (B2) -- (B3);
                \draw[black] (T3) -- (M3) -- (M2) -- (T3);
                \draw[black] (T4) -- (T5);
                \draw[black, dotted] (T5) .. controls +(+0.50,0) and +(+0.50,0) .. (B5);
                \draw[black] (B4) -- (B5);
                \draw[black] (M4) -- (M5) -- (B6);
                \colortop{1,2,3}{c1};
                \colortop{4}{c2};
                \colortop{5,6}{c3};
                \colormid{1,2}{c1};
                \colormid{3,4}{c2};
                \colormid{5,6}{c3};
                \colorbot{1,2}{c1};
                \colorbot{3}{c2};
                \colorbot{4,5,6}{c3};
            \end{tikzpicture}
        \end{array}\\
        \pi\ast\nu&=\begin{array}{c}
            \begin{tikzpicture}[xscale=.5,yscale=.5,line width=1.25pt] 
                \foreach \i in {1,2,3,4,5,6}  { \path (\i,1.25) coordinate (T\i); \path (\i,.25) coordinate (M\i); \path (\i,-.75) coordinate (B\i); } 
                \filldraw[fill= black!12,draw=black!12,line width=4pt]  (T1) -- (T6) -- (B6) -- (B1) -- (T1);
                \draw[gray] (T2)--(T3);
                \draw[gray] (B1) -- (B3);
                \draw[gray] (T1) -- (M1) .. controls +(0,+0.50) and +(0,+0.50) .. (M3);
                \draw[gray] (T4) -- (T5);
                \draw[gray] (B5) -- (B6);
                \draw[gray] (B4) -- (M4) -- (M5);
                \colortop{1,2,3,4,5,6}{gray};
                \colormid{1,2,3,4,5,6}{gray};
                \colorbot{1,2,3,4,5,6}{gray};
            \end{tikzpicture}
        \end{array}\\
        \gamma&=\begin{array}{c}
            \begin{tikzpicture}[xscale=.5,yscale=.5,line width=1.25pt] 
                \foreach \i in {1,2,3,4,5,6}  { \path (\i,1.25) coordinate (T\i); \path (\i,.25) coordinate (M\i); \path (\i,-.75) coordinate (B\i); } 
                \filldraw[fill= black!12,draw=black!12,line width=4pt]  (T1) -- (T6) -- (B6) -- (B1) -- (T1);
                \draw[gray] (T2)--(T3);
                \draw[gray, dotted] (T3) .. controls +(+0.50,0) and +(+0.50,0) .. (B3);
                \draw[gray] (B1) -- (B3);
                \draw[gray] (T1) -- (M1) .. controls +(0,+0.50) and +(0,+0.50) .. (M3);
                \draw[gray] (T4) -- (T5);
                \draw[gray, dotted] (T5) .. controls +(+0.50,0) and +(+0.50,0) .. (B5);
                \draw[gray] (B5) -- (B6);
                \draw[gray] (B4) -- (M4) -- (M5);
                \colortop{1,2,3,4,5,6}{gray};
                \colormid{1,2,3,4,5,6}{gray};
                \colorbot{1,2,3,4,5,6}{gray};
            \end{tikzpicture}
        \end{array}
\end{align*}
\end{exam}

\begin{lemma}\label{lem:how_many_set_partitions} Let $\mu$ be a set partition of $[m]$ and $\bfa\in W_{m,\ell}$ such that the blocks of $\kappa_\bfa(\mu)=\tmu$ are all sets. Then the number of set partitions $\gamma$ of $[m]$ such that $\kappa_\bfa(\gamma)=\tmu$ is \begin{align*}
    \frac{\bfa_1!\dots\bfa_\ell!}{m(\tmu)!}.
\end{align*}
\end{lemma}

\begin{proof}
    First, observe that \begin{align*}
        \{\gamma:\kappa_\bfa(\gamma)=\tmu\}=\SG_\bfa.\mu.
    \end{align*} By the orbit-stabilizer theorem (see \cite[Proposition 6.8.4]{artin1991algebra}), the size of this orbit is the same as the number of cosets of the corresponding stabilizer $\SG_\bfa^\mu$. Because the blocks of $\tmu$ are sets, no permutation of $\SG_\bfa$ swaps elements within a block of $\mu$. Hence, the permutations that fix $\mu$ are precisely the ones that swap whole blocks, so $\abs{\SG_\bfa^\mu}=m(\tmu)!$. The number of cosets is then obtained using Lagrange's theorem as the quotient of $\abs{\SG_\bfa}$ by the size of this stabilizer. 
\end{proof}

\begin{lemma}\label{lem:size_of_gamma} Fix $\pi,\nu\in \Pi_{2(r)}$ and suppose $\tmu\in\tilde{\Gamma}^{\kappa_{\bfa,\bfb}(\tpi)}_{\kappa_{\bfb,\bfc}(\tnu)}$. If $\Gamma_\nu^\pi(\tmu)\neq\emptyset$, then \[\abs{\Gamma^\pi_\nu(\tmu)}=\frac{m(\tpi_+)!m(\tnu_-)!}{m(\tmu_\pm)!}.\]
\end{lemma}

\begin{proof} Let $\gamma\in\Gamma_\nu^\pi(\tmu)$. Because $\gamma$ differs from $\pi\ast\nu$ by connecting some number of blocks in the very top and very bottom, we can recover $\gamma$ uniquely from the partial matching of blocks of $(\pi\ast\nu)_\pm$ induced by $\gamma_\pm$. The question then becomes how many set partitions $\rho$ on the blocks of $(\pi\ast\nu)_\pm$ there are such that $\kappa_{\bfa,\bfc}(\rho)=\tmu_\pm$. Because we are only connecting blocks on top to blocks on bottom, we can apply Lemma \ref{lem:how_many_set_partitions} where the $\ell$ colors are the different multisets that appear in $\tmu_\pm$. The number is then \begin{align*}
    \frac{m(\tmu_\pm|_{[k]})!m(\tmu_\pm|_{\ovovb{k}})!}{m(\tmu_\pm)!}&=\frac{m(\tpi_+)!m(\nu_-)!}{m(\tmu_\pm)!}
\end{align*} where $\tpi=\kappa_{\bfa,\bfb}(\pi)$ and $\tnu=\kappa_{\bfb,\bfc}(\nu)$.
\end{proof}

\bibliography{bibliography}{}
\bibliographystyle{alpha}

\end{document}